%% file: arxiv.tex
\documentclass[11pt]{article}

\usepackage[margin=1in]{geometry}

\usepackage{amsmath,amssymb,amsthm}
\usepackage{mathtools}

\usepackage{graphicx}
\usepackage{subcaption}
\usepackage{array}
\usepackage{tikz}
\usetikzlibrary{automata,positioning}

\usepackage{xcolor}

\usepackage[numbers]{natbib}
\usepackage[colorlinks=true,
            citecolor=blue,
            linkcolor=blue,
            urlcolor=blue]{hyperref}
\usepackage[nameinlink,capitalise,noabbrev]{cleveref}

\newtheorem{theorem}{Theorem}
\newtheorem{lemma}{Lemma}
\newtheorem{proposition}{Proposition}

\newtheorem{assumption}{Assumption}

\theoremstyle{remark}

\usepackage{pgfplots}
\pgfplotsset{compat=1.18}

\newcommand\numberthis{\addtocounter{equation}{1}\tag{\theequation}}

\newcommand{\opt}{x^\star}

\allowdisplaybreaks 
\usepackage{wrapfig}

\title{Efficiency--Reward Trade-Off in Queues with Dynamic Arrivals}

\author{
Tianze Qu\thanks{Department of Industrial and Operations Engineering, University of Michigan.}
\and
Sushil Mahavir Varma\thanks{Department of Industrial and Operations Engineering, University of Michigan.}
}

\date{\today}

\begin{document}

\maketitle

\begin{abstract}
Motivated by applications in online marketplaces such as ride-hailing platforms and payment channel networks, we study a single server queue with service rate $\mu$ and queue-length-dependent arrival rate $\lambda(q) \in [0, \lambda_{\max}]$. The system operator receives a reward $F(\lambda)$ by setting an arrival rate $\lambda$, where $F$ can represent objectives such as throughput, revenue, or social welfare. The goal is to characterize the Pareto frontier between long-run average reward and expected queue length ($\mathbb{E}[q]$). We first establish a fluid upper bound on the reward achievable under any control policy $\{\lambda(q)\}$, and define regret as the gap from this upper bound. We then minimize $\mathbb{E}[q]$ subject to the regret being at most $\varepsilon$.
In the small-market regime $(\lambda_{\max} \leq \mu)$,
we show that no control policy can improve over the classical heavy-traffic scaling of $\mathbb{E}[q] = \Omega(1/\varepsilon)$. In contrast, in the large-market regime $(\lambda_{\max} > \mu)$, we substantially improve over the classical scaling by explicitly constructing control policies with expected queue lengths at most $O(1/\sqrt{\varepsilon})$ or $O(\log (1/\varepsilon))$, depending on the curvature structure of \(F\). Moreover, we establish universal lower bounds of $\Omega(1/\sqrt{\varepsilon})$ or $\Omega(\log (1/\varepsilon))$ respectively under any general control policy. These results can be viewed as (non asymptotic) heavy-traffic theory for queues with dynamic arrivals.
\end{abstract}


\section{Introduction}
Over the past century, queueing theory has emerged as a well-established discipline with a widespread impact on several real-life applications \cite{erlang1909theory} such as telecommunication systems, manufacturing systems, cloud computing and transportation systems. Fundamental queueing models provide a principled way of  understanding several naturally arising design and optimization trade-offs in such service systems. For example, Kingman's formula \cite{kingman1961single} formalizes the trade-off between delay and utilization by studying a single server queue. The Halfin-Whitt  regime \cite{halfin1981heavy} optimizes the trade-off between delay and staffing costs in a multi-server queue. More examples include trade-off between infrastructure cost and throughput for systems with finite waiting space, and trade-off between flexibility cost and delay for systems with heterogeneous customer and server types \cite{shi2019process}. Such analyses are often carried out under the setting of queue-length-independent exogenous arrivals. On the other hand, queueing systems with state-dependent arrivals are relatively less studied, and similar trade-offs arising in design of such systems are not fully understood.

The widespread deployment of digital service platforms has enabled system operators to monitor system states in real time and dynamically influence customer arrivals through pricing, information disclosure, or access control. As a result, many modern service systems operate in dynamic environments where arrivals are no longer exogenous. Examples include dynamic pricing in ride-hailing platforms \cite{varma2023dynamic}, congestion-based pricing for electric vehicle charging stations \cite{anjos2025optimal}, information design in hospital waiting rooms \cite{zang2024impact}, and admission control in telecommunication and cloud computing systems \cite{maguluri2012stochastic}.  
Across such systems, operators seek to optimize various performance objectives, such as throughput, profit, or social welfare, while maintaining operational efficiency, typically measured through customer waiting times or queue lengths. Understanding how to jointly optimize system performance and efficiency is therefore a central challenge.

Motivated by these applications, we develop a theoretical framework for queues with dynamic arrivals, focusing on a single server queue with queue-length-dependent arrival rate $\lambda(q)$ and service rate $\mu=1$. In particular, for a given queue length $q \in \mathbb{Z}_+$, the system operator is free to set $\lambda(q) \in [0, \lambda_{\max}]$ and receives a reward $F(\lambda(q))$. The objective is to maximize the long-run average reward $\mathbb{E}\left[F(\lambda(\bar{q}))\right]$, where $\bar{q}$ is the steady-state queue length. At the same time, the operator must control system congestion, measured by the expected steady-state queue length $(\mathbb{E}[\bar{q}])$. The goal of this paper is to characterize the Pareto frontier between these two key objectives for a general function $F$ and any state-dependent control policy $\{\lambda(q)\}$. In particular, let $F^\star$ be the maximum possible reward under stable control policies and define the regret to be $R(\lambda) = F^\star - \mathbb{E}\left[F(\lambda(\bar{q}))\right]$. Then, our objective is to minimize $\mathbb{E}[\bar{q}]$ subject to the regret being at most $\varepsilon$.

To put the problem formulation in context, consider the special case of a linear reward function $F(x) = x$. In this setting, we are minimizing the expected queue length while ensuring small regret $(R(\lambda) \leq \varepsilon)$, which corresponds to the throughput $(\mathbb{E}[\lambda(\bar{q})])$ being at most $\varepsilon$ less than the maximum possible value $F^\star = \mu$. 
Observe that this is a natural generalization of the heavy-traffic regime for queues with constant arrival rate \cite{kingman1961single}, where the arrival rate approaches the capacity $\mu$. Thus, our problem formulation naturally extends the definition of the heavy-traffic regime to queues with dynamic arrivals.

The power of our problem formulation lies in its generality, i.e., we allow for general reward functions $F$. For example, we illustrate how dynamic pricing for revenue maximization in an $M/M/1$ queue can be viewed as a special case of our model. Suppose that a customer arrives to the system and observes a price of $p$. The customer then compares this price $p$ with its own valuation $v$ and joins the system if and only if $v > p$. A popular choice to model heterogeneous customers is to sample $v$ from a known distribution with CDF $G$. In this case, the incoming customer's joining probability is $\mathbb{P}\{v > p\} = 1 - G(p)$, resulting in an expected revenue of $p\mathbb{P}\{v > p\} = (1 - G(p))p$ per potential customer. One can write the expression of this revenue in terms of the effective arrival rate by a simple change of variables. In particular, if the exogenous arrival rate of customers is $\lambda_{\max}$, then the effective arrival rate is $\lambda = \lambda_{\max} (1-G(p))$. Thus, the revenue is equivalently given by $\lambda_{\max}p\mathbb{P}\{v > p\}=\lambda G^{-1}(1-\lambda / \lambda_{\max})$. So, setting $F(\lambda) = \lambda G^{-1}(1-\lambda / \lambda_{\max})$ in our model is exactly the problem of dynamic pricing for revenue maximization with price-sensitive customers. It turns out that our framework can also be used to study the setting of price and delay sensitive customers \cite{kim2018value} and we refer the reader to Section~\ref{value of DP in large queue system} for more details.

\subsection{Overview of Main Results and Contributions}

When the market size is small, i.e., $\lambda_{\max} \leq \mu =1$, we show that any admissible control policy must incur \(\mathbb{E}[\bar q] =  \Omega(1/\varepsilon)\), which coincides with the classical heavy-traffic scaling.
Consequently, dynamic arrival control does not yield efficiency gains in small markets. In contrast, when the market is large \((\lambda_{\max}>1)\), dynamic arrival control enables a much richer set of efficiency-reward trade-offs. In this regime, a key insight of our analysis is that the nature of this trade-off depends critically on whether \(F\) exhibits strict Jensen-type behavior at the capacity threshold \(\mu=1\). Specifically, we distinguish between two regimes depending on whether \(F\) satisfies
\begin{equation}
\begin{aligned}
F(1)>pF(x_1)+(1-p)F(x_2) \label{eq:concave_like}
\end{aligned}
\end{equation}
for all \(x_1,x_2\in[0,\lambda_{\max}]\setminus\{1\},
\,\,p\in(0,1)\) with \(px_1+(1-p)x_2=1\).
This condition captures a strictly concave-like behavior of \(F\) around the capacity. The table below summarizes the resulting asymptotic behavior of the optimal expected steady-state queue length under different classes of arrival control policies, depending on whether \(F\) satisfies \eqref{eq:concave_like}.

\begin{table}[h]
    \centering
    \begin{tabular}{|m{4.7cm}<{\centering}|m{3.7cm}<{\centering}|m{3.7cm}<{\centering}|}
         \hline
        Optimal $\mathbb{E}[\bar q]$ & $F$ is not concave-like (does not satisfy \eqref{eq:concave_like}) & $F$ is concave-like (satisfies \eqref{eq:concave_like}) \\
         \hline
           Static arrival rate & $\Theta\left(1/\varepsilon\right)$ &  $\Theta\left(1/\varepsilon\right)$\\
         \hline
         Two arrival rates & $\Theta\left(\log\frac{1}{\varepsilon}\right)$ & $\Theta\left(\frac{1}{\sqrt{\varepsilon}}\sqrt{\log\frac{1}{\varepsilon}}\right)$\\
         \hline
         Fully dynamic arrival rates & $\Theta\left(\log\frac{1}{\varepsilon}\right)$ & $\Theta\left(\frac{1}{\sqrt{\varepsilon}}\right)$ \\
         \hline \hline
         Universal lower bound  & $\Omega\left(\log\frac{1}{\varepsilon}\right)$ & $\Omega\left(\frac{1}{\sqrt{\varepsilon}}\right)$ \\
         \hline
    \end{tabular}
    \caption{Summary of the behavior under different arrival settings when $\lambda_{\max}> 1$.}
\end{table} 

When the reward function \(F\) is concave-like, our main result shows that the optimal expected queue length satisfies \(q^\star = \Theta(1/\sqrt{\varepsilon})\). In particular, we establish a universal lower bound showing that the expected queue length must be at least \(\Omega(1/\sqrt{\varepsilon})\) under any \emph{arbitrary} admissible arrival policy. We then construct a fully dynamic, state-dependent arrival policy that matches this lower bound, thereby attaining the fundamental efficiency limit. By contrast, we show that restricting attention to simpler policy classes leads to strictly worse performance. Any two-arrival-rate policy necessarily incurs an expected queue length of at least \(\Omega(\sqrt{\log(1/\varepsilon)}/\sqrt{\varepsilon})\), exhibiting an unavoidable logarithmic inefficiency, while static arrival policies suffer an even larger delay of order \(\Omega(1/\varepsilon)\).

Our main methodological contribution is to go beyond simple two-arrival policies and understand the performance of fully dynamic policies. In particular, to establish the upper bound of $q^\star = O(1/\sqrt{\varepsilon})$, we provide an explicit analytical characterization of fully dynamic arrival control policy and demonstrate that such a policy is necessary to attain the order-optimal efficiency–reward trade-off for concave-like reward functions $F$. Moreover, another methodological contribution is to establish the lower bound $q^\star = \Omega(1/\sqrt{\varepsilon})$ while allowing for general arrival control policies and a general concave-like function $F$. 
We identify that using the properties of birth-death chain, one can view the regret $R(\lambda) = F^\star - \mathbb{E}_\pi [F(\lambda(\bar q))]$ as an $f-$divergence between the stationary distribution and its shifted and appropriately scaled version for an appropriately defined $f$. Then, in a nutshell, our proof  lower bounds the $f-$divergence in terms of the TV-distance, which results in a pointwise upper bound on the stationary distribution, completing the proof.

Next, when \(F\) is not concave-like, we show that the system exhibits a fundamentally different trade-off. In this case, the optimal expected queue length satisfies $q^\star = \Theta(\log 1/\varepsilon)$. A carefully designed two-arrival policy is sufficient to achieve an expected steady-state queue length of \(O(\log 1/\varepsilon)\), which matches the universal lower bound of $\Omega(\log 1/\varepsilon)$ under any \textit{arbitrary} admissible policy. Meanwhile, the static arrival policy continues to incur an expected queue length of at least \(\Omega(1/\varepsilon)\). Note that the optimal queue length scaling is much smaller than in the concave-like case as one can exploit the non-concavity of $F$ at the capacity $(\mu = 1)$ to design a better performing policy. In particular, violation of \eqref{eq:concave_like} implies the existence of $x_1 > 1 > x_2 > 0$, which when mixed $(p F(x_1) + (1-p) F(x_2))$, results in a better reward compared to deterministically operating at the capacity $(F(1))$. One can therefore operate close to this mixture while ensuring low queue lengths by setting the arrival rate of $x_2 < 1$ when the system is congested and the arrival rate of $x_1 > 1$ when it is less busy.

Taken together, the concave-like and non-concave-like regimes yield a complete characterization of efficiency–reward trade-offs in single-server queues with endogenous arrivals. Our results precisely identify when fully dynamic arrival control is necessary to achieve order-optimal performance and when simpler policies already suffice, and show how the curvature of the reward function at capacity shapes the optimal control structure. This unified framework strengthens and generalizes existing results in revenue management and dynamic pricing, moving beyond restricted policy classes in \cite{kim2018value} to fully dynamic control.



\section{Literature Review}
{\bf Dynamic Arrival/Service Rate Control.}
State-dependent arrival/service control has been studied widely in the queueing literature. For example, \cite{george2001dynamic}
is one of the earliest works studying the dynamic control of a queue. They show that, in order to minimize the holding and service costs, the optimal service rate policy exhibits a non-decreasing structure in the queue length. Afterwards, \cite{ata2006dynamic, adusumilli2010dynamic} extend this framework to joint arrival and service rate control. There is also a stream of literature that investigates the structure of the optimal policy under more general arrival processes beyond the Poisson assumption, such as \cite{kumar2013dynamic, xia2017optimal}. Moreover, \cite{arapostathis2019optimal} derive asymptotically optimal control policies for a Markov-modulated multiclass many-server queue. Dynamic control also finds applications in a wide range of service and networked systems, such as load balancing in server farms \cite{tsitsiklis2013power}, throughput optimization in payment channel networks \cite{varma2021throughput} and the design of energy-aware systems \cite{down2022optimal}. Out of these papers, the most relevant to our work is \cite{tsitsiklis2013power}, which studies throughput maximization in large-scale server systems and shows that simple threshold-type policies are optimal in heavy traffic, achieving a queue lengths of order $\log(1/\varepsilon)$ as opposed to the standard heavy traffic scaling of $1/\varepsilon$ when even limited resource pooling is available. Our results for the non-concave-like regime (in which throughput is a special case) also exhibit a similar exponential improvement in expected queue length. The underlying reason in both settings is the capability of overloading the queue when its queue length is small, which is done via resource pooling in \cite{tsitsiklis2013power} and via dynamic arrivals in our setting.


{\bf Dynamic Pricing for Queues.} Pricing serves as a natural mechanism for dynamic arrival rate control, as customer demand typically responds to price changes. A substantial body of literature studies dynamic pricing in queueing systems by formulating a Markov decision process. A long line of work proves several structural properties of optimal pricing policies for various increasingly general models such as queues with finite buffer \cite{low1974a}, queues with infinite buffer \cite{low1974b}, customers sensitive to price and waiting time \cite{chen2001state}, time-varying system parameters \cite{yoon2004optimal}, multi-class queues \cite{kim2003optimal, maoui2007congestion, ccil2011dynamic}.


The literature most related to us is the asymptotic analysis of pricing policies in queueing systems such as \cite{kim2018value} for a single server queue, \cite{balseiro2025dynamic} for a loss system, and \cite{varma2023dynamic} for two-sided queues. We next compare our results with each of these papers.

\citet{kim2018value} shows that a simple two-price scheme is near-optimal with an optimality rate of $O((n \log n)^{1/3})$ where $n$ is the system size. Moreover, they show that for a class of pricing policies, one cannot achieve better than $\Omega(n^{1/3})$ optimality rate. On the other hand, we design a fully dynamic policy, which removes the additional log factor and achieves an optimality rate of $O(n^{1/3})$. Moreover, we strengthen the $\Omega(n^{1/3})$ lower bound to hold for any admissible policy. A detailed comparison with \cite{kim2018value} is provided in Section~\ref{value of DP in large queue system}.

\citet{varma2023dynamic} analyzes a two-arrival policy and establishes an optimality rate of $O(n^{1/3})$. Thus, a two-arrival policy itself is optimal unlike the two-price policy in a one-sided queue \cite{kim2018value} and so one needs a fully dynamic policy for a one-sided queue to achieve optimality. These results show that optimizing revenue in a two-sided queue is inherently easier than in a classical single server queue.
In particular, in a two-sided system, keeping the imbalance close to zero is desirable: when the system is balanced, arrivals from either side can still occur and subsequently be cleared through matching. In a one-sided queue, however, an empty system implies an idle server and hence lost utilization. Achieving near-optimal reward therefore requires maintaining sufficient backlog while avoiding excessive congestion, making simple two-level control intrinsically less effective than in two-sided systems.



Our scaling results are closely related to \cite{balseiro2025dynamic}. Similar curvature-based insights also appear more broadly in the dynamic resource allocation literature \cite{besbes2025dynamic}. In particular, the locally smooth and locally linear regimes 
identified in \cite{balseiro2025dynamic} correspond closely to our concave-like and non-concave-like behavior around capacity, as formalized in \eqref{eq:concave_like}. A key distinction, however, is that the scaling in \cite{balseiro2025dynamic} is governed by the local curvature of the reward function, whereas ours depends on the curvature of the convex hull of the reward function, which is a global property.


{\bf Heavy Traffic in Single Server Queues.}
A classical result of \cite{kingman1961single} shows that, for a $GI/GI/1$ queue with arrival rate $\lambda=\mu-\varepsilon$ and service rate $\mu$, the expected queue length grows on the order of $\Theta(1/\varepsilon)$. Building on this seminal result, a large body of literature develops heavy-traffic approximations for more complex stochastic systems and obtains asymptotic characterizations of queue lengths and system performance for parallel server queues \cite{harrison1998heavy}, multiclass queues \cite{williams1998diffusion, gamarnik2012multiclass}, and input queued switch \cite{maguluri2018optimal}.

Heavy-traffic analysis for queues with dynamic arrivals, however, remains relatively limited. In this work, we characterize the heavy-traffic scaling of a single-server queue with state-dependent arrival control, extending the classical heavy-traffic perspective with exogenous arrivals. We believe our results can serve as building blocks for understanding more complex queueing systems with dynamic arrival control. Prior work on queues with dynamic arrivals has primarily focused on fluid regimes \cite{maglaras2003pricing, maglaras2006revenue, lee2014optimal}, while \cite{kim2018value} heuristically derives the diffusion scaling without a formal proof. Our results can be interpreted as characterizing pre-limit behavior that underlies these fluid and diffusion approximations.

\section{Notation}
In this paper, we use $\mathbb{R}_+$ and $\mathbb{Z}_+$ to denote the sets of non-negative real numbers and non-negative integers, respectively, and $\mathbb{R}_{++}$ and $\mathbb{Z}_{++}$ to denote the sets of positive real numbers and positive integers. We define $[n] := \{0,1,\ldots,n\}$. For any real number $x$, we use $\lfloor x \rfloor$ and $\lceil x \rceil$ to denote the largest integer less than or equal to $x$ and the smallest integer greater than or equal to $x$, respectively.

Throughout the paper, $\varepsilon>0$ denotes a sufficiently small parameter, and all order notation is with respect to $\varepsilon$. For two nonnegative functions $f(\varepsilon)$ and $g(\varepsilon)$, we write $f(\varepsilon) = O(g(\varepsilon))$ if there exists a constant $C>0$ such that $f(\varepsilon) \le C g(\varepsilon)$ for all sufficiently small $\varepsilon$. Similarly, we write $f(\varepsilon) = \Omega(g(\varepsilon))$ if there exists a constant $c>0$ such that $f(\varepsilon) \ge c g(\varepsilon)$ for all sufficiently small $\varepsilon$. We write $f(\varepsilon) = \Theta(g(\varepsilon))$ if both relations hold. 

\section{Model Description}\label{model}
We model a generic service system as an $M_q/M/1$ queue with state-dependent arrivals. In particular, let the number of customers waiting in the system (including the one in service) be denoted as $q$ with state space $\mathcal{S} \subseteq \mathbb{Z}_+$. The arrival rate to the queue is state-dependent and given by the arrival rate function $\lambda: \mathbb{Z}_+\to \mathbb{R}_+$. In turn, the arrivals are governed by a Poisson process with rate $\lambda(q)$ when $q \in \mathbb{Z}_+$ customers are waiting in the system. The server is assumed to serve the customers at a fixed rate of $\mu = 1$. In turn, the service time is exponential with rate $1$.

The queue length $\{q(t) \in \mathbb{Z}_+ : t \geq 0\}$ of an $M_q/M/1$ queue evolves as a continuous time Markov chain (CTMC); more specifically, a birth-and-death chain with birth rate $\lambda(q)$ and death rate $\mu=1$ as depicted in Figure~\ref{fig:birth_and_death}. We restrict our attention to the communicating class involving state $0$, i.e., the state-space of the CTMC is given by $\mathcal{S} = [q_0]$, where $q_0 = \inf\left\{q \geq 1: \lambda(q) = 0\right\}$. If $q_0 = \infty$, then, $\mathcal{S} = \mathbb{Z}_+$. In addition, we assume that the arrival rate is bounded, i.e., $\lambda(q) \in [0, \lambda_{\max}]$ for some $\lambda_{\max} \geq 1$ for all $q \in \mathbb{Z}_+$, ensuring that the CTMC is non-explosive. In a real service system, $\lambda_{\max} \geq 1$ models the size of the available market, while $\lambda(q)$ is the actual arrival rate depending on the incentives offered to the customer by the system operator.




\begin{figure}[htbp]
    \centering
    \begin{tikzpicture}[->, >=stealth, auto, node distance=1.5cm, semithick]

        \node[state] (S0) {0};
        \node[state] (S1) [right=of S0] {1};
        \node[state] (S2) [right=of S1] {2};
        \node (dots) [right=of S2] {$\cdots$};
        \node[state] (Sn) [right=of dots] {$n$};
        \node (dots2) [right=of Sn] {$\cdots$};

        \path
            (S0) edge[bend left] node {$\lambda(0)$} (S1)
            (S1) edge[bend left] node {$1$} (S0)
            (S1) edge[bend left] node {$\lambda(1)$} (S2)
            (S2) edge[bend left] node {$1$} (S1)
            (S2) edge[bend left] node {$\lambda(2)$} (dots)
            (dots) edge[bend left] node {$1$} (S2)
            (dots) edge[bend left] node {$\lambda(n-1)$} (Sn)
            (Sn) edge[bend left] node {$1$} (dots)
            (Sn) edge[bend left] node {$\lambda(n)$} (dots2)
            (dots2) edge[bend left] node {$1$} (Sn);

    \end{tikzpicture}
    \caption{M/M/1 Birth--Death Chain with State-Dependent Arrivals} \label{fig:birth_and_death}
\end{figure}

For completeness, the transition rate matrix (Chapter 2.1 in \cite{norris1998markov}) $Q = (Q_{ij}: i,j\in \mathbb{Z}_+)$ for this CTMC is as follows:
\begin{equation*}
Q_{ij} = \begin{cases}
    -\lambda(i)-\mathbf{1}_{i>0} & \text{for } i \in \mathcal{S} \text{ and } j=i, \\
    \lambda(i) & \text{for } i \in \mathcal{S} \text{ and } j=i+1, \\
    1 & \text{for } i \in \mathcal{S} \backslash \{0\} \text{ and } j=i-1, \\
    0 &\text{elsewhere}.
\end{cases}
\end{equation*}
To ensure positive recurrence of the CTMC, we make the following mild assumptions on the control policy: 
\begin{assumption}\label{assump3}
We assume that $\sum_{i\in\mathcal{S}} \prod_{q=0}^{i}\lambda(q)<\infty$.
\end{assumption}
We refer to any control policy satisfying the above assumption as a ``stable control policy'' or simply a ``stable policy''. Thus, by standard results on CTMC (see e.g., Theorem 3.5.3 in \cite{norris1998markov}), there exists a unique stationary distribution $\{\pi_q\}_{q\in \mathcal{S}}$ of the CTMC given by
\begin{equation*}
\pi_q = \pi_0\prod_{k=0}^q \lambda(k) \ \forall q \in \mathcal{S}, \text{ with } \pi_0 = \frac{1}{\sum_{i\in \mathcal{S}} \prod_{q=0}^i\lambda(q)}. 
\end{equation*}
We define the \textit{efficiency} of the control policy as the long-term average queue lengths $\mathbb{E}[\bar{q}]$, where $\bar{q}$ is a random variable endowed with the stationary distribution $\pi$
In addition to efficient system operation, usually in practice, the goal of the system operator is to maximize a certain objective like throughput, revenue, social welfare, etc. We model this by a generic reward function $F: [0, \lambda_{\max}] \rightarrow \mathbb{R}_+$ and consider the following long-run average reward:
\begin{equation*}
r(\lambda) := \liminf_{T\to \infty}\frac{1}{T}\int_{0}^T F(\lambda(q(t))) dt = \mathbb{E}_\pi\left[F(\lambda(\bar{q}))\right],
\end{equation*}
where the last equality holds due to the Ergodic Theorem (see Theorem 3.8.1 in \cite{norris1998markov}). As an example, a linear $F$ would correspond to throughput maximization; or a concave $F$ could be used to model revenue maximization \cite{varma2023dynamic}. More generally, $F$ may not even be concave for systems with strategic agents \cite{varma2020near}. As we will see, our results are general and encompass all of these cases. We make the following mild structural assumptions on $F$:
\begin{assumption}\label{assump2}
    We assume that $F : [0,\lambda_{\max}] \rightarrow \mathbb{R}_+$ is continuously differentiable. 
\end{assumption} 
{
\begin{assumption} \label{eq:smooth_F}
    We make the following assumptions for the second and higher order derivative of $F$:
    \begin{itemize}
        \item The second order derivative $F''$ exists and continuous in $(0,\lambda_{\max})$, and $F''(1)<0$.
        \item The third and fourth order derivative $F^{(3)}$ and $ F^{(4)}$ exist and bounded, i.e. $\exists M, K>0$ such that
        \[
        \sup\limits_{x\in [0,\lambda_{\max}]} |F^{(3)}(x)|\leq M, \quad  \sup\limits_{x\in [0,\lambda_{\max}]} |F^{(4)}(x)|\leq K.
        \]
    \end{itemize}
\end{assumption}
While Assumption~\ref{assump2} is already assumed in Assumption~\ref{eq:smooth_F}, not all results in the paper require the stronger assumption (Assumption~\ref{eq:smooth_F}). In what follows, we explicitly specify for each theorem and proposition, the assumptions under which it holds.
}

In addition, without loss of generality, we assume $F(0) = 0$.
To understand the reward maximization problem, we consider an alternate optimization problem defined as follows:
\begin{align}
    F^\star = \sup_{\alpha \in \mathcal{P}([0, \lambda_{\max}]): \ X \sim \alpha}  \mathbb{E}_\alpha[F(X)] \quad
    \text{s.t.} \quad  \mathbb{E}_\alpha[X] \leq 1,\label{eq:def Fstar}
\end{align}
where $\mathcal{P}([0, \lambda_{\max}])$ is the set of probability measures on $[0, \lambda_{\max}]$. It turns out that the above optimization problem provides an upper bound on the achievable reward $r(\lambda)$ under any stable control policy as stated below:
\begin{proposition}\label{lemma3.2}
It holds that $\mathbb{E}_\pi[F(\lambda(\bar q))]\leq F^\star$ for any control policy satisfying Assumption~\ref{assump3}.
\end{proposition}
The above proposition is proved by constructing a feasible $\alpha$ such that the objective function of \eqref{eq:def Fstar} equals $\mathbb{E}_{\pi}[F(\lambda(\bar{q}))]$ for any control policy $\lambda$. In particular, define $\alpha \in \mathcal{P}([0, \lambda_{\max}])$ as the pushforward measure of the steady-state distribution $\pi$ by the control policy $\lambda$ and note that $\mathbb{E}_\alpha[f(X)] = \mathbb{E}_\pi[f(\lambda(\bar{q}))]$ for any function $f$. In particular, we have $\mathbb{E}_\alpha[F(X)] = \mathbb{E}_\pi[F(\lambda(\bar{q}))]$. In addition, under the stability assumption (Assumption~\ref{assump3}), one can show that the effective arrival rate is less than the service rate, and so, $\mathbb{E}_\alpha[X] = \mathbb{E}_\pi[\lambda(\bar{q})] < 1$ establishing feasibility of $\alpha$. We refer to Appendix~\ref{proof for prop 4.1} for the details of the proof.


Motivated by Proposition~\ref{lemma3.2}, we now define the regret of a control policy as $R(\lambda) := F^\star - \mathbb{E}_\pi[F(\lambda(\bar q))] \geq 0$ and focus on minimizing the regret. More precisely, our goal is to understand Pareto frontier between optimizing \textit{cost} (defined as regret minimization) and \textit{efficiency} (defined as the average queue length minimization) as stated in the following optimization problem:
\begin{align}
    q^\star = \min_{\lambda \in \Lambda}\,\mathbb{E}_{\pi}[\bar{q}] 
    \quad \text{s.t.}\quad R(\lambda) = F^\star-\mathbb{E}_{\pi}[F(\lambda(\bar{q}))]\leq \varepsilon, \label{eq:constraint}
\end{align}
where $\Lambda$ is some set of control policies satisfying Assumption~\ref{assump3}. We now show that the above optimization problem can be used to model several settings in the literature.

\textbf{Case 1 (Classical Heavy Traffic):} Consider a linear reward function ($F(x) = x$) and $\Lambda$ to be the set of all stable \textit{static} policies, i.e., $\lambda \equiv c$ for some $c \in [0, 1)$. Then, we simply have $F^\star = 1$ by \eqref{eq:def Fstar}, and so, we need $c \geq 1 -\varepsilon$ to ensure at most $\varepsilon$ regret as in \eqref{eq:constraint}. In turn, we recover the classical heavy traffic scaling of $q^\star = \Theta(\frac{1}{\varepsilon})$.   

\textbf{Case 2 (Throughput Maximization):} Consider a linear reward function ($F(x) = x$) and $\Lambda$ to be the set of \textit{all} stable policies. In this case, $F^\star = 1$ continues to hold like Case 1. Thus, at most $\varepsilon$ regret corresponds to $\mathbb{E}_{\pi}[\lambda(\bar{q})] \geq 1 - \varepsilon$, i.e., the throughput is at least $1-\varepsilon$. Unlike the previous case, we are now allowed to vary the arrival rate as a function of the queue lengths. We shall see that this freedom allows us to reduce the optimal expected queue length $q^\star$ from $\Theta(\frac{1}{\varepsilon})$ in Case 1 to $\Theta(\log\frac{1}{\varepsilon})$.

\textbf{Case 3 (Revenue Maximization):} Consider a setting where the arrival rates are modulated via prices, i.e., the system operator sets a state-dependent price $(p(q))$, which in turn, determines the arrival rate $(g(p(q)))$, defined via the demand curve $g$. Then, consider $F$ to be the revenue function defined as the product of the price and arrival rate, i.e., $F(\lambda) = \lambda g^{-1}(\lambda)$. A popular assumption in the revenue management literature is that $F$ is a (strictly) concave function. Under this setting, we show that $q^\star = \Theta(\frac{1}{\sqrt{\varepsilon}})$. Thus, one continues to do better than a static control policy (Case 1) but does not achieve exponential improvement (Case 2).

In the next section, we establish the above three cases with much more generality. In particular, we show that we fall into Case 1 whenever $\lambda_{\max} = 1$ irrespective of the reward function $F$. Otherwise, if $\lambda_{\max} > 1$, then we fall into Case 3 whenever $F$ is concave-like (to be defined later) and into Case 2 for all other reward functions $F$.

{First, consider the case that there exists $\opt < 1$ such that $F(\opt) = F^\star$. In this case, the static policy $\lambda \equiv \opt$ ensures $R(\lambda) = 0$ while achieving $\mathbb E[\bar q] = \Theta_{\varepsilon}(1)$.

To avoid trivial scenarios, we make the following two assumptions:
\begin{assumption}\label{assump4}
    Assume $ F(x)<F^\star \text{ for any } x\in [0,1)$.
\end{assumption}

\begin{assumption}\label{assump F'(1)>0}
    We assume $F'(1)>0$.
\end{assumption}
Observe that $ F(x)<F^\star$ for all $x\in [0,1)$ implies that the optimizer of \eqref{eq:def Fstar} is achieved at the boundary $(\mathbb{E}_{\alpha}[X] = 1)$, or in other words, the service system is capacity constrained. 

This condition also implies that the reward does not decrease as $x \to 1$, and hence $F'(1)\ge 0$. In particular, when the Assumption~\ref{assump F'(1)>0} $F'(1)>0$ hold, the reward is strictly increasing as we approach the capacity limit. We primarily focus on this case.

We shall see that when Assumption~\ref{assump F'(1)>0} is not satisfied, even the static policy performs as well as the dynamic policy as we will see in the next section. Thus, to meaningfully distinguish the advantage of dynamic policy, we make this assumption.
}

\section{Main Results: Efficiency-Reward Trade-Off}
\subsection{Small-Market Control Policies} \label{small_market}
We consider the impact of market size on optimization problem \eqref{eq:constraint}. We show that if the market size is small, i.e. $\lambda_{\max} = 1$, then, we recover the classical heavy traffic scaling, i.e., there is no benefit of dynamic control. This result is formalized below.
\begin{proposition}\label{prop small market}
   Under Assumptions \ref{assump2}, \ref{assump4} and \ref{assump F'(1)>0}, $q^\star= \Omega\left(\frac{1}{\varepsilon}\right)$ for all small-market $(\lambda_{\max} = 1)$ control policies. 
\end{proposition}
Intuition behind Proposition~\ref{prop small market} is that we want to reduce the time the system is idle to improve its efficiency, and one way to achieve this is to attract more customers when the system is less busy. However, control policies with small market size are unable to do so. We leave the proof part in Appendix~\ref{proof_small_market}.

{
Note that Assumption~\ref{assump F'(1)>0} is essential for the above proposition. To see this, consider the following example. Suppose $F^\star = F(1)$ and let $\lambda = 1 - \sqrt{\varepsilon}$. When $F'(1)=0$, a second-order Taylor expansion yields
\(
F(1) - F(1-\sqrt{\varepsilon}) \sim \frac{-F''(1)}{2}\varepsilon = \Theta(\varepsilon).
\)
Thus, one can achieve $\Theta(\varepsilon)$ regret with $\lambda = 1 - \sqrt{\varepsilon}$, which in turn implies $\mathbb{E}[\bar{q}] = \Theta(1/\sqrt{\varepsilon})$.
}

{
We next turn to the large-market regime, where $\lambda_{\max} > \mu = 1$. In the next section, we show that the flexibility to set the arrival rate larger than the capacity, allows one to go beyond the classical heavy-traffic scaling of $1/\varepsilon$ by setting large arrival rates when the system is less congested. However, such a benefit can arise only when there is some reward incentive to set $\lambda > 1$. For example, if $F(x) < F(1)$ for all $x \in [1, \lambda_{\max}]$, then, optimal policies would avoid placing probability mass above the capacity $\mu = 1$, leading to behavior analogous to the small-market regime. Such a scenario can arise if either $F$ has a downward jump discontinuity at $\mu = 1$ or if $F^\prime(1)$ does not exist, instead, we have $\lim_{x \to 1^{-}} F^\prime(x) > 0$ and $\lim_{x \to 1^{+}} F^\prime(x) < 0$. These edge cases are avoided by assuming $F^\prime(1)$ exists and $F^\prime(1) \geq 0$.
}

\subsection{Large-Market Control Policies}
Now we focus on $\lambda_{\max} > 1$ for the rest of the paper. We start by analyzing the structural properties of \eqref{eq:def Fstar} in the following proposition:
\begin{proposition} \label{prop:how_many_supp_point}
If $F$ satisfies \eqref{eq:concave_like}, then, $\alpha(\{1\})=1$ is the unique optimal solution of \eqref{eq:def Fstar}. Otherwise, there exists $\opt_1 \in (1, \lambda_{\max}], \opt_2 \in [0, 1), p^\star \in (0, 1)$ such that $\alpha(\{\opt_1\})=p^\star$ and $\alpha(\{\opt_2\})=1-p^\star$ is an optimal solution of \eqref{eq:def Fstar}.
\end{proposition}
We now provide the intuition for the above result and defer the proof details to Appendix~\ref{opt_sol_large_market}. Observe that \eqref{eq:def Fstar} is a class of risk averse optimization problem
that falls into the category of the problem of moments. So, by Proposition 6.40 in \cite{moment_book}, there always exists an optimal solution of \eqref{eq:def Fstar} with at most two support points as we have exactly one moment constraint. Thus, either all optimal solutions of \eqref{eq:def Fstar} are Dirac measures, or there exists an optimal solution of \eqref{eq:def Fstar} with exactly two support points. Figure~\ref{fig:2 conditions} illustrates these two possible scenarios. 
Note that, under the former, the optimal solution is necessarily $\alpha(\{1\}) = 1$ due to Assumption~\ref{assump4}. Moreover, the former holds exactly when $F(1)$ is larger than $\mathbb{E}_{\alpha}[F(X)]$ for any discrete measure $\alpha$ with two support points and $\mathbb{E}_{\alpha}[X] = 1$, i.e., \eqref{eq:concave_like} is satisfied.


\begin{figure}[htbp]
  \centering
  \begin{subfigure}[t]{0.4\textwidth}
    \centering
    \includegraphics[width=\linewidth]{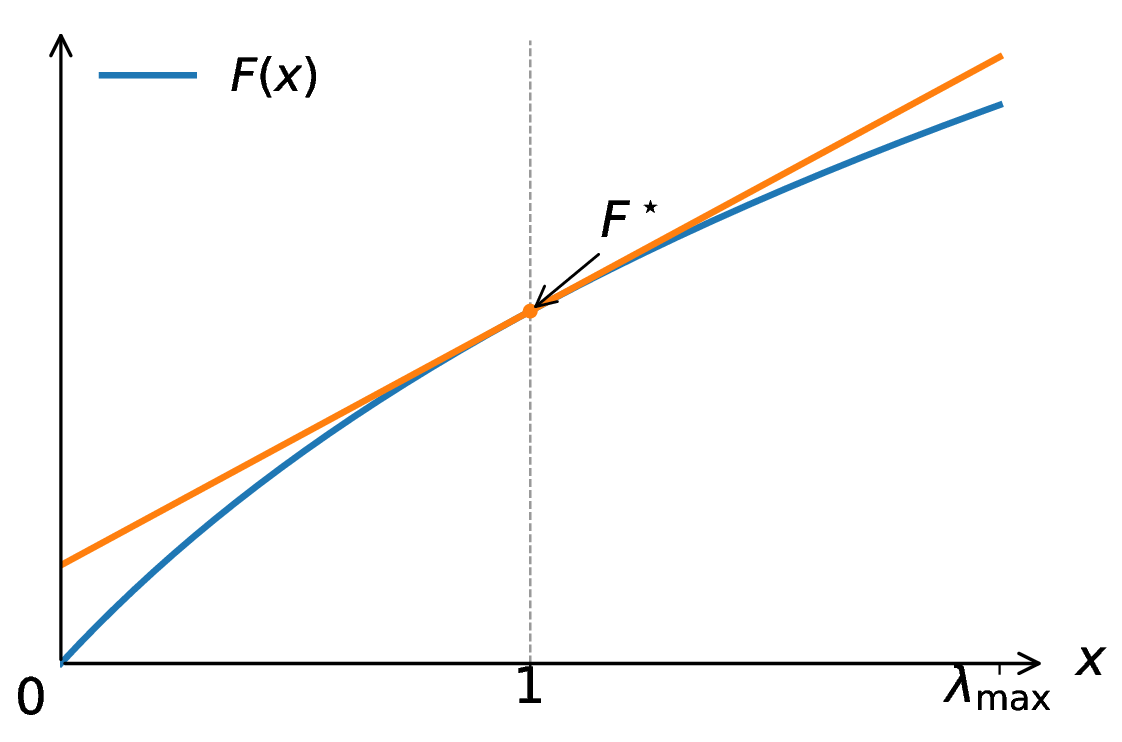}
    \subcaption{One support point ($F$ satisfies \eqref{eq:concave_like})}
    \label{fig:one point}
  \end{subfigure}
  \hspace{0.05\textwidth}
  \begin{subfigure}[t]{0.4\textwidth}
    \centering
    \includegraphics[width=\linewidth]{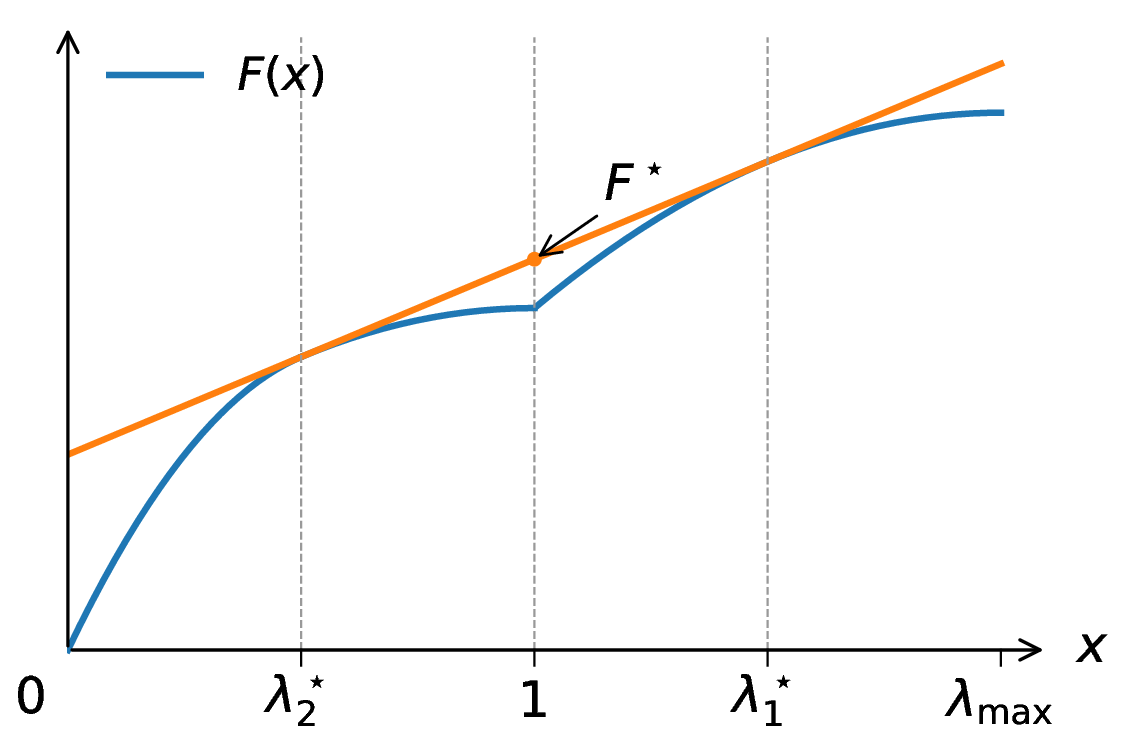}
    \subcaption{Two support point ($F$ violates \eqref{eq:concave_like})}
    \label{fig:two points}
  \end{subfigure}
  \caption{Figure for two conditions}
  \label{fig:2 conditions}
\end{figure}

It turns out that the solution of \eqref{eq:constraint} depends crucially on which case we fall under. Below we present the main results of this paper corresponding to the two cases.

\subsubsection{Deterministic Optimal Solution}\label{D_O_S}
We first consider the case when \eqref{eq:concave_like} is satisfied, i.e., the unique optimal solution of \eqref{eq:def Fstar} is equal to 1 with probability 1. To analyze this case, we need the following smoothness assumption on $F$, {which is Assumption~\ref{eq:smooth_F}}.


\begin{theorem}\label{general:1 support}
    Under Assumptions~\ref{assump2}, \ref{eq:smooth_F}, \ref{assump4} and \ref{assump F'(1)>0}, if $F$ satisfies \eqref{eq:concave_like} and $\lambda_{\max} > 1$, then, there exists $C_1, C_2, \varepsilon_0 > 0$ (depending on $F$) such that for all $\varepsilon \leq \varepsilon_0$, we have
    \begin{equation*}
        \frac{C_1}{\sqrt{\varepsilon}} \leq q^\star \leq \frac{C_2}{\sqrt{\varepsilon}}.
    \end{equation*}
If $F$ is strictly concave, then one can take $C_1 = \sqrt{-\frac{1}{8}F''(1)}, C_2 = \sqrt{-7F''(1)}$.
\end{theorem}
To prove the above theorem, we establish a universal lower bound $\Omega(\frac{1}{\sqrt{\varepsilon}})$ on $\mathbb{E}[\bar{q}]$ under any feasible control policy of \eqref{eq:def Fstar}. In addition, we also construct a control policy with $\mathbb{E}[\bar{q}] \leq O(\frac{1}{\sqrt{\varepsilon}})$ concluding the proof. Note that our proposed control policy strictly improves over all small-market control policies, establishing the value of access to a large customer pool $(\lambda_{\max} > 1)$. In addition, our proposed control policy has an optimal dependence in terms of $\varepsilon$, thus, optimizing the trade-off between regret and queue length. We refer the reader to Section~\ref{subsection:1pointanalysis} for a detailed proof of the above theorem. 
In addition, when \(F\) is strictly concave, the lower and upper bounds differ only by a universal constant term, indicating our proposed analytical policy performs well.

To ensure low regret, we would like our control policy to be close to the optimal solution of \eqref{eq:def Fstar}, i.e., $\lambda(q) \approx 1$ for all $q \in \mathcal{S}$. For instance, if we consider $\lambda \equiv 1$, then, we incur zero regret, however, the queueing system is unstable (null recurrent). Thus, we need to necessarily perturb $\lambda(q)$ around $1$ to ensure stability and finite queue lengths. On these lines, \cite{kim2018value} proposed a ``two-price policy'', when interpreted in our context
, corresponds to $\lambda(q) = 1 + k_1$ for $q \leq \tau$ and $\lambda(q) = 1 - k_2$ for $q > \tau$ with appropriately chosen $k_1, k_2, \tau$ as a function of $\varepsilon$. In Section~\ref{sec:additional_log}, we show that any such policy leads to $\mathbb{E}[\bar{q}] \geq \Omega(\sqrt{\frac{\log 1/\varepsilon}{\varepsilon}})$, incurring an additional $\sqrt{\log\frac{1}{\varepsilon}}$ factor compared to the universal lower bound of $\Omega(\frac{1}{\sqrt{\varepsilon}})$. This is mainly because of the non-smoothness of this policy as it abruptly jumps at $q = \tau$. We improve these results by proposing a fully dynamic policy that changes $\lambda(q)$ gradually as a function of $q$. In particular, let $B = \frac{C}{\sqrt{\varepsilon}}$ for some $C>0$, and consider the following control policy
\begin{equation}\label{symmetry_dynamic}
\lambda(q) = 
\begin{cases}
    \frac{(m+q+2)^2}{(m+q+1)^2}, &\quad q = 0,1,\cdots, B-1\\
    \frac{(m+2B-q)^2}{(m+2B-q+1)^2}, &\quad q = B,B+1,\cdots, 2B-1\\
    0, &\quad q\geq 2B,
\end{cases}
\end{equation}
as depicted in Figure~\ref{fig:lambda-comparison}, 
where
\(
m=\left\lceil
\frac{1}{\sqrt{\overline{\lambda}_{\max}}-1}
\right\rceil-1,
\)
for some design parameter
\(
\bar\lambda_{\max}
\in
[1+\varepsilon^{\alpha},\,\lambda_{\max}],
\)
where \(\alpha \in (0, 1/4)\) is a fixed constant.
This choice guarantees that
\(
\lambda(q)\le
\bar\lambda_{\max}
\le
\lambda_{\max},
\)
for all \(q\ge 0\).

We show that this policy achieves $\mathbb{E}[\bar q] \leq O(\frac{1}{\sqrt{\varepsilon}})$, improving upon the previously proposed ``two-price policy''. Unlike the two-price policy, this policy is \textit{not} a small perturbation around $1$ as $\lambda(q) \gg 1$ for small $q$ and $\lambda(q) \ll 1$ for large $q$. Such a design provides a large positive drift for small $q$ which reduces the idle time, and a large negative drift for large $q$ which ensures low queue lengths. Moreover, the steady-state distribution concentrates around $B = \frac{C}{\sqrt{\varepsilon}}$, where $\lambda(q)$ is a small perturbation around $1$, ensuring low regret.
\begin{figure}[htbp]
  \centering
  \begin{subfigure}[t]{0.4\textwidth}
    \centering
    \includegraphics[width=\linewidth]{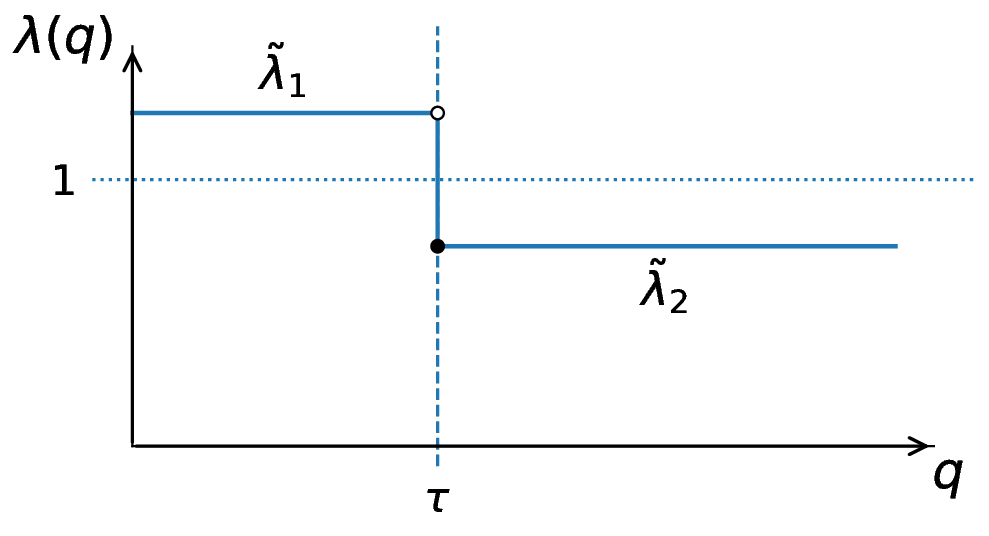}
    \subcaption{Two-arrival rate policy}
    \label{fig:lambda-static}
  \end{subfigure}
  \hspace{0.05\textwidth}
  \begin{subfigure}[t]{0.4\textwidth}
    \centering
    \includegraphics[width=\linewidth]{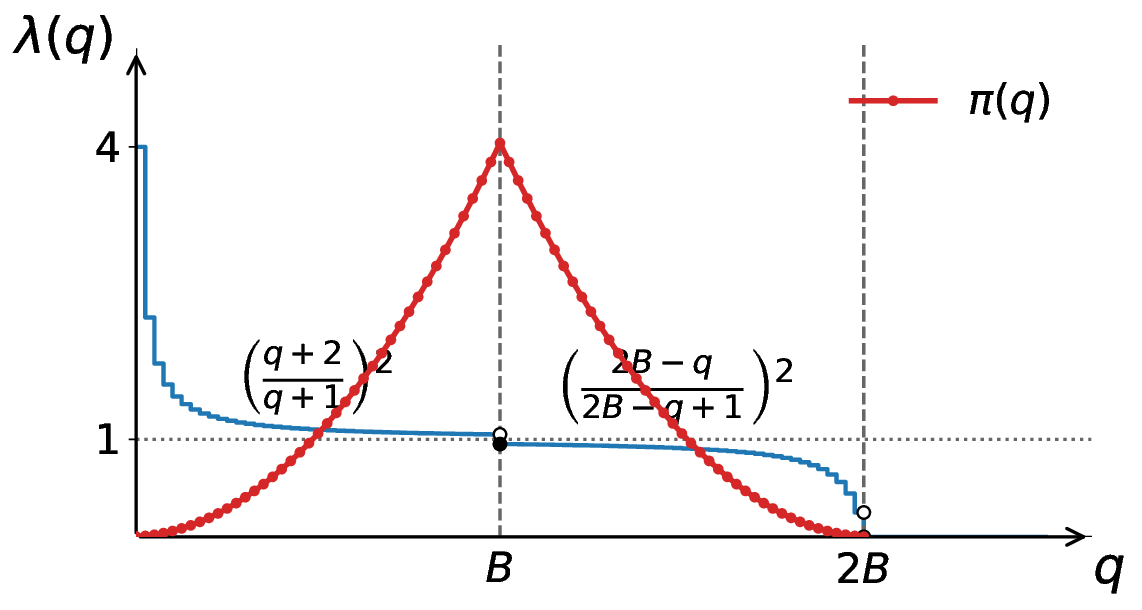}
    \subcaption{Fully dynamic policy with stationary distribution}
    \label{fig:lambda-dynamic}
  \end{subfigure}
  \caption{Comparison of arrival policies.}
  \label{fig:lambda-comparison}
\end{figure}


\subsubsection{Probabilistic Optimal Solution} \label{P_O_S}
We now consider the case when \eqref{eq:concave_like} is not satisfied, i.e., \eqref{eq:def Fstar} has an optimal solution with exactly two support points. For this case, we show that one can achieve a $O(\log\frac{1}{\varepsilon})$ expected queue length as stated in the following theorem:
\begin{theorem}\label{general:2 supports}
    Under Assumptions~\ref{assump2} and \ref{assump4}, if $F$ does not satisfy \eqref{eq:concave_like}, and $\lambda_{\max} > 1$, then,
    \begin{equation*}
\log_{\lambda_{\max}} \frac{1}{\varepsilon} \leq q^\star \leq \log_{\opt_1} \frac{1}{\varepsilon},
    \end{equation*}
    where $\opt_1$ is defined in Proposition~\ref{prop:how_many_supp_point}. If \(F\) is linear, then one can take $x_1^\star = \lambda_{\max}$.
\end{theorem}
To prove above theorem, we establish a universal lower bound $\Omega(\log \frac{1}{\varepsilon})$ on $\mathbb{E}[\bar{q}]$ under any feasible control policy of \eqref{eq:def Fstar} and also present a feasible control policy that achieves $\mathbb{E}[\bar{q}] = O(\log \frac{1}{\varepsilon})$. Thus, our proposed control policy has an optimal dependence on $\varepsilon$. Moreover, in this setting, we exponentially improve the $\mathbb{E}[\bar{q}]$ compared to the previous sections. We defer the details of the proof to Section~\ref{subsection:2pointanalysis} and focus on explaining the intuition here.

Similar to the previous section, we would like to construct a control policy that is close to the optimal solution of \eqref{eq:def Fstar} and also ensures low queue lengths. In this case, the optimal solution of \eqref{eq:def Fstar} is a discrete measure $\alpha^\star$ with two support points $\opt_2 < 1 < \opt_1$ and $\alpha^\star(\{\opt_1\}) = p^\star$. One could then set $\lambda(q) = \opt_1$ w.p. $p^\star$ and $\opt_2$ otherwise for all $q \in \mathcal{S}$. However, this would lead to an effective arrival rate of $p^\star \opt_1 + (1-p^\star)\opt_2 = 1$ which results in an unstable system. Our novel idea is then to consider the following control policy:
\begin{equation} \label{eq:control_policy_two_supp_point}
    \lambda(q) = \begin{cases}
        \opt_1 &\textit{if } q \leq \tau, \\
        \opt_2 &\textit{if } q > \tau,
    \end{cases}
\end{equation}
for an appropriately defined $\tau$ such that $\sum_{q=0}^\tau \pi_q \approx p^\star$. While $\lambda(q)$ is far away from the optimal solution of \eqref{eq:def Fstar} for all $q \in \mathcal{S}$, by appropriately picking $\tau$, we ensure that the pushforward of $\pi$ by $\lambda$ is approximately $\alpha^\star$, which then implies $\mathbb{E}_\pi [F(\lambda(\bar{q}))] \approx F^\star$. Moreover, as $\opt_2 < 1$, the above control policy imposes a negative drift, ensuring low queue lengths. As we do not perturb the support points $(\opt_1, \opt_2)$ of the optimal policy while designing $\lambda(q)$, we are able to ensure exponentially lower queue lengths in this case as compared to Theorem~\ref{general:1 support}.

The work by \cite{tsitsiklis2013power} aims to minimize the expected queue length in a multi-server resource pooling model while keeping the throughput close to maximum capacity, which is similar to Theorem~\ref{general:2 supports} with a linear $F$. Our analysis extends their result to the broader class of reward functions $F$ that are not concave-like, establishing analogous order results under more general structural conditions.

\subsubsection{Connections to the Locally Polyhedral Condition \cite{huang2009delay}} \label{polyhedral}
In our setting, the fluid benchmark~\eqref{eq:def Fstar} admits a natural Lagrange dual formulation given by
\begin{equation}\label{dual}
    \max\quad   q(U)
        \text{ s.t. }\quad   U \geq 0, \text{ with } q(U) = \inf_{\alpha\in \mathcal{P}([0,\lambda_{\max}])} \{\mathbb{E}_\alpha[-F(X)+UX]-U\}.
\end{equation}
We show that $F$ violating \eqref{eq:concave_like} is equivalent to the dual function $q$ being locally polyhedral \cite{huang2009delay} around its optimum $U^\star$.  
\begin{theorem}\label{relative of 18eq}
    Let $U^\star$ be the optimal solution of \eqref{dual} and assume that $F$ does not satisfy \eqref{eq:concave_like}, then, there exists $L > 0$ such that
\begin{equation}\label{paper eq18}
        q(U^\star) \geq q(U) +L\lvert U^\star-U\rvert
    \end{equation}
    for all $U\geq 0$, where $q(U^\star) = \max\limits_{U\geq 0} q(U) = -F^\star$.
\end{theorem}
The locally polyhedral condition was first introduced in \cite{huang2009delay}, in the context of stochastic network optimization. We refer the reader to Appendix~\ref{Polyhedral} for the proof of the above theorem and the comparison with \cite{huang2009delay}.

\section{When \(F\) is concave-like}\label{subsection:1pointanalysis}

\subsection{Proof of the upper bound of $q^\star$ in Theorem~\ref{general:1 support}}\label{sec:qstar_upper_bound}

Recall the following fully dynamic policy introduced in \eqref{symmetry_dynamic}
and let $B=\left\lceil\frac{\sqrt{-7F''(1)}}{\sqrt{\varepsilon}}\right\rceil$.
We now show that $\mathbb{E}[\bar q] = B$, and $R(\lambda) = F(1)-\mathbb{E}_\pi[F(\lambda(\bar q))] \leq \varepsilon$. We start by establishing the former.
Using the detailed balance equations, we get the following stationary distribution: 
\begin{align}
\pi_0 &= \frac{3(m+1)^2}{3(2B+1)m^2 + 3(B+1)^2(2m+1)+B(B+1)(2B+1)}, \label{plug_pi_0}
\\
\text{ and }\pi_i&=\begin{cases}
    \pi_0\left(\frac{m+1+i}{m+1}\right)^2, \quad i=0,1,\cdots,B\\
    \pi_0\left(\frac{m+2B+1-i}{m+1}\right)^2, i = B+1,B+2,\cdots, 2B
\end{cases} \notag
\end{align}
Therefore
\begin{align*}
    \mathbb{E}[\bar q] &= \sum_{i=0}^{2B} i\pi_i = \sum_{i=0}^{B-1}i\pi_i + B\pi_B + \sum_{i=B+1}^{2B}i\pi_i
    = \sum_{i=0}^{B-1}i\pi_i + B\pi_B + \sum_{i=0}^{B-1}(2B-i)\pi_{2B-i}\\
    &\overset{(a)}= \sum_{i=0}^{B-1}i\pi_i + B\pi_B + \sum_{i=0}^{B-1}(2B-i)\pi_i 
    = 2\sum_{i=0}^{B-1}B\pi_i+B\pi_B
    \overset{(b)}= B\sum_{i=0}^{B-1} \pi_i+B\sum_{i=B+1}^{2B}\pi_i + B\pi_B \\
    &= B\sum_{i=0}^{2B}\pi_i = B \leq \frac{\sqrt{-7F''(1)}}{\sqrt{\varepsilon}} + 1,
\end{align*}
where (a), (b) is by $\pi_i=\pi_{2B-i}$ for all $i$. Next, we will show $R(\lambda) = F(1)-\mathbb{E}_\pi[F(\lambda(\bar q))] \leq \varepsilon$, with detailed mathematical proof in Appendix~\ref{longest_proof}. 

Therefore, the fully dynamic policy we provided satisfies the regret constraint \eqref{eq:constraint}, hence $q^\star = B \leq \frac{\sqrt{-7F''(1)}}{\sqrt{\varepsilon}} + o(\frac{1}{\sqrt{\varepsilon}})$.


\subsection{Proof of the lower bound of $q^\star$ in Theorem~\ref{general:1 support}}\label{sec:qstar_lower_bound}

We present the proof of the lower bound in Theorem~\ref{general:1 support} for the special case in which \(F\) is strictly concave, stated in Proposition~\ref{prop:strictly_concave}. The extension to the general concave-like \(F\) is provided in Appendix~\ref{concavelike_nonconcave}. To facilitate this extension, we do not require $F''(1)$ to  exist.


\begin{proposition} \label{prop:strictly_concave}
    Let $F: [0, \lambda_{\max}] \rightarrow \mathbb{R}_+$ be a continuously differentiable strictly concave function. In addition, $F$ is thrice continuously differentiable for $x \in (1, \lambda_{\max}]$ with $\lim\limits_{x \rightarrow 1^+} F''(x) < 0$ and $\sup\limits_{x > 1}|F^{(3)}(x)| \leq M$, 
    then, there exists $C_1, \varepsilon_0 > 0$ (depending on $F$), such that $q^\star \geq C_1/\sqrt{\varepsilon}$ for all $\varepsilon\leq \varepsilon_0$.
\end{proposition}

This proof is one of the main contributions of the paper, as it extends the lower bound in \cite{kim2018value} from a restricted class of policies to arbitrary control policies. 


\begin{proof}[Proof of Proposition~\ref{prop:strictly_concave}]\label{proof_of_1/sqrt(eps)}
Fix a control policy $\lambda$ and let $\{\pi_i\}_{i=0}^\infty$ be the stationary distribution with $\pi_i = 0$ for all $i \notin \mathcal{S}$ as $\mathcal{S}$ is the state space. For a given $j \in \mathcal{S}$, we start by upper bounding $\pi_j$. Since $F$ is concave, for any $c\in [0,\lambda_{\max}]$, we have:
\begin{align*}
    F(1)-\varepsilon \leq{}& \mathbb{E}_\pi[F(\lambda(\bar q))] = \sum_{i \in \mathcal{S}} F(\lambda(i))\pi_i = \sum_{i \in \mathcal{S}} F\left(\frac{\pi_{i+1}}{\pi_i}\right)\pi_i\\
    \overset{(a)}{\leq}{}& \sum_{i=0}^j \left(F(c)+F'(c)\left(\frac{\pi_{i+1}}{\pi_i}-c\right)\right)\pi_i + \sum_{i \in \mathcal{S} \backslash [j]} \left(F(1)+F'(1)\left(\frac{\pi_{i+1}}{\pi_i}-1\right)\right)\pi_i\\
    ={}& \sum_{i=0}^j\left(\left(F(c)-cF'(c)\right)\pi_i + F'(c)\pi_{i+1}\right)+\sum_{i \in \mathcal{S} \backslash [j]}\left(\left(F(1)-F'(1)\right)\pi_i+F'(1)\pi_{i+1}\right)\\
    ={}& \left(F(c)-cF'(c)\right)\pi_0 + \sum_{i=1}^j\left(F(c)-(c-1)F'(c)\right)\pi_i + \left(F'(c)+F(1)-F'(1)\right)\pi_{j+1} \\
    &+ \sum_{i \in \mathcal{S} \backslash [j+1]} F(1)\pi_i
\end{align*}
where step (a) uses concavity to upper bound $F$ by a tangent at $c$ when $i \leq j$ and by a tangent at $1$ otherwise. Therefore we have:
\begin{align}
    \left(F'(1)-F'(c)\right)\pi_{j+1}&\leq \varepsilon+\left(F(c)-cF'(c)-F(1)\right)\pi_0+\sum_{i=1}^j\left(F(c)-(c-1)F'(c)-F(1)\right)\pi_i\notag\\
    &= \varepsilon-F'(c)\pi_0+\sum_{i=0}^j\left(F(c)-(c-1)F'(c)-F(1)\right)\pi_i \label{F'(1)=0_appendix_from}
\end{align}
Now, let $\delta = \sqrt{\frac{2\varepsilon}{-F''(1)}}$ and set $c=1+\delta$. Also, denote by $F''(1) = \lim_{x \rightarrow 1^+} F''(x)$. Then, we get
\begin{align*}
    &F(c)-(c-1)F'(c)-F(1) = F(1+\delta) - \delta F'(1+\delta) - F(1)\\
    \leq{}& F(1)+F'(1)\delta + \frac{F''(1)}{2}\delta^2 +\frac{M}{6}\delta^3 - \delta\left(F'(1)+F''(1)\delta-\frac{M}{2}\delta^2\right) - F(1)\\
    ={}& -\frac{F''(1)}{2}\delta^2 + \frac{2M}{3}\delta^3,
\end{align*}
and $F'(1)-F'(c) = F'(1)-F'(1+\delta) \geq -F''(1)\delta - \frac{M}{2}\delta^2$. 
Lastly, $-F'(c)=-F'(1+\delta)\leq -F'(1)-F''(1)\delta+\frac{M}{2}\delta^2<0$ for $\varepsilon$ small enough as $\delta = \sqrt{\frac{2\varepsilon}{-F''(1)}}\to 0$ as $\varepsilon \to 0$. Therefore, by the above inequalities, we have:
\begin{align} \label{cite_in_appendix}
    \left(-F''(1)\delta-\frac{M}{2}\delta^2\right)\pi_{j+1} &\leq \varepsilon +\sum_{i=0}^j\left(-\frac{F''(1)}{2}\delta^2+\frac{2M}{3}\delta^3\right)\pi_i\leq \varepsilon+\frac{-F''(1)}{2}\delta^2+\frac{2M}{3}\delta^3.
\end{align}
Note that $F''(1) < 0$ as $F$ is strictly concave, so, for $\varepsilon$ small enough, we have $-F''(1)\delta-\frac{M}{2}\delta^2 > 0$ as $\delta = \sqrt{\frac{2\varepsilon}{-F''(1)}}$. Thus, we get
\begin{align*}
    \pi_{j+1}&\leq \left(2\sqrt{\frac{-F''(1)\varepsilon}{2}} + \frac{2M}{3}\frac{2\varepsilon}{-F''(1)}\right)\left(-F''(1)-\frac{M}{2}\sqrt{\frac{2\varepsilon}{-F''(1)}}\right)^{-1} \\
    &= \left(\sqrt{\frac{2\varepsilon}{-F''(1)}} + \frac{4M \varepsilon}{3F''(1)^2}\right)\left(1-\frac{M\sqrt{2\varepsilon}}{2 (-F''(1))^{3/2}}\right)^{-1} \leq \sqrt{\frac{2\varepsilon}{-F''(1)}} + o(\sqrt{\varepsilon}),
\end{align*}
where the last inequality holds as $1/(1-x) \leq 1+2x$ for $x \in (0, 1/2)$.

Next, we upper bound $\pi_0$. Since $F$ is concave, so $\mathbb{E}_\pi[F(\lambda(\bar q))] \leq F(\mathbb{E}_\pi[\lambda(\bar q)]) \leq F(1) + F'(1)(\mathbb{E}_\pi[\lambda(\bar{q})]-1)$, and by $\mathbb{E}_\pi[F(\lambda(\bar q))]\geq F(1)-\varepsilon$, we get $\mathbb{E}_\pi[\lambda(\bar q)] \geq 1 - \varepsilon/F'(1)$. Thus, we get (refer to Equation \eqref{idle time} in Section~\ref{proof for prop 4.1} for details) 
\begin{equation}
\pi_0=1- \mathbb{E}_\pi[\lambda(\bar q)] \leq \frac{\varepsilon}{F'(1)}  \leq \sqrt{\frac{2\varepsilon}{-F''(1)}} \label{appendix_pi_0},
\end{equation}
where the last inequality holds for $\varepsilon > 0$ small enough.
Therefore, we get
\[
\pi_i\leq \sqrt{\frac{2\varepsilon}{-F''(1)}} +o(\sqrt{\varepsilon}) \quad \text{ for all } i \in \mathbb{Z}_+,
\]
so the the following holds:
\begin{equation} \label{eq:one point lower bound}
\begin{aligned}
    q^\star\geq \min\,\, &\mathbb{E}[\bar q]=\sum_{i=0}^\infty i\pi_i
    \text{ s.t. }\,\,  \pi_i\leq \sqrt{\frac{2\varepsilon}{-F''(1)}} + o(\sqrt{\varepsilon}) \quad \forall i \in \mathbb{Z}_+, \quad  \sum_{i=0}^\infty \pi_i=1.
\end{aligned}
\end{equation}
The optimal solution of \eqref{eq:one point lower bound} achieved at 
\[
\hat{\pi_i} = \begin{cases}
\sqrt{\frac{2\varepsilon}{-F''(1)}} + o(\sqrt{\varepsilon}), &\quad i\leq B,\\
0, &\quad i\geq B+2
\end{cases}
\]
with $\hat{\pi}_{B+1}$ chosen s.t. $\sum_{i=0}^\infty \hat{\pi}_i = 1$. Also, $B$ is the smallest integer s.t. $(B+2)\left(\sqrt{\frac{2\varepsilon}{-F''(1)}} + o(\sqrt{\varepsilon})\right) \geq 1$. Thus, we get $B \geq \sqrt{\frac{-F''(1)}{2\varepsilon}} + o(1/\sqrt{\varepsilon})$. Thus, we have
\begin{align*}
    q^\star&\geq \sum_{i=0}^\infty\hat{\pi_i} \geq \left(\sqrt{\frac{2\varepsilon}{-F''(1)}} +o(\sqrt{\varepsilon})\right) \sum_{i=0}^B i = \left(\sqrt{\frac{2\varepsilon}{-F''(1)}} +o(\sqrt{\varepsilon})\right) \frac{B(B+1)}{2} \\
    &\geq \frac{1}{2}\left(\sqrt{\frac{2\varepsilon}{-F''(1)}} +o(\sqrt{\varepsilon})\right)\left(\sqrt{\frac{-F''(1)}{2\varepsilon}} + o\left(\frac{1}{\sqrt{\varepsilon}}\right)\right)^2 = \sqrt{\frac{-F''(1)}{8\varepsilon}} + o\left(\frac{1}{\sqrt{\varepsilon}}\right)
\end{align*}
Hence $q^\star = \Omega\left(\frac{1}{\sqrt{\varepsilon}}\right)$ for all control policies with strictly concave reward function $F$.
\end{proof}

\textbf{Remark (Connection to $f-$divergence):} 
One can view the regret $F(1) - \mathbb{E}_\pi[F(\lambda(\bar{q}))]$ as an $f-$divergence for an appropriately defined $f$. In particular, one can set $\mathbb{E}_\pi[F(\lambda(\bar{q}))] = \sum \pi_i F(\pi_{i+1}/\pi_i)$, which looks like an $f-$ divergence between $\pi$ and its shifted version. However, the shifted version needs to be normalized for it to be a probability measure. To do this, define
\begin{align*}
    f_a(x) = F(a) + a F^\prime(a)(x-1) - F(ax), \quad a = 1-\pi_0
\end{align*}
and note that 
\begin{align*}
    \varepsilon\geq F(1) - \mathbb{E}[F(\lambda(\bar{q}))] = F(1)-F(a) + \sum \pi_i f_a\left(\frac{\pi_{i+1}}{a\pi_i}\right) \geq \sum \pi_i f_a\left(\frac{\pi_{i+1}}{a\pi_i}\right),
\end{align*}
where $\sum \pi_i f\left(\frac{\pi_{i+1}}{a\pi_i}\right)$ is now an $f-$ divergence between $\pi$ and its shifted and appropriately normalized version. Then, one can use the generalization of Pinsker's known as Csiszár-Kullback-Kraft Inequality to lower bound $f$ divergence by the TV distance to get
\begin{align*}
    \varepsilon \geq \sum \pi_i f_a\left(\frac{\pi_{i+1}}{a\pi_i}\right) \geq   4c_f\text{TV}\left(\pi_{i+1}/a, \pi_i\right)^2,
\end{align*}
where \(c_f = \frac{f''_a(1)}{2} \approx -\frac{F''(1)}{2}\) \cite{gilardoni2010pinsker}. Now, we know that for any set $A$, TV\(\left(\pi_{i+1}/a, \pi_i\right) \geq \left|\sum_{i\in A}\pi_{i+1}/a-\pi_i\right|\). Picking $A = \{0, ..., j\}$ gets $\pi_{j+1}\leq \sqrt{\frac{\varepsilon}{-2F''(1)}}$, which is of the same order with a slightly better constant than the upper bound derived in \eqref{appendix_pi_0} in the proof above. 
However, the difficulty is that the induced function $f_a$ need not be globally strongly convex, since $F$ may be linear on one side of \(1\), as allowed by our class of concave-like functions discussed in Appendix~\ref{concavelike_nonconcave}.
Nonetheless, at heart, our proof is essentially a by-product of a Pinsker's-type inequality.

\begin{wrapfigure}{r}{0.6\textwidth}
    \centering
    \includegraphics[width=0.8\linewidth]{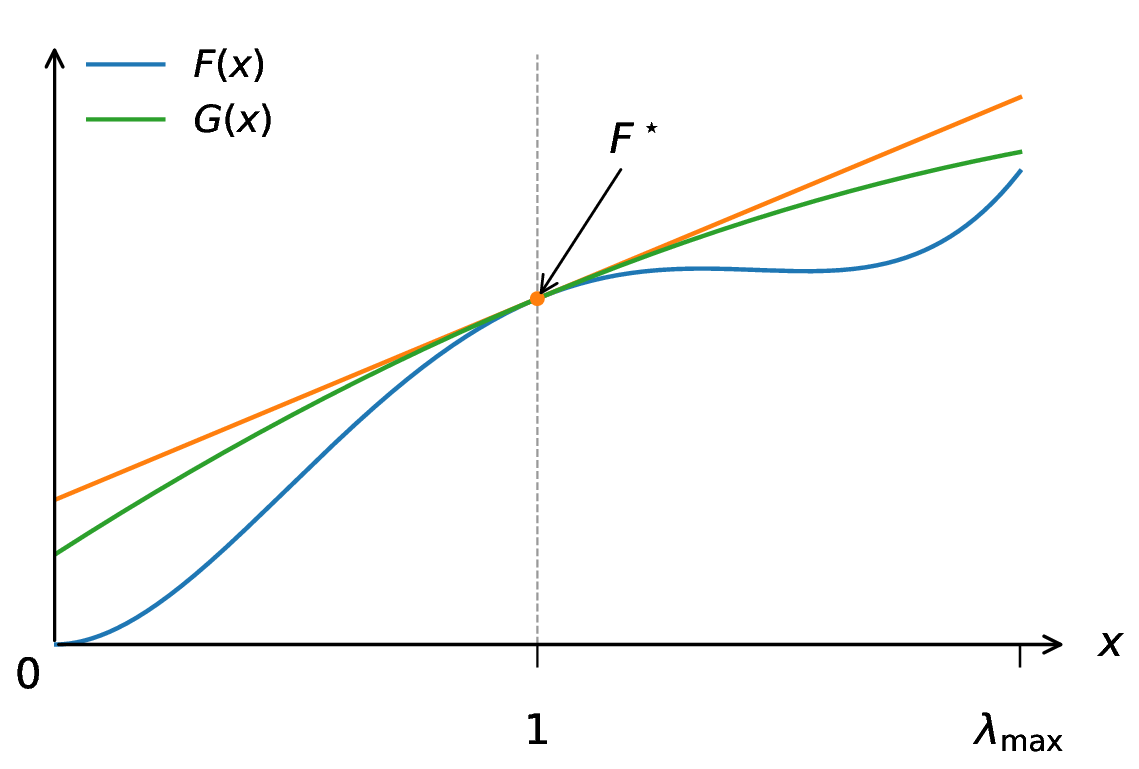}
    \caption{Construction of a quadratic function $G$}
    \label{fig:quadratic_main_part}
\end{wrapfigure}

Next, we generalize the \(\Omega(1/\sqrt{\varepsilon})\) lower bound from strictly concave reward functions to general functions \(F\) satisfying \eqref{eq:concave_like}. The main idea is to construct a strictly concave function \(G\) that upper bounds \(F\) while lying below the tangent line at \((1, F(1))\), as illustrated in Fig.~\ref{fig:quadratic_main_part}. A complete proof appears in Appendix~\ref{concavelike_nonconcave}.

\subsection{Two-arrival Policy}\label{sec:additional_log}
The below Theorem~\ref{thm4.13} establishes that, in many settings where the strictly concave function $(F)$ is in fact strongly concave, then any two-arrival policy must satisfy \(\mathbb{E}[\bar q]\geq \Omega(\sqrt{\frac{\log1/\varepsilon}{\varepsilon}})\), thereby incurring an additional $\sqrt{\log1/\varepsilon}$ factor relative to the universal lower bound $\Omega(1/\sqrt{\varepsilon})$ in Proposition~\ref{prop:strictly_concave}.

\begin{theorem}\label{thm4.13}
Under Assumption~\ref{assump2}, \ref{assump4}, \ref{assump F'(1)>0} and the first item of Assumption~\ref{eq:smooth_F}, if $F$ is {\bf strongly concave} and $\lambda_{\max}>1$, then there exists $\varepsilon_0>0$(depending on $F$), such that $q^\star_{TA} = \min\limits_{\lambda\in \Lambda_2} \mathbb{ E}_\pi[\bar q] = \Omega(\sqrt{\frac{\log
{1/\varepsilon}}{\varepsilon}})$ for all $\varepsilon\leq \varepsilon_0$,
where \(\Lambda_2\) is the class of all two-arrival policies.
\end{theorem}

The strong concavity assumption is imposed mainly for technical convenience. Since \(F''(1)<0\), under the assumption of continuous \(F''\), any strictly concave reward function is strongly concave on a neighborhood of 1. Intuitively, under near-optimal policies with small regret \(\varepsilon\), the arrival rates \(\lambda(q)\) must remain close to the fluid-optimal value 1, as larger deviations would incur significant reward loss for concave $(F)$. Hence, on the relevant interval \([1-k_2,1+k_1]\) , strict concavity already implies strong concavity when \(k_1,k_2\) are small. Extending the result to all strictly concave functions is therefore possible but omitted here to keep the proof concise. The proof is provided in Appendix ~\ref{proof of simulation in appendix}.

Next, consider the two-arrival policy with \(\lambda(q)=1+k_1\) when \(q<\tau\), and \(\lambda(q) = 1-k_2\) when \(q\geq \tau\), with parameters chosen as 
\begin{equation}\label{simulation_2arrival}
k_1 = \sqrt{\frac{\varepsilon}{-F''(1)}}\sqrt{\log\frac{1}{\varepsilon}}, k_2 = \sqrt{\frac{\varepsilon}{-F''(1)}}\frac{1}{\sqrt{\log\frac{1}{\varepsilon}}}, \text{ and } \tau = \left\lceil \frac{1}{2}\sqrt{\frac{-F''(1)}{\varepsilon}}\sqrt{\log\frac{1}{\varepsilon}} \right\rceil.
\end{equation}
Under this choice of parameters, one can obtain the following upper bound on the optimal expected steady-state queue length \(q^\star \leq \frac{3}{2}\sqrt{\frac{-F''(1)}{\varepsilon}}\sqrt{\log\frac{1}{\varepsilon}}\left(1+o(1)\right),\)
which matches the order of the lower bound established earlier. 
The verification that the resulting policy \eqref{simulation_2arrival} satisfies the regret constraint \(R(\lambda)\leq \varepsilon\) follows similar arguments as in Section~\ref{sec:qstar_upper_bound}, relying on a third-order Taylor expansion, and is therefore omitted.


\section{Discussion}
\subsection{Market Size-Throughput Trade-off under the Fully Dynamic Policy}\label{alpha1/4}

\begin{wrapfigure}{r}{0.6\textwidth}
\centering
\begin{tikzpicture}
\begin{axis}[
    width=0.85\linewidth,
    height=0.55\linewidth,
    axis lines=left,
    xlabel={Market size $\bar{\lambda}_{\max}$},
    ylabel={Throughput},
    xmin=0.95, xmax=2.6,
    ymin=0.72, ymax=1.03,
    xtick={1.15,1.40,2.30},
    xticklabels={
        {$1+\epsilon^{1/4}$},
        {},
        {fixed $\bar{\lambda}_{\max}>1$}
    },
    ytick={0.80,0.96,1},
    yticklabels={
        {$1-\Theta(\epsilon)$},
        {$1-\Theta(\epsilon^{3/2})$},
        {$1$}
    },
    tick label style={font=\small},
    label style={font=\small},
    clip=false,
]

\draw[thick]
    (axis cs:1.15,0.80)
    .. controls (axis cs:1.28,0.93)
                and (axis cs:1.75,0.96) ..
    (axis cs:2.30,0.96);

    \addplot[
        dashed,
        domain=0.98:1.15
    ]
    {0.80};

    \addplot[
        dashed,
        domain=0.98:2.30
    ]
    {0.96};

    \draw[dashed]
        (axis cs:1.15,0.72) --
        (axis cs:1.15,0.80);

    \draw[dashed]
        (axis cs:2.30,0.72) --
        (axis cs:2.30,0.96);

    \addplot[
        only marks,
        mark=*,
        mark size=2pt
    ]
    coordinates {(1.15,0.80) (2.30,0.96)};

\end{axis}
\end{tikzpicture}
\caption{Market size--throughput trade-off.
}
\label{fig:market_size_throughput_tradeoff}
\end{wrapfigure}

Consider the class of fully dynamic policies defined in \eqref{symmetry_dynamic} and parameterized by
\(\bar{\lambda}_{\max}\in [1+\varepsilon^\alpha,\lambda_{\max}]\), where \(0<\alpha<1/4\) and
\(m=\left\lceil 1/(\sqrt{\bar{\lambda}_{\max}}-1)\right\rceil-1\).
For any choice of \(\bar{\lambda}_{\max}\) in this range, Section~\ref{sec:qstar_upper_bound} shows that the fully dynamic policy achieves a reward of at least \(F(1)-\varepsilon\) and an average queue length of at most \(\sqrt{-7F''(1)}/\sqrt{\varepsilon}\).
Thus, changing \(\bar{\lambda}_{\max}\) preserves the same asymptotic efficiency--reward trade-off between reward and congestion.
It is then natural to ask how \(\bar{\lambda}_{\max}\) should be chosen in practice.

This choice can be viewed as a trade-off between market size and system throughput.
Under \eqref{symmetry_dynamic}, the maximum arrival rate is attained at \(q=0\), with
\(\lambda(0)=((m+2)/(m+1))^2\approx \bar{\lambda}_{\max}\), so \(\bar{\lambda}_{\max}\) controls the maximum market size allowed by the policy.
Meanwhile, using \eqref{plug_pi_0}, together with
\(m=\Theta(1/(\bar{\lambda}_{\max}-1))\) and \(B=\Theta(\varepsilon^{-1/2})\), by Equation \eqref{idle time} in Appendix A, the throughput satisfies
\begin{equation}
\mathbb{E}_{\pi}[\lambda(q)]
=
1-\pi_0
=
1-\Theta\left(\frac{(m+1)^2}{Bm^2+B^2m+B^3}\right)
=
1-\Theta\left(\frac{\varepsilon^{3/2}}{(\bar{\lambda}_{\max}-1)^2}\right).
\end{equation}
Hence, increasing \(\bar{\lambda}_{\max}\) increases the system throughput while preserving both the reward guarantee and the \(O(1/\sqrt{\varepsilon})\) average queue-length bound.
The two extremes are as follows: when \(\bar{\lambda}_{\max}=1+\varepsilon^\alpha\) with \(\alpha\approx 1/4\), the throughput is \(1-\Theta(\varepsilon)\), whereas for a fixed \(\bar{\lambda}_{\max}>1\), it is \(1-\Theta(\varepsilon^{3/2})\).
We illustrate this market size--throughput trade-off in Fig.~\ref{fig:market_size_throughput_tradeoff}.

\subsection{Price Stability under Dynamic Arrival Control}

In revenue management settings, the arrival rate is often controlled indirectly through pricing, where different prices induce different customer arrival rates. Therefore, in addition to congestion performance, the variability of prices faced by customers is also a relevant performance metric. Since the fully dynamic policy employs a much larger set of arrival rates than the two-arrival policy, one may expect its improved congestion performance to come at the cost of substantially larger price fluctuations. Surprisingly, this is not the case. Under mild regularity conditions, we show that the variance of prices under the two policies has the same asymptotic order.

Recall from \textbf{Case 3 (Revenue Maximization)} in Section~\ref{model} that the reward (revenue) function is defined as the product of the price and arrival rate,
so that
\(
p(\lambda)=\frac{F(\lambda)}{\lambda}.
\)
Suppose that \(p(\lambda)\) is twice continuously differentiable and that \(p''(\lambda)\) is uniformly bounded over the relevant range of arrival rates. A second-order Taylor expansion around \(\lambda=1\) yields
\(
p(\lambda)
\approx
p(1)+p'(1)(\lambda-1).
\)
Under some regularity conditions ensuring that the Taylor remainder is negligible relative to the leading-order terms, which can be verified for both policies, we obtain
\begin{align}
\operatorname{Var}(p(\lambda))
&\approx
\left(p'(1)\right)^2
\operatorname{Var}(\lambda).
\label{eq:price_variance}
\end{align}
Thus, it suffices to compare the arrival-rate variance under the two policies.

By rate balance (see Equation \eqref{idle time} in Appendix A for precise calculations), we have
\(\mathbb{E}[\lambda]=1-\pi_0\), where \(\pi_0\) is the stationary probability of an empty system. Therefore,
\begin{align}
\operatorname{Var}(\lambda)
=\mathbb{E}\!\left[(\lambda-1)^2\right]-\left(\mathbb{E}[\lambda]-1\right)^2
=\mathbb{E}\!\left[(\lambda-1)^2\right]-\pi_0^2.
\label{eq:lambda_variance_decomposition}
\end{align}
So now we focus on calculating  \(\mathbb{E}[(\lambda-1)^2]\) and \(\pi_0^2\) for the two arrival policy \eqref{simulation_2arrival} and the fully dynamic policy \eqref{symmetry_dynamic} below. For the two-arrival policy, we have (by \eqref{pi_0_twoarrive} in Appendix~\ref{proof of simulation in appendix})
\[
\pi_0
=
\left(
\frac{(1+k_1)^\tau-1}{k_1}
+
\frac{(1+k_1)^\tau}{k_2}
\right)^{-1}
=
\Theta\!\left(
\frac{\varepsilon}
{\sqrt{\log(1/\varepsilon)}}
\right),
\]
and hence
\(\pi_0^2=o(\varepsilon)\).
Moreover (by \eqref{pi_q_less_tau} in Appendix~\ref{proof of simulation in appendix}),
\begin{align}
\mathbb{E}\!\left[(\lambda-1)^2\right]
&=
k_1^2\pi_{q<\tau}
+
k_2^2\pi_{q\geq\tau}
=
k_1k_2+\pi_0(k_2-k_1)
=
\Theta(\varepsilon),
\label{eq:twoarrival_secondmoment}
\end{align}
where the last equality follows from
\(k_1k_2=\Theta(\varepsilon)\) and
\(\pi_0(k_2-k_1)=o(\varepsilon)\).

For the fully dynamic policy in
\eqref{symmetry_dynamic}, we present the calculation only for the case \(m=0\) for simplicity. ($m =0$ has most change in the arrival rate) By the result in the previous section ~\ref{alpha1/4},
\(\pi_0=\Theta(B^{-3})=\Theta(\varepsilon^{3/2})\), so that
\(\pi_0^2=o(\varepsilon)\). Then
\begin{align}
\mathbb{E}\!\left[(\lambda-1)^2\right]
&=
\sum_{q=0}^{B-1}
\left(
8+\frac{2}{(q+1)^2}
\right)\pi_0
+
\left(
4-\frac{4}{B+1}
+\frac{1}{(B+1)^2}
\right)\pi_0
\notag\\
&=
(8B+O(1))\pi_0
=
\Theta\!\left(\frac1{B^2}\right)
=
\Theta(\varepsilon),
\label{eq:dynamic_secondmoment}
\end{align}
where we use
\(
\pi_0=\Theta(B^{-3})
\text{ and }
B=\Theta(\varepsilon^{-1/2}).
\)

Therefore, since \(\pi_0^2=o(\varepsilon)\) under both policies,
\eqref{eq:twoarrival_secondmoment} and
\eqref{eq:dynamic_secondmoment}
imply
\(
\operatorname{Var}(\lambda)
=
\Theta(\varepsilon)
\)
for both the two-arrival and fully dynamic policies.
Substituting this into \eqref{eq:price_variance} yields
\(
\operatorname{Var}(p(\lambda))
=
\Theta(\varepsilon)
\)
under both policies. Thus, the variance of prices under the fully dynamic policy has the same asymptotic order as the two arrival policy.


\subsection{Dynamic Pricing with Delay Sensitive Customers}\label{value of DP in large queue system}
The current state-of-the-art for dynamic pricing in a single server queue is \cite{kim2018value}, with the focus on maximizing the long-run profit. Their model can be translated to our more general setting as follows: They consider an $M_q/M/1$ queue with service rate of $n$ and a maximum arrival rate $n \lambda_{\max}$. Given a queue length $q$, the system operator sets a price of $p(q)$. Customers respond to both the posted price and the current waiting time: an incoming customer joins with probability $\bar{G} \left(p(q) + \frac{hq}{n}\right)$, where $h > 0$ is a constant, $G$ is the cumulative distribution function of customer valuations with a non-decreasing hazard rate, and $\bar{G}(\cdot) = 1-G(\cdot)$. Under this formulation, the arrival rate is given by
\begin{align*}
    \lambda(q) = n \lambda_{\max} \bar{G} \left(p(q) + \frac{hq}{n}\right) \implies p(q) = \bar{G}^{-1} \left(\frac{\lambda(q)}{n\lambda_{\max}}\right) - \frac{hq}{n}.
\end{align*}
The revenue/profit $(P)$ in steady-state is then 
\begin{align*}
    P_n = \mathbb{E}\left[p(\bar{q})\lambda(\bar{q})\right] = \mathbb{E}\left[\lambda(\bar{q})\bar{G}^{-1} \left(\frac{\lambda(\bar{q})}{n\lambda_{\max}}\right)\right] - \mathbb{E}\left[\frac{h\bar{q}\lambda(\bar{q})}{n}\right] := n\mathbb{E}\left[F\left(\frac{\lambda(\bar{q})}{n}\right)\right] - \mathbb{E}\left[\frac{h\bar{q}\lambda(\bar{q})}{n}\right],
\end{align*}
where we define $F(\lambda) = \lambda \bar{G}^{-1}(\lambda/\lambda_{\max})$, which coincides with the reward function in our model. Assuming $G$ to have a non-decreasing hazard rate ensures $F$ is a concave function. Moreover, \cite{kim2018value} defines the fluid profit $P^\star_n$ as $\max\limits_{\lambda \leq n} nF(\lambda/n) = nF(1)$, which is used to define the revenue loss as follows:
\begin{align*}
    P^\star_n - P_n = n\left(F(1) - \mathbb{E}\left[F\left(\frac{\lambda(\bar{q})}{n}\right)\right]\right) + \mathbb{E}\left[\frac{h\bar{q}\lambda(\bar{q})}{n}\right].
\end{align*}
Equivalently, one may consider a rescaled \(M_q/M/1\) system with unit service rate and arrival rates divided by \(n\). Under this scaling, the stationary distribution of the underlying CTMC remains unchanged, and thus \(\mathbb{E}[\bar q]\) is invariant with respect to \(n\). In our terminology, \(\mathbb{E}\left[F\left(\frac{\lambda(\bar{q})}{n}\right)\right]\) corresponds to the long-run average reward. Therefore, the above is equivalent to \(P^\star_n - P_n = n R(\lambda) + h\mathbb{E}\left[\bar{q}\lambda(\bar{q})\right].\)
Under the fully dynamic policy \eqref{symmetry_dynamic}, we have shown that the regret $R(\lambda) \leq \varepsilon$, and $\mathbb{E}[\bar{q}] \leq C/\sqrt{\varepsilon}$, for some $C>0$. It follows that \(P^\star_n - P_n \leq n \varepsilon + \frac{4hC}{\sqrt{\varepsilon}}.\)
Choosing $\varepsilon = (4hC/n)^{2/3}$, we get
\begin{align}
    P^\star_n - P_n \leq 2 (4hC)^{2/3} n^{1/3}. \label{eq: n_one_third}
\end{align}
Thus, our proposed fully dynamic policy achieves the optimal revenue loss of $O(n^{1/3})$, while a two-price policy of \cite{kim2018value} incurs a revenue loss of $O((n \log n)^{1/3})$. This result is in-line with the observation in \cite{kim2018value} that the additional logarithmic term appears due to the discontinuity of the
two-price policy. Our suggested dynamic policy overcomes this issue by varying the price/arrival-rate more gradually with the queue lengths.


This phenomenon parallels our earlier observation that two-support arrival policies with regret at most \(\varepsilon\) cannot reduce the expected queue length below \(\Omega(\sqrt{\log (1/\varepsilon)/\varepsilon})\), whereas fully dynamic policies achieve \(\Theta(1/\sqrt{\varepsilon})\). Similarly, \cite{kim2018value} show that while a broad class of dynamic pricing policies must incur a revenue loss of at least \(\Omega(n^{1/3})\), their two-price policy is suboptimal by a logarithmic factor.

Finally, Theorem \ref{general:1 support} implies that the expected queue length is at least $\Omega(1/\sqrt{\varepsilon})$ under any arbitrary dynamic policy with regret of at most $\varepsilon$. This result immediately implies the bounds in \eqref{eq: n_one_third} holds in the other direction as well (albeit with a different constant). Thus, we conclude that the revenue loss $(P^\star_n-P_n)\geq \Omega(n^{1/3})$ for an arbitrary admissible policy, with no structural assumptions on how the arrival rate varies as a function of the queue length. This result considerably generalizes the lower bound presented in \cite{kim2018value}, which is restricted to policies that are monotonic, and the arrival rates are a small perturbation of a static arrival rate, i.e., $|\lambda(q)-\lambda^\star| \leq \delta$ for some $\lambda^\star > 0$ and sufficiently small $\delta > 0$.

\section{Numerical Experiments}

\subsection{Fully dynamic policies versus two-arrival policy when \(F\) is strictly concave}\label{sec:numerical}
In this section, we present numerical experiments to compare the performance of asymptotically optimal two-arrival policies, fully dynamic arrival policies, and the optimal state-dependent policy obtained from the corresponding average-reward Markov decision process (MDP). The MDP solution serves as a numerically obtained optimal benchmark,
allowing us to quantify the optimality gap of the proposed
analytical policies; see Appendix~\ref{app:mdp} for details
of the numerical MDP formulation and solution procedure.

We emphasize that \(\varepsilon\) is used as a {\bf tuning parameter} to generate candidate control policies, rather than the theoretical regret. Specifically, we select a grid of tuning parameters $\varepsilon \in \{0.001, 0.004, 0.007, 0.01, 0.025,\allowbreak
0.04, 0.055, 0.07, 0.085, 0.1\}$. For each \(\varepsilon\), we instantiate the two-arrival and fully dynamic policies as following:

{\bf Two-arrival policy.}
For the two-arrival policy, we implement the asymptotically optimal policy introduced in \eqref{simulation_2arrival}. 

{\bf Fully dynamic arrival policies.}
To illustrate the performance of fully dynamic arrival control, we consider the following class of state-dependent arrival policies parameterized by \(k>1\):
\[
\lambda(q) = \begin{cases}
    \left(\frac{q+2}{q+1}\right)^k, \quad & q = 0,1,\cdots, B-1\\
    \left(\frac{2B-q}{2B-q+1}\right)^k, \quad & q = B, B+1, \cdots, 2B-1\\
    0, \quad &q\geq 2B
\end{cases}\quad , B = \left\lceil \sqrt{\frac{-F''(1)}{\varepsilon}(\frac{k^2(k+1)}{2(k-1)}+1)}\right\rceil.
\]

Note that we set \(m=0\) and allowed the exponent in \eqref{symmetry_dynamic} to vary over \(k>1\). For each fixed \(k>1\), the choice of \(B\) follows from the same upper bound analysis used for \(k=2\). Although the proof in Section~\ref{sec:qstar_upper_bound} is carried out explicitly for \(k=2\) to keep the analysis concise, the same construction and upper-bound scaling apply to all \(k>1\). In our numerical experiments, we consider \(k=1.2, \frac{\sqrt{5}+1}{2}, 2\) to illustrate how the choice of \(k\) affects performance in practice.

For each policy, we compute $R(\lambda)$ as defined in \eqref{eq:constraint} and plot the normalized regret ratio $R(\lambda) / F(1)$ on the horizontal axis to remove scaling effects caused by different magnitudes of \(F(1)\) under different reward functions. The vertical axis reports the simulated steady-state expected queue length. Therefore, each curve represents the efficiency-reward trade-off induced by the corresponding policy. For comparison, we also solve the corresponding average-reward MDP numerically using policy iteration.


We report simulation results for two representative reward functions: a concave quadratic reward \(F(x)=5x-x^2\), and a strictly concave but non-quadratic reward \(F(x)=\sqrt{x}\). The resulting trade-off curves are shown in Fig.~\ref{fig:Simulation for different F}, where we compare the two-arrival policy, fully dynamic policies, and the numerically optimal MDP benchmark. 

\begin{figure}[htbp]
  \centering
  \begin{subfigure}[t]{0.48\textwidth}
    \centering
    \includegraphics[width=\linewidth]{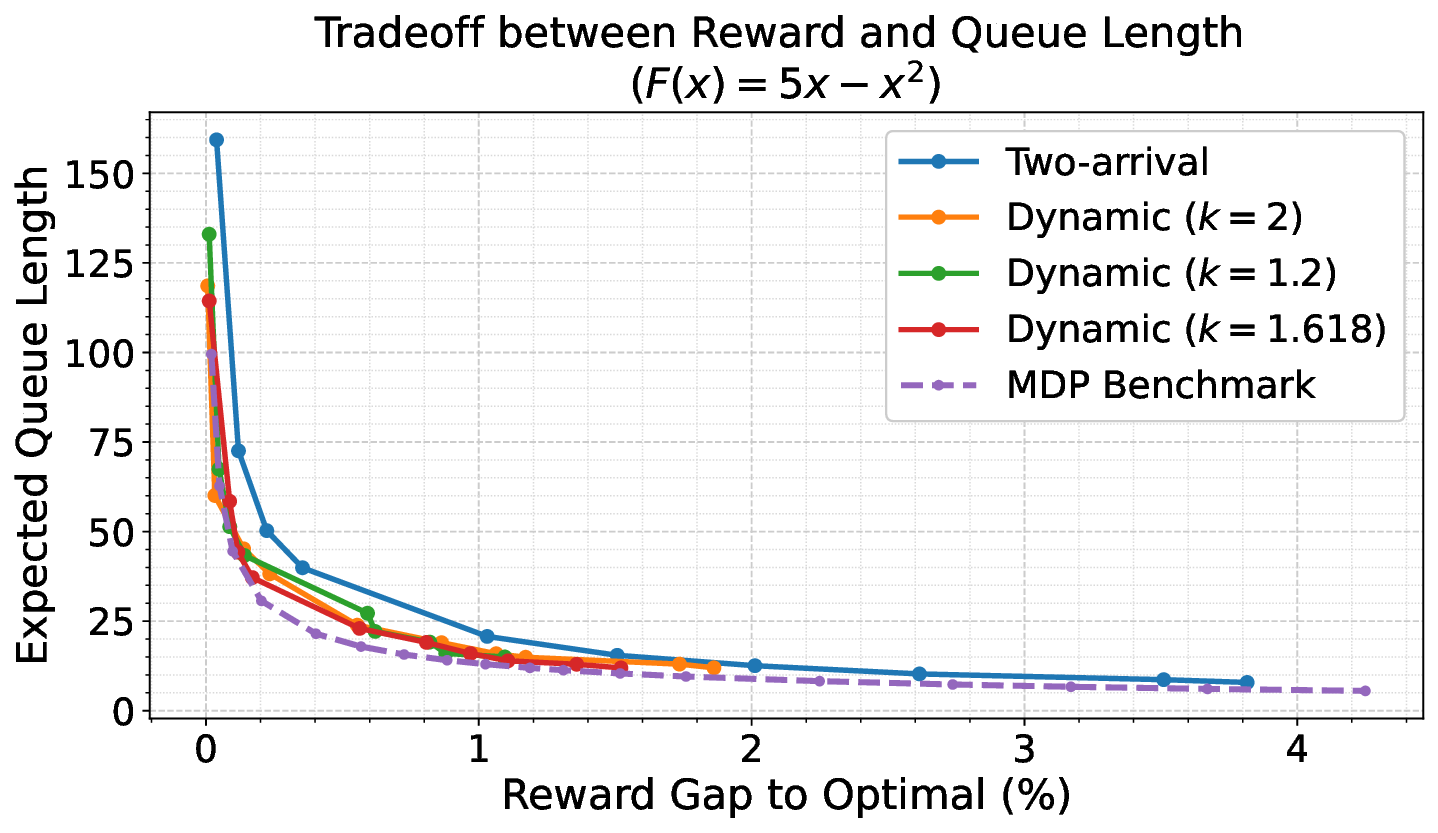}
    \subcaption{$F(x)=-x^2+5x$}
    \label{fig:when F is quadratic}
  \end{subfigure}
  \begin{subfigure}[t]{0.48\textwidth}
    \centering
    \includegraphics[width=\linewidth]{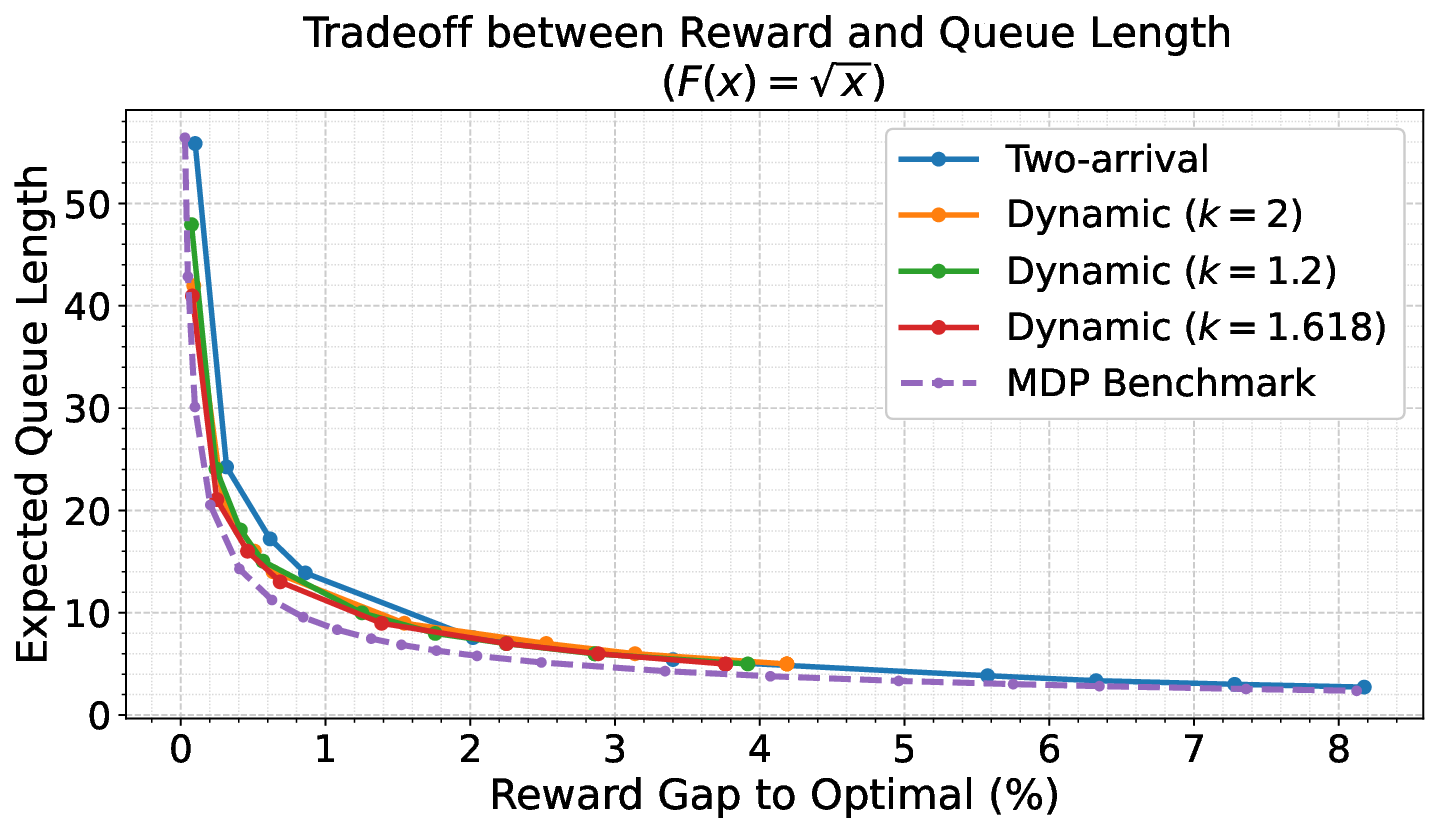}
    \subcaption{$F(x)=\sqrt{x}$}
    \label{fig:When F is square root}
  \end{subfigure}
  \caption{Tradeoff comparison among two-arrival policies, fully dynamic policies, and the optimal MDP benchmark.}
  \label{fig:Simulation for different F}
\end{figure}

The numerical results strongly support our theoretical findings. For both reward functions, fully dynamic policies consistently achieves lower expected queue lengths than the two-arrival policy while achieving the same regret. Moreover, the performance differences among fully dynamic policies with different values of \(k>1\) are relatively modest, although \(k = \frac{\sqrt{5}+1}{2}\) yields slightly better performance. This observation supports the analytically convenient choice \(k=2\) adopted in our theoretical analysis. 
Furthermore, the fully dynamic policies seems to be closer to the MDP optimal than the two-arrival policy, 
indicating that our proposed fully dynamic policies capture most of the efficiency gains attainable by unrestricted state-dependent control while remaining analytically tractable. This advantage is particularly pronounced when regret is small. For example, in Fig.~\ref{fig:when F is quadratic}, at $R(\lambda)/F(1) \approx 0.13\%$, the expected queue length reduces from $75$ for two-arrival policy to around $45$ for fully dynamic policies, while the MDP optimal is $40$.



We note that \(F(x) = \sqrt{x}\) does not satisfy the bounded third-derivative condition in Assumption~\ref{eq:smooth_F}, yet, as observed in Fig.~\ref{fig:When F is square root}, the fully dynamic policy continues to perform better than the two-arrival policy and exhibits small gap compared to the MDP optimal. This suggests that the qualitative efficiency-reward trade-offs identified in this paper are
robust beyond the regularity conditions imposed for analytical tractability.

{
\subsection{Beyond exponential service-time distributions}\label{sec:pareto}

In this section, we provide additional experiments to examine whether the fully dynamic policy continues to outperform the simpler two-arrival policy in general \(M/G/1\) setting. In particular, we consider two Pareto service-time distributions with same mean \(\mu = 1\), and different shape parameters \(\alpha = 2.5\) and \(\alpha = 2.1\), respectively. The service time density is
\(
f(x) = \frac{\alpha x_m^\alpha}{x^{\alpha+1}}, \quad x \ge x_m,
\)
where \(x_m = (\alpha - 1)/\alpha\) is chosen so that \(\mathbb{E}[S]=1\).

Both distributions have finite variance with
\(
\mathrm{Var}(S) = \frac{1}{\alpha(\alpha-2)},
\)
In particular, the variances are \(0.8\) and \(4.76\) respectively. A smaller value of \(\alpha\) corresponds to a heavier tail, meaning that large service times occur more frequently.


\begin{figure}[htbp]
\centering

\begin{subfigure}[t]{0.48\textwidth}
    \centering
    \includegraphics[width=\linewidth]{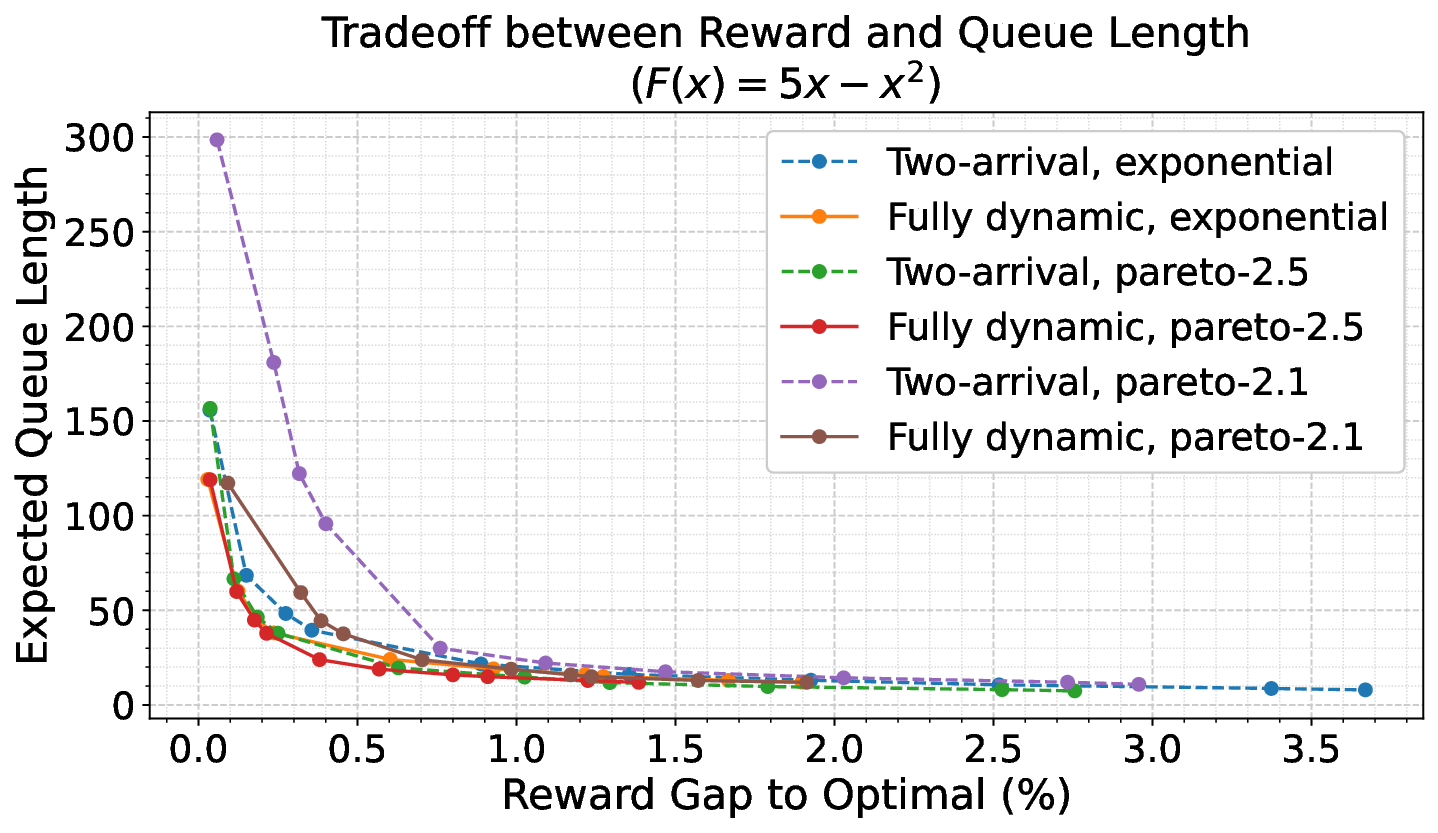}
    \caption{$F(x)=5x-x^2$}
    \label{fig:MG1_quadratic}
\end{subfigure}
\hfill
\begin{subfigure}[t]{0.48\textwidth}
    \centering
    \includegraphics[width=\linewidth]{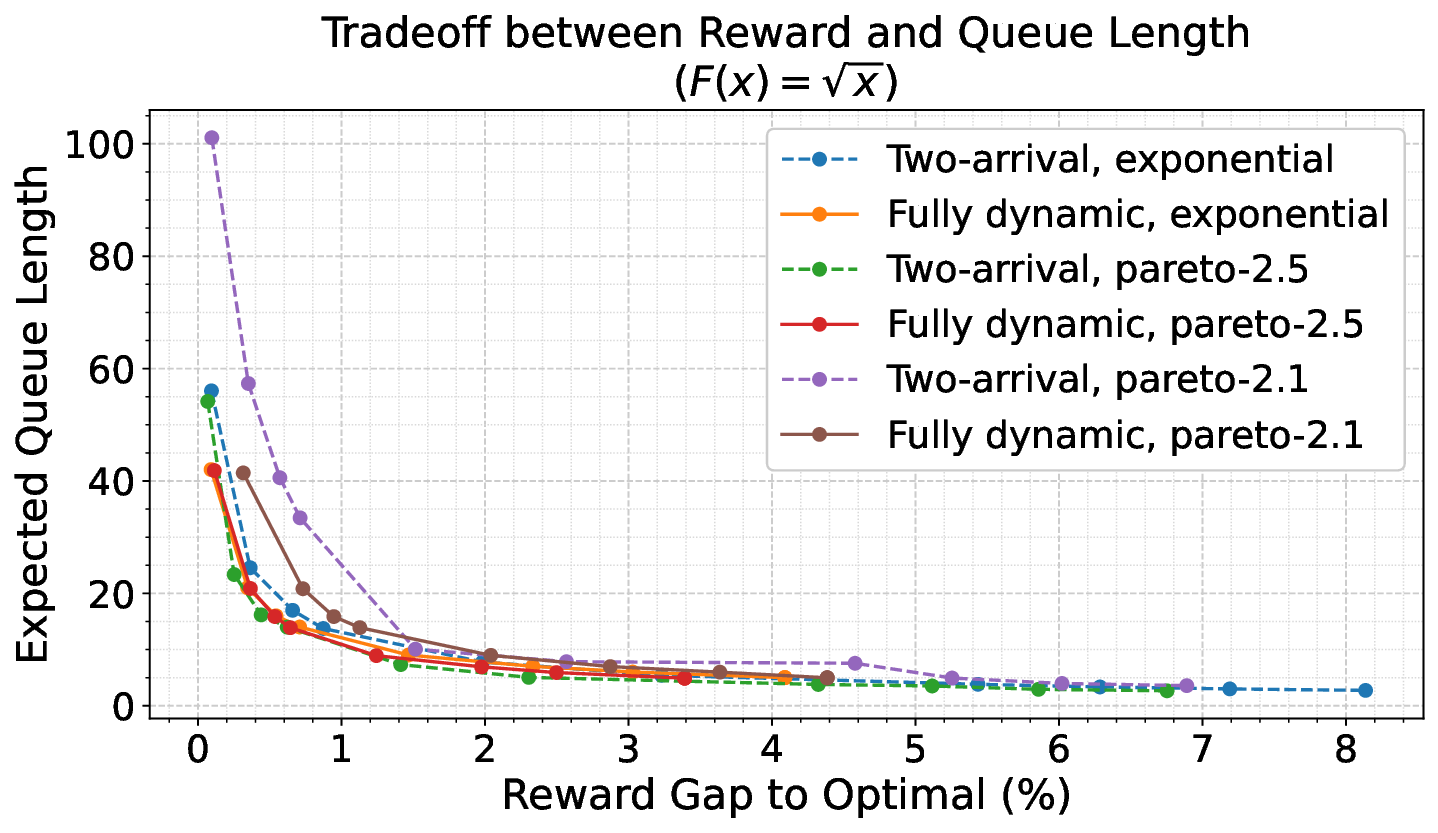}
    \caption{$F(x)=\sqrt{x}$}
    \label{fig:MG1_sqrt}
\end{subfigure}

\caption{
Comparison of two-arrival and fully dynamic policies under different service-time distributions.
}
\label{fig:MG1_robustness}

\end{figure}

From Fig.~\ref{fig:MG1_robustness}, the fully dynamic policy consistently outperforms the two-arrival policy across both reward functions and all three service-time distributions. The performance gap becomes substantially more pronounced for the heavier-tailed Pareto distribution with $\alpha=2.1$. Intuitively, this robustness arises from the structure of the fully dynamic policy, which induces a strong positive drift when the queue length is below a threshold $B$, and a strong negative drift when it exceeds $B$, together with an effective upper buffer around $2B$. This drift structure mitigates the impact of large service-time fluctuations and keeps the queue length more tightly concentrated around $B$.



For further comparison, we introduce a three-arrival policy as a robust version of the two-arrival policy. Specifically, this policy coincides with the two-arrival policy when $q< 6\tau$, but sets the arrival rate to zero when $q \ge 6\tau$. This modification eliminates the buffering effect present in the fully dynamic policy, and allows us to examine whether the fully dynamic policy continues to provide additional benefits beyond such threshold-based controls.

\begin{figure}[htbp]
\centering

\begin{subfigure}[t]{0.48\textwidth}
    \centering
    \includegraphics[width=\linewidth]{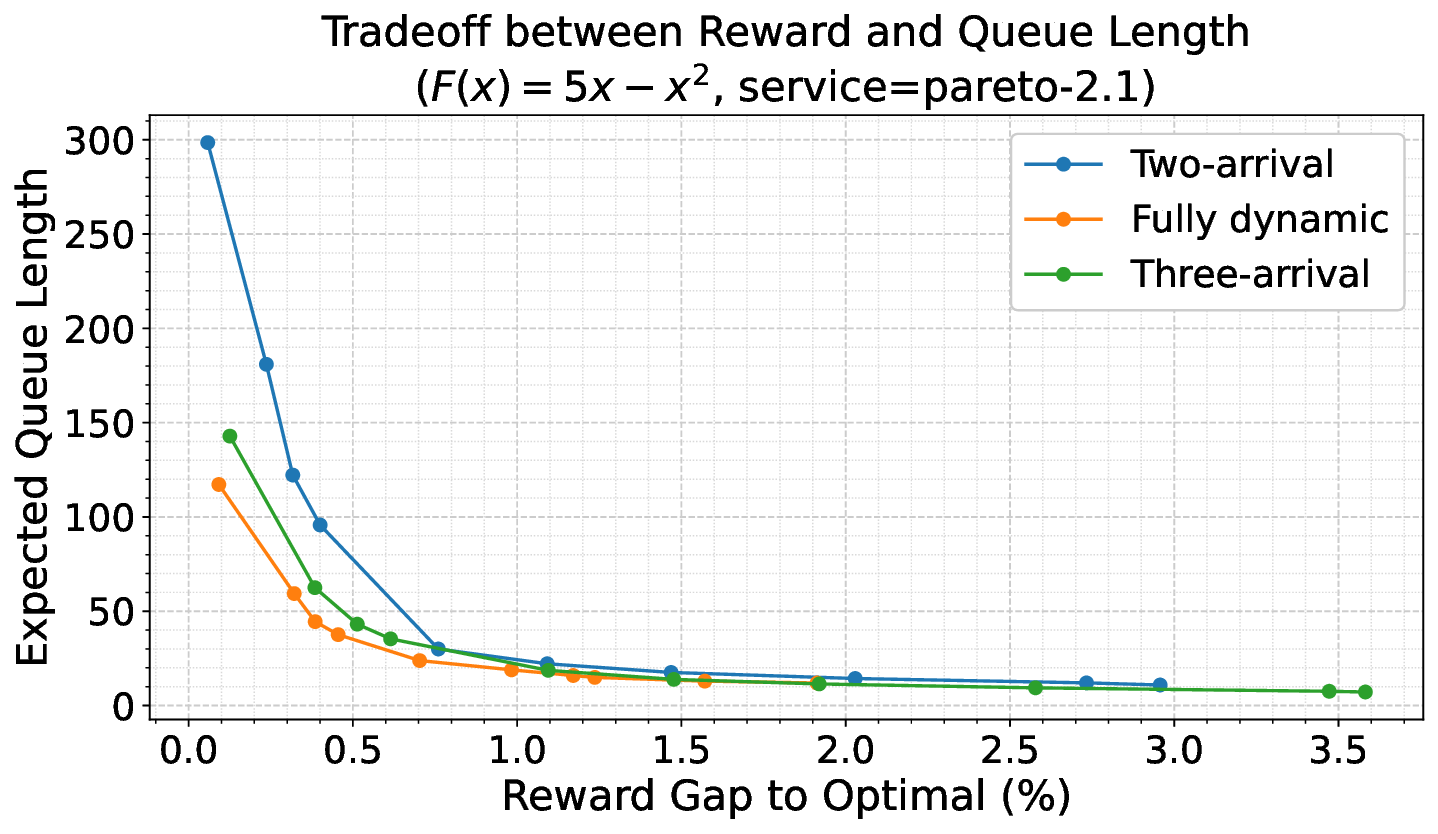}
    \caption{$F(x)=5x-x^2$}
    \label{fig:three_policy_quadratic}
\end{subfigure}
\hfill
\begin{subfigure}[t]{0.48\textwidth}
    \centering
    \includegraphics[width=\linewidth]{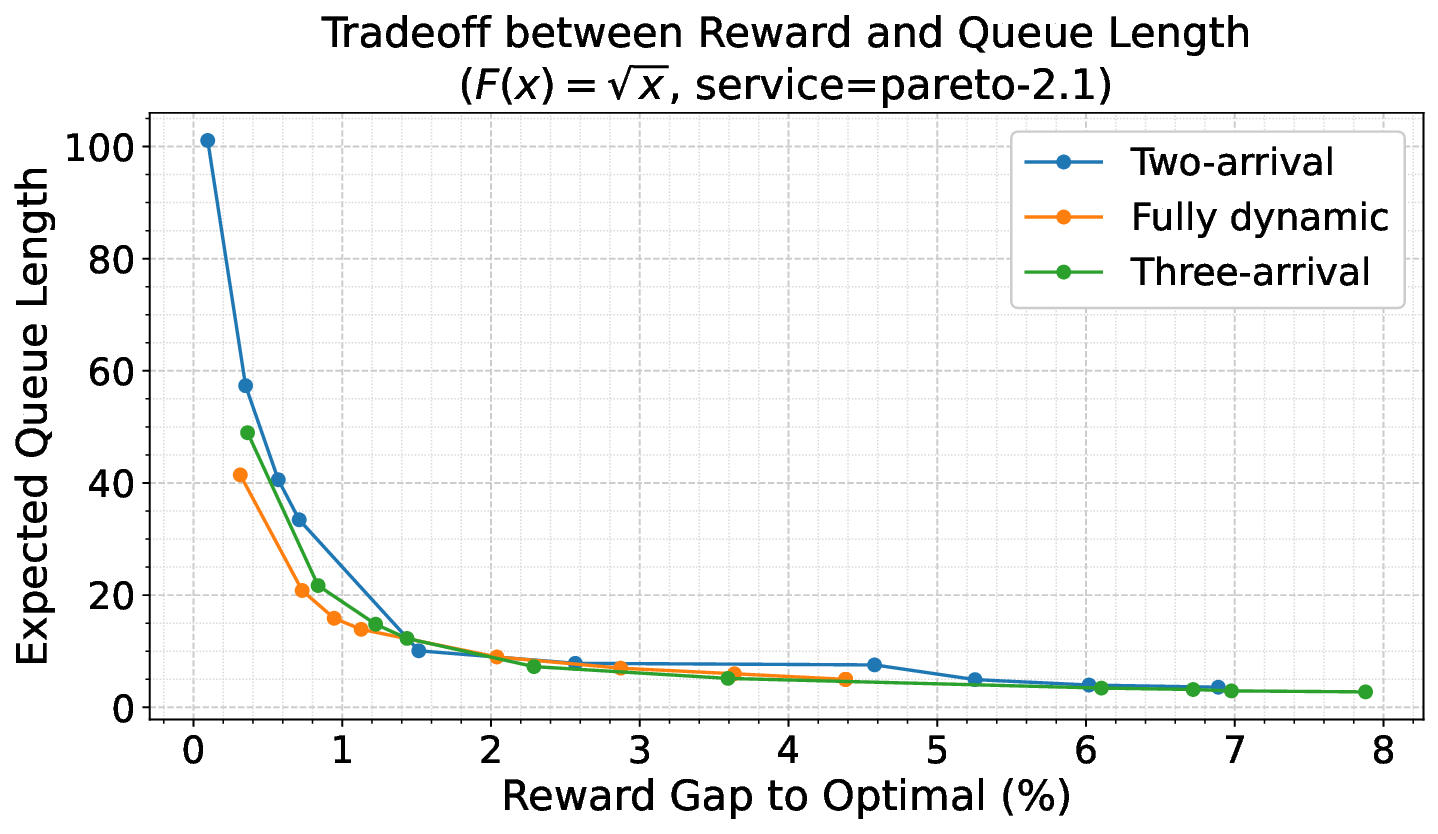}
    \caption{$F(x)=\sqrt{x}$}
    \label{fig:three_policy_sqrt}
\end{subfigure}

\caption{
Comparison of two-arrival, three-arrival, and fully dynamic policies with Pareto ($\alpha=2.1$) service times.
}
\label{fig:three_policy_comparison}

\end{figure}

From Fig.~\ref{fig:three_policy_comparison}, we observe that the efficiency--reward trade-off of the three-arrival policy is very similar to that of the two-arrival policy. Even after eliminating the effective upper buffer, the fully dynamic policy continues to outperform both threshold-based policies. This suggests that the performance advantage of the fully dynamic policy is not primarily driven by the presence of the upper buffer but rather by the stronger state-dependent drift around the target level \(q=B\). 


}

\subsection{Transient Performance across Initial Conditions and Service Distributions}

We next study the transient queue-length behavior under different initial conditions and service-time distributions. Building on the two-arrival and fully dynamic policies studied in Subsections~\ref{sec:numerical} and~\ref{sec:pareto}, respectively, we additionally include the static policy as a baseline in order to compare the fluctuations of sample paths for various control policies. We plot \(20\) independent queue-length trajectories, each for a time horizon of \(T=10^5\). Throughout this experiment, we set \(\varepsilon=0.01\). We consider two choices of initial queue lengths $q(0) \in \{0, 150\}$ and consider either exponential or Pareto service distributions with parameter \(\alpha=2.1\) introduced in Subsection~\ref{sec:pareto}. The results are plotted in Fig.~\ref{fig:transient_sample_paths}.

\begin{figure}[htbp]
    \centering

    \includegraphics[width=0.95\textwidth]{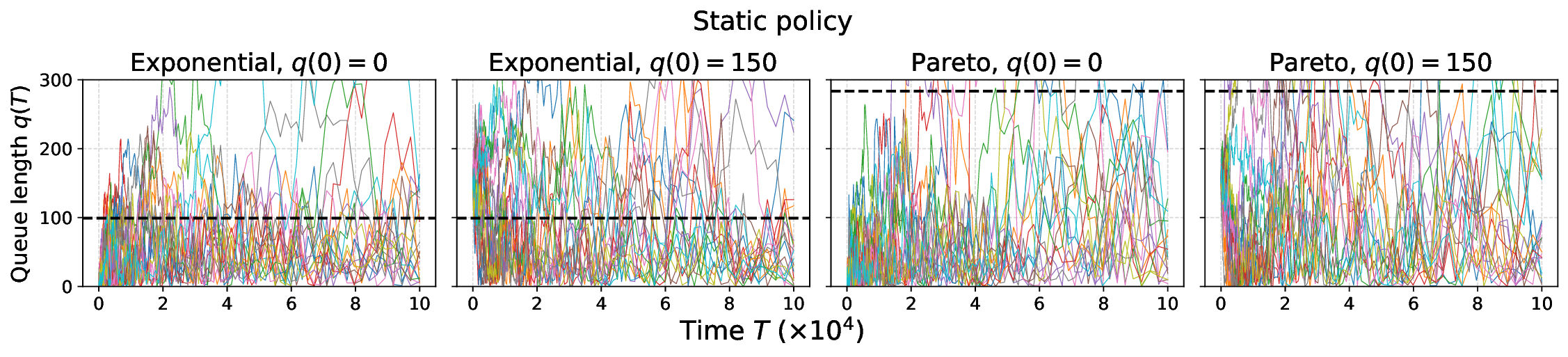}

    \vspace{-0.2em}

    \includegraphics[width=0.95\textwidth]{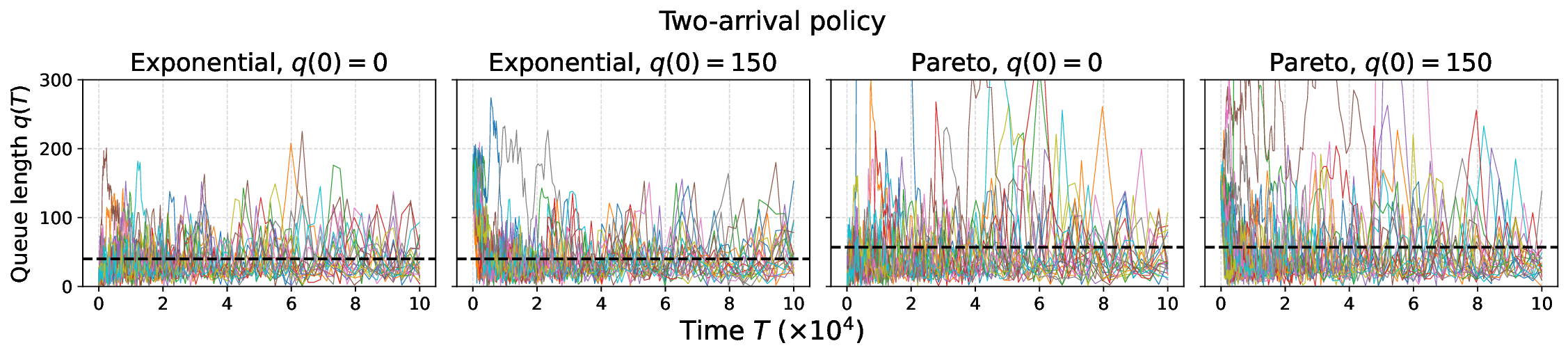}

    \vspace{-0.2em}

    \includegraphics[width=0.95\textwidth]{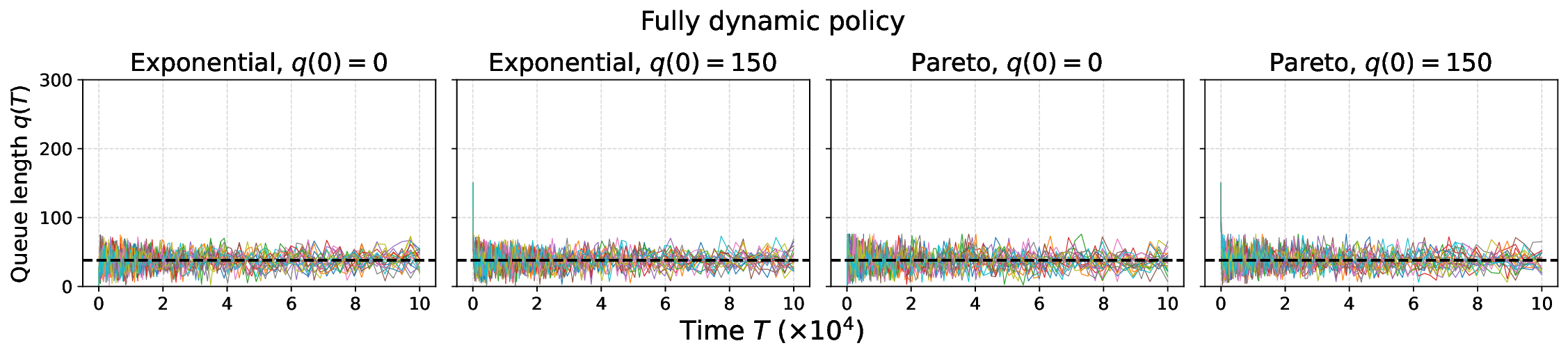}

    \caption{
    Sample-path trajectories of the queue length under the static,
    two-arrival, and fully dynamic policies.
    }
    \label{fig:transient_sample_paths}
\end{figure}

The first observation is that for all the 12 cases, the effect of initial queue length disappears over time. 
Next, one can observe the effect of the service time distribution, especially for the static and two-arrival policies. In particular, replacing exponential service times by the heavier-tailed Pareto distribution generates substantially greater sample-path variability. More specifically, the queue length exceeds 300 for several trajectories under the static arrival policy, a few trajectories under the two-arrival policy, and rarely any under the fully dynamic policy. Thus, the queue length fluctuation is the least under the fully dynamic policy. It is also interesting to see that the queue length trajectories under the fully dynamic policy remain largely insensitive to the service distributions, indicating robustness to the variance of the service distribution. This robustness can be understood through the state-dependent drift induced by the fully dynamic policy. When a long service time causes the queue to increase, the policy continuously reduces the arrival rate as the queue grows, thereby creating a progressively stronger corrective effect against further congestion. In comparison, the two-arrival policy provides only a coarser form of feedback through a single threshold, while the static policy provides no queue-dependent adjustment at all. 

\subsection{Comparison across difference market size when \(F\) is linear}
In throughput optimization, i.e., \(F(x)=x\), the optimal arrival control admits an explicit threshold structure. Specifically, the optimal policy takes the form 
\begin{equation} \label{lambdamax_0}
\lambda(q) = \begin{cases}
    \lambda_{\max}, \quad &\text{ if } q<\tau \\
    0, \quad &\text{ if } q\geq \tau
\end{cases} \quad, \tau = \left\lceil \frac{\log \frac{\lambda_{\max}-1+\varepsilon}{\varepsilon}}{\log\lambda_{\max}}-1 \right\rceil,
\end{equation}
where the threshold \(\tau\) is chosen to satisfy the regret constraint \eqref{eq:constraint}.


\begin{wrapfigure}{r}{0.6\textwidth}
    \centering
    \includegraphics[width=0.8\linewidth]{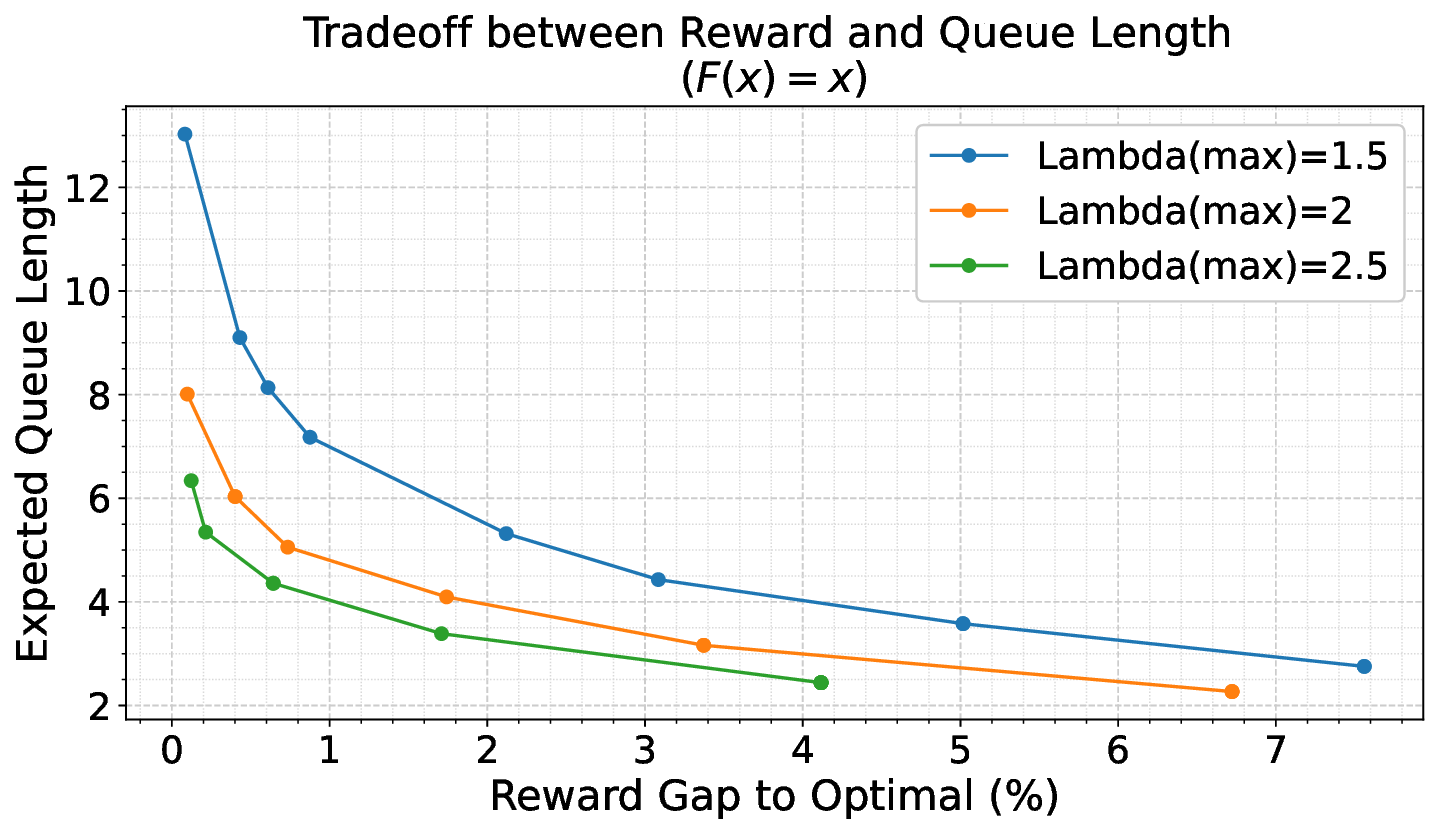}
    \caption{Tradeoff comparison between different \(\lambda_{\max}\)}
    \label{fig:Linear Throughput}
\end{wrapfigure}

In our numerical experiments, we fix the same tuning parameter \(\varepsilon\) and generate arrival policies according to \eqref{lambdamax_0} for different values of \(\lambda_{\max}\), which represent different market sizes. The simulation results are shown in Fig.~\ref{fig:Linear Throughput}.

Two observations emerge clearly from Fig.~\ref{fig:Linear Throughput}. First, under the same normalized regret ratio, a larger \(\lambda_{\max}\) yields a smaller expected steady-state queue length. This reflects the fact that, in larger markets, the system can more effectively concentrate congestion by admitting traffic aggressively when the queue is short and sharply cutting off arrivals once congestion builds up. Second, comparing Fig.~\ref{fig:Linear Throughput} with Fig.~\ref{fig:Simulation for different F}, linear rewards yield substantially lower queue lengths than strictly concave rewards at comparable normalized regret levels, even when \(\lambda_{\max}\) is only slightly greater than 1. This behavior is consistent with our theoretical characterization: the linear reward function falls into the non-concave-like regime, for which the optimal efficiency-reward trade-off scales as \(\Theta(\log 1/\varepsilon)\), in contrast to the \(\Theta(1/\sqrt{\varepsilon})\) scaling that arises under strictly concave rewards.

\section{Conclusion}
We study a single-server queue with state-dependent arrival control and characterize the fundamental efficiency–reward trade-off under a fluid benchmark formulation. By viewing classical heavy-traffic scaling as a special case, we develop a unified framework that explains how dynamic arrival control alters delay behavior when operating near optimal reward. Our results show that both market size and the local structure of the reward function at capacity play decisive roles. In small markets, dynamic control offers no efficiency improvement beyond classical heavy-traffic behavior, whereas in large markets it enables substantially improved delay performance. Moreover, when the reward function exhibits a concave-like behavior at capacity, fully dynamic arrival control is necessary to achieve the optimal \(\Theta(1/\sqrt{\varepsilon})\) scaling, while simpler policy classes incur unavoidable inefficiencies. In contrast, when this condition fails, simple two-level policies already achieve order-optimal performance \(\Theta(\log 1/\varepsilon)\). Together, these results provide a sharp and interpretable characterization of efficiency–reward trade-offs in queueing systems with endogenous arrivals.

\section*{Acknowledgments}

We thank Pengyu Qian for pointing out the potential connection between our work and the locally polyhedral condition \cite{huang2009delay}, which motivated the development of Section~\ref{polyhedral}.

\bibliographystyle{plainnat}
\bibliography{sample-base}

\clearpage

\appendix
\input{app}

\end{document}

%% file: app.tex
\section{} \label{appendix:proof for 4, 5.1}
\subsection{Proof of Proposition~\ref{lemma3.2}}\label{proof for prop 4.1}
\begin{proof}
    Let $\widetilde{C} = \{\lambda(q) : q\in \mathcal{S}\}\subseteq [0, \lambda_{\max}]$ is the image of arrival rate function $\lambda$, with $X$ is a random variable on $\widetilde{C}$. By definition we know $|\widetilde{C}|\leq|\mathcal{S}|$, which indicates that the cardinality of set $\widetilde{C}$ is countable. Thus, we can express $\widetilde{C} = \{\widetilde{\lambda_1}, \widetilde{\lambda_2}, \cdots\}$. Let $\alpha = (\alpha_1,\alpha_2,\cdots)\in \mathbb{R}_+^{|\widetilde{C}|}$ be the probability mass function corresponding to $\widetilde{C}$ satisfying 
    \[
    \alpha_i = \mathbb{P}(X=\widetilde{\lambda_i}) = \sum_{q\in\mathcal{S}, \lambda(q) = \widetilde{\lambda_i}}\pi_q, \] 
    then $\sum_{i=1}^{|\widetilde{C}|}\alpha_i=1$ and $\alpha_i\geq 0$ for all $i$ are satisfied.

    Therefore, with $\mathbb{E}_\alpha[F(X)] \leq F^\star$ in \eqref{eq:def Fstar}, we have
    \[
    \mathbb{E}_\pi[F(\lambda(\bar q))] = \sum_{q\in \mathcal{S}} F(\lambda(q))\pi_q = \sum_{i=1}^{|\widetilde{C}|}\sum_{q\in \mathcal{S}, \lambda(q)=\widetilde{\lambda_i}}F(\widetilde{\lambda_i})\pi_q = \sum_{i=1}^{|\widetilde{C}|}F(\widetilde{\lambda_i})\alpha_i = \mathbb{E}_\alpha[F(X)] \leq F^\star.
    \]
    This leads us to the conclusion that $F^\star$ is the supremum reward under any control policy.

    And we want to show $\mathbb{E}_\pi[\lambda(\bar q)] < 1$ if the CTMC is positive recurrent: 

For convenience, define $\pi_q = 0$ for all $q\in \mathbb{Z_+}\setminus \mathcal{S}$, therefore $\pi_{q+1} = \lambda(q)\pi_q$ for all $q\in \mathbb{Z_+}$, 
    \begin{equation} \label{idle time}
    \mathbb{E}_\pi[\lambda(\bar q)] = \sum_{q=0}^\infty \lambda( q)\pi_q = \sum_{q=0}^\infty \pi_{q+1} = 1-\pi_0,
    \end{equation}
    since $\pi_0>0$, we get $\mathbb{E}_\pi[\lambda(\bar q)]<1$,

    therefore
\[
    1\geq \mathbb{E}_\pi(\lambda(\bar{q})) = \sum_{q\in \mathcal{S}} \lambda(q)\pi_q = \sum_{i=1}^{|\widetilde{C}|}\sum_{q\in \mathcal{S}, \lambda(q)=\widetilde{\lambda_i}}\widetilde{\lambda_i}\pi_q = \sum_{i=1}^{|\widetilde{C}|}F(\widetilde{\lambda_i})\alpha_i = \mathbb{E}_\alpha[X].
\]
\end{proof}

\subsection{Proof of Proposition~\ref{prop small market}} \label{proof_small_market}
\begin{proof}
    To prove Proposition~\ref{prop small market}, we first prove the following proposition.
\begin{lemma}\label{lemmma to prove small market}
    For any $F$ satisfies Assumption~\ref{assump2}, ~\ref{assump4} and ~\ref{assump F'(1)>0}, there exists a slope $k\in (0, F'(1)]$, such that $F(x)\leq k(x-1) +F(1)$ for all $x\in [0,1]$.
\end{lemma}
\begin{proof}
    Under Assumption~\ref{assump2} $F'(1)$ is continuous, and under Assumption~\ref{assump F'(1)>0}, $F'(1)>0$, so there exists a $s\in (0,1)$ such that $F'(x)>0$ for all $x\in [s,1]$.\\
    Define $k_1 = \min\limits_{x\in[s,1]}F'(x)$, then for any $x\in [s,1]$
    \[
    \begin{split}
    F(1) - F(x) = \int_{x}^1 F'(x) dx &\geq \int_{x}^1 k_1 dx = k_1 (1-x),
    \end{split}
    \]
    hence $F(x)\leq k_1(x-1)+F(1)$ for all $x\in [s,1]$.\\
    Next is to show $F(x)\leq k_2(x-1)+F(1)$ for some $k_2>0$ and for all $x\in [0, s]$.\\
    By Assumption~\ref{assump4}, $F(x)<F(1)$ for any $x\in [0,1)$, so for $s<1$,
    \[
    \max_{x\in [0,s]} F(x) < 1.
    \]
    Let $k_2 = F(1) - \max\limits_{x\in [0,s]} F(x)>0$, then for any $x\in [0,s]$, we get
    \[
    F(x)\leq \max_{x\in [0,s]} F(x) = F(1) - k_2 \leq k_2 (x-1)+F(1).
    \]
    Above all, choose $k = \min\{k_1,k_2\}>0$, then
    \begin{equation}
    F(x)\leq k(x-1)+F(1), \quad \forall x\in [0,1]. \label{5.1line}
    \end{equation}
\end{proof}

Since $\lambda_{\max}=1$, and by Assumption~\ref{assump4}, $F^\star=F(1)$, so $\mathbb{E}_\pi[F(\lambda(\bar q))]\geq F^\star-\varepsilon = F(1)-\varepsilon$, 
then for any policy $\lambda$, let $\widetilde{C} = \{\lambda(q) : q\in \mathcal{S}\} = \{\widetilde{\lambda_1}, \widetilde{\lambda_2}, \cdots\} \subseteq [0, 1]$ is the image of arrival rate function $\lambda$. $X$ is a random variable on $\widetilde{C}$ with probability measure $\alpha$ satisfying $P(X=\tilde{\lambda_i}) = \alpha_i$ for all $i$ and $\sum_{i=1}^{|\widetilde{C}|}\alpha_i=1$. Hence we have
\[
\begin{split}
    F(1)-\varepsilon &\leq \mathbb{E}_\alpha[F(X)] = \sum_{i=0}^{|\widetilde{C}|} F(\widetilde{\lambda_i})\alpha_i
    \overset{(a)}\leq \sum_{i=0}^{|\widetilde{C}|} (\widetilde{\lambda_i}(k-1)+F(1))\alpha_i \\
    &\overset{(b)}= k\mathbb{E}_\alpha[X] - k + F(1)
    \overset{(c)}= -k\pi_0+F(1),
\end{split}
\]
where (a) is from \eqref{5.1line}, (b) is by $\mathbb{E}_\alpha[X] = \sum_{i=0}^{|\tilde{C}|}\tilde{\lambda_i}\alpha_i$ and $\sum_{i=1}^{|\widetilde{C}|}\alpha_i=1$. The last one (c) is by Section~\ref{proof for prop 4.1}, equation \eqref{idle time}.

Hence $\pi_0\leq \frac{\varepsilon}{k}$ is a necessary but not sufficient condition of $R(\lambda) = F(1)-\mathbb{E}_\pi[\lambda(\bar q)]\leq \varepsilon$, so we have:
\begin{equation} \label{eq:opt in small market}
\begin{aligned}
    q^\star&\geq \min\,\, \mathbb{E}[q]=\sum_{i=0}^\infty i\pi_i\quad 
    \text{s.t.}\,\,\pi_0\leq \frac{\varepsilon}{k} \\
    & \text{and }\pi_{i+1}\leq \lambda_{\max}\pi_i = \pi_i \,\,\forall i\geq 0,
    \quad \sum_{i=0}^\infty \pi_i=1. 
\end{aligned}
\end{equation}
The optimal solution of \eqref{eq:opt in small market} achieved at
\[
\hat{\pi_i} = \begin{cases}
\frac{\varepsilon}{k}, &\quad i\leq B,\\
1-(B+1)\frac{\varepsilon}{k}, &\quad i=B+1\\
0, &\quad i\geq B+2
\end{cases}
\]
where $B$ is the threshold satisfying $(B+1)\frac{\varepsilon}{k}\leq 1< (B+2)\frac{\varepsilon}{k}$, so we have
\begin{align*}
    q^\star\geq \sum_{i=0}^\infty\hat{\pi_i} &= \frac{\varepsilon} {k} \sum_{i=0}^B i + (B+1)(1-(B+1)\frac{\varepsilon}{k})
    =B+1+\frac{\varepsilon}{k}\left(\frac{B(B+1)}{2}-(B+1)^2\right)\\
    &= B+1 - \frac{\varepsilon}{k}(\frac{B(B+1)}{2} + B+1)
    \overset{(a)}{\geq} B+1-\frac{B}{2}-1=\frac{B}{2}=\frac{k}{2\varepsilon} + o(\frac{1}{\varepsilon}),
\end{align*}
where (a) is by $(B+1)\frac{\varepsilon}{k}
\leq 1$. Hence $q^\star= \Omega\left(\frac{1}{\varepsilon}\right)$ for all control policies with small-market $(\lambda_{\max} = 1)$.

\end{proof}

\subsection{Proof of Proposition~\ref{prop:how_many_supp_point}}\label{opt_sol_large_market}
From the textbook \cite{moment_book}, Section 6.6, consider the problem

\begin{equation*}
    \sup_{\alpha\in \mathfrak{M}} \mathbb{E}_\alpha[F(X)], \quad F: [0, \lambda_{\max}] \rightarrow \mathbb{R}
\end{equation*}
We denote $\bar{\mathfrak{P}}$ as the set of probability measures on the measurable space is $([0,\lambda_{\max}], \mathcal{B}([0,\lambda_{\max}]))$, and $\mathfrak{M}\subseteq \bar{\mathfrak{P}}$ is defined by moment constraints as follows:

\begin{align*}
    \mathfrak{M} := \left\{\alpha \in \bar{\mathfrak{P}} :\quad \mathbb{E}_\alpha [X] \leq 1 \right\}
\end{align*}
Denote $\bar{\mathfrak{P}}_2^*$ as the set of probability measures on $([0,\lambda_{\max}], \mathcal{B}([0,\lambda_{\max}]))$ having a finite support of at most $2$ support points. Then, by 
Proposition 6.40 of \cite{moment_book},
we have
\begin{equation*}
    \sup_{\alpha\in \mathfrak{M}_{2}^*} \mathbb{E}_P[F(X)] = \sup_{\alpha\in \mathfrak{M}} \mathbb{E}_P[F(X)],
\end{equation*}
where $\mathfrak{M}_{2}^* = \mathfrak{M} \cap \bar{\mathfrak{P}}_2^*$.

Therefore, there exists an optimal solution of \eqref{eq:def Fstar} whose probability measure has at most two support points $0\leq x_2^\star\leq x_1^\star\leq \lambda_{\max}$, with $\alpha(x_1^\star)=p^\star$ and $\alpha(x_2^\star)=1-p^\star$ for some $p^\star\in [0,1]$.

Next we will show that 
\begin{itemize}
    \item if $F$ violates condition \eqref{eq:concave_like}, then the optimal support points satisfy $x_1^\star \in (1,\lambda_{\max}]$, $x_2^\star \in [0,1)$, and $p^\star = \frac{1 - x_2^\star}{x_1^\star - x_2^\star}$,
    \item if $F$ satisfies condition \eqref{eq:concave_like}, then the unique optimal solution is degenerate with $x_1^\star = x_2^\star = 1$ and $\alpha(\{1\}) = 1$.
\end{itemize}

\subsubsection{When $F$ violates \eqref{eq:concave_like}}\label{F_violate_(1)}
We aim to show that \eqref{eq:def Fstar} has at least one optimal solution with two support points. 

In this case, there exists $x_1\neq x_2$ and some $p \in (0,1)$ such that $p x_1 + (1-p) x_2 = 1$ and satisfying $F(1)\leq pF(x_1)+(1-p)F(x_2)$. Hence, either $F(1) < F^\star$, or $F(1) = F^\star$ but the maximizer is not unique. Moreover, by Assumption~\ref{assump4}, $F(x)<F^\star$ for all $x\in [0,1)$.

Thus, we conclude that \eqref{eq:def Fstar} admits a feasible two support points solution that achieves an objective value as large as any single-point solution. Combining this with 
Proposition 6.40 of \cite{moment_book}
guarantees that there exists an optimal solution of \eqref{eq:def Fstar} with exactly two support points, i.e., $\exists x_1^\star>x_2^\star$ and $\alpha(x_1^\star)=p^\star$ , $\alpha(x_2^\star)=1-p^\star$ for some $p^\star\in (0,1)$. We now establish the following properties:
\begin{itemize}
    \item $x_1^\star\in (1,\lambda_{\max}]$ and $x_2^\star\in [0,1)$.
    \item $F(x_1^\star)>F(x_2^\star)$
    \item $p^\star=\frac{1-x_2^\star}{x_1^\star-x_2^\star}$
\end{itemize}
For the first property, we can prove by contradiction: If $x_1^\star\leq 1$, $x_2^\star<1$, then by Assumption~\ref{assump4}, $F(x_1^\star)\leq F^\star$ and $F(x_2^\star)<F^\star$, so $p F(x_1^\star) + (1-p)F(x_2^\star) < F^\star$ for all $p\in (0,1)$, which contradicts optimality. On the other hand, if $x_1^\star>1$, $x_2^\star\geq 1$, then $\mathbb{E}_\alpha[X] = px_1^\star+(1-p)F(x_2^\star)>1$ for all $p\in (0,1)$, so it violates the feasibility condition. Therefore, we conclude that $x_1^\star\in (1,\lambda_{\max}]$ and $x_2^\star\in [0,1)$.

For the second property, also prove by contradiction. Suppose $F(x_1^\star)\leq F(x_2^\star)$ given $x_1^\star>x_2^\star$, then $p F(x_1^\star) + (1-p)F(x_2^\star) \leq F(x_2^\star)< F^\star$ since $x_2^\star<1$ implies $F(x_2^\star) < F^\star$.

For the last one, suppose $p^\star x_1^\star + (1-p^\star)x_2^\star<1$. Lets define a new probability measure  $\alpha(x_1^\star)=p' = p^\star+\delta$, where $\delta = \frac{1-p^\star x_1^\star - (1-p^\star)x_2^\star}{x_1^\star-x_2^\star}>0$, so that $p' x_1^\star + (1-p')x_2^\star=1$. Then
\begin{align}
    p'F(x_1^\star) + (1-p')F(x_2^\star) &= (p^\star+\delta)F(x_1^\star) + (1-p^\star-\delta)F(x_2^\star)\notag\\
    &= p^\star F(x_1^\star) + (1-p^\star)F(x_2^\star) + \delta(F(x_1^\star)-F(x_2^\star))\notag\\
    &\overset{(a)}{>}p^\star F(x_1^\star) + (1-p^\star)F(x_2^\star) \label{contra_opt}
\end{align}
where (a) follows from $F(x_1^\star)>F(x_2^\star)$. Equation \eqref{contra_opt} contradicts the optimality of $(p^\star, x_1^\star, x_2^\star)$. Therefore, the optimal solution must satisfying $p^\star x_1^\star + (1-p^\star)x_2^\star=1$, thus $p^\star=\frac{1-x_2^\star}{x_1^\star-x_2^\star}$.
\subsubsection{When $F$ satisfies \eqref{eq:concave_like}}
When $F$ satisfies \eqref{eq:concave_like}, then $F(1)>pF(x_1) + (1-p)x_2$ for all $p\in (0,1)$ and $x_1>x_2$ satisfying $px_1+(1-p)x_2=1$.

Consider any two support point solution for which $px_1+(1-p)x_2<1$. Using a similar argument as in Section~\ref{F_violate_(1)}, let $p'=p+\delta$ with $\delta = \frac{1-p x_1 - (1-p)x_2}{x_1-x_2}>0$, so that $p'x_1+(1-p')x_2=1$. Then 
\begin{align*}
    pF(x_1)+(1-p)F(x_2)<p'F(x_1)+(1-p')F(x_2)<F(1).
\end{align*}
Therefore, $F(1)> \mathbb{E}_\alpha[F(X)]$ for all non-constant random variables $X$. Moreover, by Assumption~\ref{assump4}, we know that $F(x)<F^\star$ for all $x\in [0,1)$, so $F^\star=F(1)$, and it can only be attained at $x=1$ with probability $\alpha(\{1\})=1$.

\subsection{Proof of connections to the Locally Polyhedral Condition \cite{huang2009delay}} \label{Polyhedral}
The paper \cite{huang2009delay} observes similar trade-off between queue length minimization and reward maximization as ours. To interpret their results in our notations, they propose a primal-dual policy that achieves queue lengths of $\Theta(\log 1/\varepsilon)$ when the locally polyhedral condition is satisfied and $\Theta(1/\sqrt{\varepsilon} \log 1/\varepsilon)$ when it is not while ensuring at most $\varepsilon$ regret. Note that we obtain the $\Theta(\log 1/\varepsilon)$ scaling mainly as we are able to impose a large negative drift $\Theta(1)$ towards the threshold $\tau = \Theta(\log 1/\varepsilon)$. Moreover, in the case of $\Theta(1/\sqrt{\varepsilon})$, we are only able to induce a negative drift of $\Theta(\sqrt{\varepsilon})$, close to the threshold $B$. This intuition seems to be consistent with the intuition in \cite{huang2009delay} for the emergence of such scalings.

However, the model of \cite{huang2009delay} is completely different than ours. We highlight three of the differences here. First, \cite{huang2009delay} allows the ``action'' to depend on the realized system noise, which is a much stronger assumption compared to our scenario. For example, in the context of dynamic pricing, this assumption would correspond to knowing the utility of the incoming customer exactly, while our framework would only assume a distribution over this utility. 
Second, the locally smooth condition in Equation~(42) of
\cite{huang2009delay} is assumed to obtain $\Theta(1/\sqrt{\varepsilon} \log 1/\varepsilon)$ queue lengths, while we do not make any such assumptions. Third, in \cite{huang2009delay}, the number of random states are finite, which would correspond to at most finite number of possible arrival rates in our setting, which is restrictive.

Now we start to prove Theorem~\ref{relative of 18eq}, we can write the fluid benchmark \eqref{eq:def Fstar} as:
\begin{equation} \label{primal}
\begin{aligned}
    \underset{\alpha\in \mathcal{P}([0,\lambda_{\max}])}{\text{min}} \quad & -\mathbb{E}_\alpha[F(X)] \\
    \text{ s.t. }\qquad & \mathbb{E}_\alpha[X] \leq 1.
\end{aligned}
\end{equation}
The dual problem of \eqref{primal} is:
\begin{equation}\label{dual}
    \begin{aligned}
        \max\quad  & q(U)\\
        \text{ s.t. }\quad  & U>0,
    \end{aligned}
\end{equation}
where the dual function $q(U)$ is defined as:
\begin{equation}\label{lagrange}
    q(U) = \inf_{\alpha\in \mathcal{P}([0,\lambda_{\max}])} \{\mathbb{E}_\alpha[-F(X)+UX]-U\}
\end{equation}
Since \eqref{lagrange} doesn't have any constraint, so we can write
\begin{equation}\label{dual objective}
    q(U) = \inf_{x\in [0,\lambda_{\max}]} \{-F(x)+Ux-U\}
\end{equation}
When \(F\) is not concave-like, from Lemma~\ref{lemma4.3}, we have
\begin{equation}\label{2 support}
F(x)\leq \frac{F(x_1^\star)-F(x_2^\star)}{x_1^\star-x_2^\star}x + \frac{x_1^\star F(x_2^\star)-x_2^\star F(x_1^\star)}{x_1^\star-x_2^\star}
\end{equation}
for all $x\in [0,\lambda_{\max}]$, and $0\leq x_1^\star<1<x_2^\star\leq \lambda_{\max}$. 

Moreover, the optimal solution of \eqref{primal} is $-F^\star$, where $F^\star = \frac{F(x_1^\star)-F(x_2^\star)}{x_1^\star-x_2^\star}x + \frac{x_1^\star F(x_2^\star)-x_2^\star F(x_1^\star)}{x_1^\star-x_2^\star}$.
\begin{lemma}
    The optimal solution of \eqref{dual objective} is $\max\limits_{U\geq 0} q(U) = -F^\star$, where $\arg \max\limits_{U\geq 0} q(U) = \frac{F(x_1^\star)-F(x_2^\star)}{x_1^\star-x_2^\star}$.
\end{lemma}
\begin{proof}
    \begin{align*}
        &q\left(\frac{F(x_1^\star)-F(x_2^\star)}{x_1^\star-x_2^\star}\right) = \inf_{x\in [0,\lambda_{\max}]} \left\{-F(x)+\frac{F(x_1^\star)-F(x_2^\star)}{x_1^\star-x_2^\star}x-\frac{F(x_1^\star)-F(x_2^\star)}{x_1^\star-x_2^\star}\right\}\\
        &\overset{(a)}{\geq} \inf_{x\in [0,\lambda_{\max}]} \left\{-\frac{F(x_1^\star)-F(x_2^\star)}{x_1^\star-x_2^\star}x - \frac{x_1^\star F(x_2^\star)-x_2^\star F(x_1^\star)}{x_1^\star-x_2^\star}+\frac{F(x_1^\star)-F(x_2^\star)}{x_1^\star-x_2^\star}x-\frac{F(x_1^\star)-F(x_2^\star)}{x_1^\star-x_2^\star}\right\}\\
        &= -F^\star
    \end{align*}
\end{proof}
where (a) is derived from \eqref{2 support}, and for $i=1,2$, we have
\begin{align*}
    &q\left(\frac{F(x_1^\star)-F(x_2^\star)}{x_1^\star-x_2^\star}\right) = \inf_{x\in [0,\lambda_{\max}]} \left\{-F(x)+\frac{F(x_1^\star)-F(x_2^\star)}{x_1^\star-x_2^\star}x-\frac{F(x_1^\star)-F(x_2^\star)}{x_1^\star-x_2^\star}\right\}\\
    &\leq  -F(x_i^\star)+\frac{F(x_1^\star)-F(x_2^\star)}{x_1^\star-x_2^\star}x_i^\star-\frac{F(x_1^\star)-F(x_2^\star)}{x_1^\star-x_2^\star} \\
    &= -\frac{F(x_1^\star)-F(x_2^\star)}{x_1^\star-x_2^\star}x_i^\star - \frac{x_1^\star F(x_2^\star)-x_2^\star F(x_1^\star)}{x_1^\star-x_2^\star}+\frac{F(x_1^\star)-F(x_2^\star)}{x_1^\star-x_2^\star}x_i^\star-\frac{F(x_1^\star)-F(x_2^\star)}{x_1^\star-x_2^\star}\\
    &=-F^\star
\end{align*}
Therefore, $q\left(\frac{F(x_1^\star)-F(x_2^\star)}{x_1^\star-x_2^\star}\right)=-F^\star$, and by weak duality, $\max\limits_{U\geq 0} q(U) \leq -F^\star$, so we have $\max\limits_{U\geq 0} q(U) = q\left(\frac{F(x_1^\star)-F(x_2^\star)}{x_1^\star-x_2^\star}\right) = -F^\star$.

\begin{theorem}
    Let $U^\star = \frac{F(x_1^\star)-F(x_2^\star)}{x_1^\star-x_2^\star}$, and $L = \min\{x_1^\star-1, 1-x_2^\star\}>0$, such that
    \begin{equation}\label{paper eq18}
        q(U^\star) \geq q(U) +L\lvert U^\star-U\rvert
    \end{equation}
    for all $U\geq 0$, where $q(U^\star) = \max\limits_{U\geq 0} q(U) = -F^\star$
\end{theorem}
\begin{proof}
    When $U = U^\star$, trivial.

    When $0\leq U<U^\star$, 
    \begin{align*}
        q(U) &= \inf_{x\in [0,\lambda_{\max}]}\{ -F(x)+Ux-U\} \\
        &\leq -F(x_1^\star) + Ux_1^\star-U \\
        &= -F(x_1^\star) + U^\star (x_1^\star - 1) + (U-U^\star)(x_1^\star-1) \\
        &= q(U^\star) - (x_1^\star-1)(U^\star-U),
    \end{align*}
hence $q(U^\star) \geq q(U) + (x_1^\star-1)(U^\star-U) \geq q(U) + L(U^\star-U)$.

Similarly, when $U>U^\star$,
\begin{align*}
    q(U) &= \inf_{x\in [0,\lambda_{\max}]}\{ -F(x)+Ux-U\} \\
        &\leq -F(x_2^\star) + Ux_2^\star-U \\
        &= -F(x_2^\star) + U^\star(x_2^\star-1) + (U-U^\star)(x_2^\star-1) \\
        &= q(U^\star) - (1-x_2^\star)(U-U^\star)
\end{align*}
hence $q(U^\star) \geq q(U) + (1-x_2^\star)(U-U^\star) \geq q(U) + L(U-U^\star)$.

Above all, \eqref{paper eq18} hold for all $U\geq 0$.
\end{proof}
Here we finished the proof of Theorem~\ref{relative of 18eq}.

\section{Continue the proof in Section ~\ref{sec:qstar_upper_bound}} \label{longest_proof}
\begin{align*}
    &\mathbb{E}_\pi[F(\lambda(\bar q))] = \sum_{i=0}^{2B}F(\lambda(i))\pi_i\\
    ={}& \pi_0\left(\sum_{i=1}^{B-1}F\left(\left(1+\frac{1}{m+i+1}\right)^2\right)\left(\frac{m+i+1}{m+1}\right)^2 + \sum_{i=B+1}^{2B-1}F\left(\left(1-\frac{1}{m+2B-i+1}\right)^2\right)\right.\\ &\left.\left(\frac{m+2B-i+1}{m+1}\right)^2\right)
    +\pi_0F\left(\left(1-\frac{1}{m+B+1}\right)^2\right)\left(\frac{m+B+1}{m+1}\right)^2 + \pi_0 F\left(\left(1+\frac{1}{m+1}\right)^2\right) + \pi_{2B}F(0)\\
    ={}& \pi_0\left(\sum_{i=1}^{B-1}F\left(\left(1+\frac{1}{m+i+1}\right)^2\right)\left(\frac{m+i+1}{m+1}\right)^2 + \sum_{i=1}^{B-1}F\left(\left(1-\frac{1}{m+i+1}\right)^2\right)\left(\frac{m+i+1}{m+1}\right)^2\right)\\
    &+\pi_0F\left(\left(1-\frac{1}{m+B+1}\right)^2\right)\left(\frac{m+B+1}{m+1}\right)^2 + \pi_0 F\left(\left(1+\frac{1}{m+1}\right)^2\right)+\pi_0 F(0)\\
    \overset{(c)}{\geq}& \pi_0\sum_{i=1}^{B-1}\left(2F(1)+\frac{2F'(1)+4F''(1)}{(m+i+1)^2}+\frac{3F''(1)+12F^{(3)}(1)-4K}{3(m+i+1)^4}+\frac{F^{(3)}(1)-6K}{3(m+i+1)^6}\right.\\ &\left.-\frac{K}{12(m+i+1)^8}\right)
    \cdot\left(\frac{m+i+1}{m+1}\right)^2
    +\pi_0\left(F(1)+F'(1)\left(-\frac{2}{m+B+1}+\frac{1}{(m+B+1)^2}\right)  + \frac{F''(1)}{2}\right.\\ &\left.\left(-\frac{2}{m+B+1}+\frac{1}{(m+B+1)^2}\right)^2 + \frac{M}{6}\left(-\frac{2}{m+B+1}+\frac{1}{(m+B+1)^2}\right)^3\right)\cdot  \left(\frac{m+B+1}{m+1}\right)^2 
    + \pi_0 \left(\right.\\ &\left.F(1)+F'(1)\left(\frac{2}{m+1}+\frac{1}{(m+1)^2}\right)+\frac{F''(1)}{2}\left(\frac{2}{m+1}+\frac{1}{(m+1)^2}\right)^2-\frac{M}{6}\left(\frac{2}{m+1}+\frac{1}{(m+1)^2}\right)^3\right)
    \\
    \overset{(d)}\geq{}& \frac{\pi_0}{(m+1)^2}\left(2\sum_{i=2}^{B}(m+i)^2F(1)+(2F'(1)+4F''(1))(B-1)-C\right) 
    +\frac{\pi_0}{(m+1)^2}
    \left((m+B+1)^2\right.\\ &\left.F(1)-(2m+2B+1)F'(1)+\left(2-\frac{2}{m+B+1}+\frac{1}{2(m+B+1)^2}\right)F''(1)-\frac{M}{6}\frac{(2m+2B+1)^3}{(m+B+1)^4}\right)\\
    &+ \frac{\pi_0}{(m+1)^2} \left(F(1)+F'(1)\left(2m+3\right)+F''(1)\left(2+\frac{2}{m+1}+\frac{1}{2(m+1)^2}\right)-\frac{M}{6}\frac{(2m+3)^3}{(m+1)^4}\right) \\
    ={}& F(1)(1-\pi_0) + \frac{\pi_0}{(m+1)^2}F'(1)\left(2B-2-2m-2B-1+2m+3\right) + \frac{\pi_0}{(m+1)^2}F''(1)\left(4B-4+2\right.\\ &\left.-\frac{2}{m+B+1}+\frac{1}{2(m+B+1)^2}+2+\frac{2}{m+1}+\frac{1}{2(m+1)^2}\right) - \frac{\pi_0}{(m+1)^2}C\\
    &-\frac{\pi_0}{(m+1)^2}\frac{M}{6}\left(\frac{(2m+2B+1)^3}{(m+B+1)^4}+\frac{(2m+3)^3}{(m+1)^4}\right)
    \\
    \geq& F(1)(1-\pi_0) + \frac{\pi_0}{(m+1)^2}F''(1)\left(4B-\frac{2}{m+B+1}+\frac{1}{2(m+B+1)^2}+\frac{2}{m+1}+\frac{1}{2(m+1)^2}\right)\\
    &-\frac{\pi_0}{(m+1)^2}C-\frac{\pi_0}{(m+1)^2}\frac{M}{6}\left(\frac{8}{m+B+1}+\frac{27}{m+1}\right)\\
    \overset{(e)}{=} & F(1)-\frac{3(m+1)^2F(1)}{3(2B+1)m^2 + 3(B+1)^2(2m+1)+B(B+1)(2B+1)}\\
    &+ \frac{12B F''(1)}{3(2B+1)m^2+3(B+1)^2(2m+1)+B(B+1)(2B+1)}\\
    &+ \frac{3}{3(2B+1)m^2+3(B+1)^2(2m+1)+B(B+1)(2B+1)} \cdot O(1)
    \\
    \overset{(f)}={}& F(1) - \frac{-12B F''(1)}{2B^3} + O\left(\frac{1}{B^3}\right)
    = F(1) - \frac{-6F''(1)}{B^2} + o\left(\frac{1}{B^2}\right) \overset{(g)}\geq F(1)- \frac{6}{7}\varepsilon + o(\varepsilon)\overset{(h)}\geq F(1)-\varepsilon, 
\end{align*}

where (c) is obtained by applying the fourth-order Taylor expansion to each pair $F((1+\frac{1}{i+1})^2), F((1-\frac{1}{i+1})^2)$, and applying the third-order Taylor expansion to the remaining two terms, under Assumption~\ref{eq:smooth_F}. Next, (d) follows by observing that, after multiplication by
$(m+i+1)^2$, the remaining terms in the paired Taylor expansions
are linear combinations of
$(m+i+1)^{-2}$, $(m+i+1)^{-4}$, and $(m+i+1)^{-6}$.
Their absolute values are uniformly summable in $m$ and $B$, since
\(
\sum_{i=1}^{B-1}\frac{1}{(m+i+1)^2}
\le
\sum_{i=1}^{\infty}\frac{1}{i^2}
=
\frac{\pi^2}{6}.
\)
Hence, there exists a constant $C>0$, independent of $m$ and $B$,
such that their total contribution is bounded below by $-C$. Then, (e) follows by substituting \eqref{plug_pi_0} and collecting the
remaining terms into $O(1)$. Indeed, since $m\ge 0$, all the
remaining terms involving $m$ and $B$ are uniformly bounded, and
the implicit constant is independent of $m$ and $B$.

The inequality (f) follows from
\(
m=\left\lceil
\frac{1}{\sqrt{\bar\lambda_{\max}}-1}
\right\rceil-1
\le
\left\lceil
\frac{1}{\sqrt{1+\varepsilon^\alpha}-1}
\right\rceil-1
=
2\varepsilon^{-\alpha}+O(1).
\)
Since \(0<\alpha<\frac14\), it follows that
\(
(m+1)^2=O(\varepsilon^{-2\alpha})
=o(\varepsilon^{-1/2})
=o(B).
\)

Finally (g) is by $B=\left\lceil\frac{\sqrt{-7F''(1)}}{\sqrt{\varepsilon}}\right\rceil\geq \frac{\sqrt{-7F''(1)}}{\sqrt{\varepsilon}}$, and (h) is hold when $\varepsilon$ is sufficiently small.

Note that the condition \(\alpha<\frac14\) is required for the above argument. Indeed, when
\(\alpha\ge\frac14\), we have
\(
\frac{(m+1)^2}{B^3}\not=o(\varepsilon).
\)
Hence, the loss associated with the coefficient of \(F(1)\) is no longer negligible relative to \(\varepsilon\), and the above argument does not guarantee the regret constraint.

\section{}\label{appendix:Proof for 5.2}
\subsection{When \(F\) is not concave-like}\label{subsection:2pointanalysis}
In this section, we present a detailed proof of Theorem~\ref{general:2 supports} , which establishes the lower and upper bound of $q^\star$.
The following Lemma~\ref{lemma4.3} shows that $F(x)$ is always less than or equal to the straight line passing through $(\opt_1, F(\opt_1)), (\opt_2, F(\opt_2))$, as illustrated in Fig.~\ref{fig:two points}.

\begin{lemma}\label{lemma4.3}

Suppose $\opt_1 \in (1, \lambda_{\max}], \opt_2 \in [0, 1), p^\star \in (0, 1)$ such that $\alpha(\{\opt_1\})=p^\star$ and $\alpha(\{\opt_2\})=1-p^\star$ is an optimal solution of \eqref{eq:def Fstar}, then $F(x)\leq \frac{F(\opt_1)-F(\opt_2)}{\opt_1-\opt_2}x+\frac{\opt_1 F(\opt_2)-\opt_2 F(\opt_1)}{\opt_1-\opt_2}$ for all $x\in [0,\lambda_{\max}]$.
\end{lemma}
\begin{lemma}\label{pi_0<=sth}
    The throughput $\mathbb{E}_\pi[\lambda(\bar q)] = 1-\pi_0\geq 1-\frac{\opt_1-\opt_2}{F(\opt_1)-F(\opt_2)}\varepsilon$ for all policy $\lambda$ satisfying $R(\lambda)\leq \varepsilon$, where $\opt_1, \opt_2$ is defined in Lemma~\ref{lemma4.3}.
\end{lemma} 

We defer the proof of Lemma~\ref{lemma4.3} and Lemma~\ref{pi_0<=sth} in Section~\ref{sec6_line} and Section~\ref{sec6_pi0}, respectively.

We are now ready to prove the upper bound of \(q^\star\) in Theorem~\ref{general:2 supports}.

\subsubsection{Proof of the upper bound of \(q^\star\) in Theorem~\ref{general:2 supports}}
\begin{proof}
Recall the control policy defined in \eqref{eq:control_policy_two_supp_point}, wherein $\lambda(q)=\opt_1$ when $q\leq\tau$ and $\lambda(q)=\opt_2$ when $q>\tau$, so
\begin{equation}\label{eq:distribution}
    \pi_i = \begin{cases}
        \pi_0(\opt_1)^i, \quad &i\leq \tau\\
        \pi_0(\opt_1)^{\tau+1}(\opt_2)^{i-\tau-1}, \quad &i\geq \tau+1.
    \end{cases}
\end{equation}

Let $\tau = \left\lceil\log_{\opt_1} \frac{C}{\varepsilon}\right\rceil$ with $C = F(\opt_1)-F(\opt_2) + 2\sqrt{F(\opt_1)-F(\opt_2)}$. Then, we have
\begin{align*}
    \pi_0 = \left(\frac{(\opt_1)^{\tau+1}}{1-\opt_2} + \frac{(\opt_1)^{\tau+1}-1}{\opt_1-1}\right)^{-1}
\end{align*}
which then implies
\begin{align*}
    \sum_{i=0}^\tau \pi_i &= \left(1 + \frac{\opt_1-1}{(1-\opt_2)(1-(\opt_1)^{-\tau-1})}\right)^{-1} \overset{(a)}{\geq} \left(1 + \frac{\opt_1-1}{(1-\opt_2)(1-\varepsilon/C)}\right)^{-1} \\
    &= \frac{1-\opt_2-\varepsilon(1-\opt_2)/C}{\opt_1-\opt_2 - \varepsilon(1-\opt_2)/C} \overset{(b)}{\geq} \left(p^\star - \frac{\varepsilon p^\star}{C}\right)\left(1+\frac{\varepsilon p^\star}{C}\right) = p^\star -\frac{\varepsilon p^\star (1-p^\star)}{C} - \frac{\varepsilon^2 (p^\star)^2}{C^2} \\ &\overset{(c)}{\geq} p^\star - \frac{\varepsilon}{4C} - \frac{\varepsilon^2}{C^2} \overset{(d)}{\geq} p^\star - \frac{\varepsilon}{F(\opt_1)-F(\opt_2)}, \numberthis \label{eq:lower_bound_pi_tau}
\end{align*}
where $(a)$ follows as $(\opt_1)^{-\tau-1} \leq (\opt_1)^{\log_{\opt_1} C/\varepsilon} = \varepsilon / C$. Next, $(b)$ holds by noting that $p^\star = (1-\opt_2)/(\opt_1-\opt_2)$ as $p^\star\opt_1 + (1-p^\star)\opt_2 = 1$, and $1/(1-x) \geq 1+x$ for $x \geq 0$. Further, $(c)$ follows as $p^\star \in (0, 1)$. Lastly, $(d)$ holds by noting that $C \geq \max\{F(\opt_1)-F(\opt_2), 2\sqrt{F(\opt_1)-F(\opt_2)}\}$. Now, we get

\begin{equation*}
\begin{aligned}
    \mathbb{E}_\pi[F(\lambda(\bar q))] &= F(\opt_1)\sum_{i=0}^\tau\pi_i + F(\opt_2)\sum_{i=\tau+1}^\infty\pi_i\\
    &\overset{\eqref{eq:lower_bound_pi_tau}}\geq F(\opt_1)\left(p^\star-\frac{\varepsilon}{F(\opt_1)-F(\opt_2)}\right) + F(\opt_2)\left(1-p^\star+\frac{\varepsilon}{F(\opt_1)-F(\opt_2)}\right) \\
    &= F(\opt_1)p^\star + F(\opt_2)(1-p^\star) - \varepsilon
    = F^\star-\varepsilon,
\end{aligned}
\end{equation*}
which implies $R(\lambda)\geq \varepsilon$, indicates our policy is a feasible solution of \eqref{eq:constraint}. 

Since 
\begin{align*}
    \sum_{i=0}^\tau \pi_i = \frac{(\opt_1)^{\tau+1}-1}{\opt_1-1}\pi_0 \leq p^\star,
\end{align*}
we get
\begin{align*}
    q^\star &\leq \sum_{i=0}^{\tau}i\pi_i + \sum_{i=\tau+1}^\infty i\pi_0(\opt_1)^{\tau+1} (\opt_2)^{i-\tau-1} 
    \leq \tau p^\star + \pi_0(\opt_1)^{\tau+1} \left(\frac{\tau+1}{1-\opt_2}+\frac{\opt_2}{(1-\opt_2)^2}\right) \\
    &\leq \tau p^\star + \left(p^\star(\opt_1-1)+\pi_0\right) \left(\frac{\tau+1}{1-\opt_2}+\frac{\opt_2}{(1-\opt_2)^2}\right) \\
    &\overset{(a)}{=} \tau + \pi_0\left(\frac{\tau+1}{1-\opt_2}+\frac{\opt_2}{(1-\opt_2)^2}\right) + p^\star(\opt_1-1)\left(\frac{1}{1-\opt_2}+\frac{\opt_2}{(1-\opt_2)^2}\right) \\
    &\overset{(b)}{=} \left\lceil\log_{\opt_1} \frac{C}{\varepsilon}\right\rceil \left(1+o_\varepsilon(1)\right),
\end{align*}
where $(a)$ follows as $p^\star = (1-\opt_2)/(\opt_1-\opt_2)$ and $(b)$ holds as $\pi_0 = O(\varepsilon)$ by Lemma~\ref{pi_0<=sth}.
\end{proof}

\subsubsection{Proof of the lower bound of \(q^\star\) in Theorem~\ref{general:2 supports}}
As a warm up example, consider the case where $F(x)=x$, for which $R(\lambda)= \varepsilon$, i.e., $\mathbb{E}_\pi [\lambda(\bar q)] = \mathbb{E}_\pi [F(\lambda(\bar q))]=1-\varepsilon$. Then $\pi_0 = 1-\mathbb{E}_\pi [\lambda(\bar q)]= \varepsilon$. Since $\lambda(q) \leq \lambda_{\max}$, we conclude that $\pi_i \leq \lambda_{\max}^i \pi_0 = \lambda_{\max}^i \varepsilon$, which means the stationary distribution increases at most exponentially with $i$, and consequently, $\mathbb{E}[\bar q] = \Theta\left(\log_{\lambda_{\max}} \frac{1}{\varepsilon}\right)$. We now generalize this argument to all reward functions $F$ that do not satisfy \eqref{eq:concave_like}.

\begin{proof}
Fix a control policy $\lambda$ satisfying $R(\lambda) \leq \varepsilon$. Then, by Lemma ~\ref{pi_0<=sth}, we must have $\pi_0\leq (\opt_1-\opt_2)\varepsilon / (F(\opt_1)-F(\opt_2))$, which implies
\begin{equation}\label{eq:opt for 2 points}
\begin{aligned}
    q^\star&\geq \min\,\, \mathbb{E}[\bar q]=\sum_{i=0}^\infty i\pi_i
    \quad \text{s.t.}\,\,\pi_0\leq \frac{\opt_1-\opt_2}{F(\opt_1)-F(\opt_2)}\varepsilon, \\
    &\text{ and }\pi_{i+1}\leq \lambda_{\max}\pi_i \,\, \forall i\geq 0,\quad
    \sum_{i=0}^\infty \pi_i=1.
\end{aligned}
\end{equation}
where $q^\star$ is the optimal objective of \eqref{eq:constraint}. The optimal solution of \eqref{eq:opt for 2 points} achieved at 
\begin{equation*}
    \hat{\pi_i} = \begin{cases}
        \frac{\opt_1-\opt_2}{F(\opt_1)-F(\opt_2)}\varepsilon \lambda_{\max}^i, &\quad i\leq B\\
        1 - \sum_{i=0}^B \frac{\opt_1-\opt_2}{F(\opt_1)-F(\opt_2)}\varepsilon \lambda_{\max}^i, &\quad i=B+1\\
        0, &\quad i\geq B+2
    \end{cases}
\end{equation*}
where $B$ is the threshold satisfying $\sum_{i=0}^B \frac{\opt_1-\opt_2}{F(\opt_1)-F(\opt_2)}\varepsilon \lambda_{\max}^i\leq 1<\sum_{i=0}^{B+1}\frac{\opt_1-\opt_2}{F(\opt_1)-F(\opt_2)}\varepsilon \lambda_{\max}^i$.  

Hence we get
\begin{align}
        \sum_{i=0}^B \frac{\opt_1-\opt_2}{F(\opt_1)-F(\opt_2)}\varepsilon \lambda_{\max}^i\leq 1<\sum_{i=0}^{B+1}\frac{\opt_1-\opt_2}{F(\opt_1)-F(\opt_2)}\varepsilon \lambda_{\max}^i \notag\\
        \sum_{i=0}^B \lambda_{\max}^i \leq\frac{F(\opt_1)-F(\opt_2)}{(\opt_1-\opt_2)\varepsilon}< \sum_{i=0}^{B+1} \lambda_{\max}^i\notag\\
        \frac{\lambda_{\max}^{B+1}-1}{\lambda_{\max}-1} \leq\frac{F(\opt_1)-F(\opt_2)}{(\opt_1-\opt_2)\varepsilon}< \frac{\lambda_{\max}^{B+2}-1}{\lambda_{\max}-1} \label{eq:loglambda_max}\\
        \lambda_{\max}^{B+1}\leq \frac{(F(\opt_1)-F(\opt_2))(\lambda_{\max}-1)}{(\opt_1-\opt_2)\varepsilon}+1 < \lambda_{\max}^{B+2}\notag,
\end{align}
which indicates
\begin{equation}\label{eq:B}
    B =  \frac{\log\frac{1}{\varepsilon}}{\log\lambda_{\max}}(1 + o_\varepsilon(1))
\end{equation}

Note that \(\lambda(q)\geq 1\) for all \(q=0,\cdots, B-1\), which readily implies \(q^\star\geq B/2\) by coupling, where the right hand side is the expected queue length for the CTMC with arrival \(\lambda(q)=1, \) when \(q=0,\cdots, B-1\) and \(\lambda(q)=0\) for all \(q\geq B\). Therefore by \eqref{eq:B}, we obtain that 
\begin{align*}
    q^\star \geq (1+o_\varepsilon(1)) \frac{\log 1/\varepsilon}{ 2\log \lambda_{\max}}.
\end{align*}
A more fine-tuned calculations allows us to eliminate the factor of 2 from the denominator:
\begin{align*}
    q^\star &\geq \sum_{i=0}^\infty i\hat{\pi_i}
    =\frac{\opt_1-\opt_2}{F(\opt_1)-F(\opt_2)}\varepsilon\sum_{i=0}^Bi\lambda_{\max}^i + (B+1)(1-\frac{\opt_1-\opt_2}{F(\opt_1)-F(\opt_2)}\varepsilon\sum_{i=0}^B\lambda_{\max}^i)\\
    &= \frac{\opt_1-\opt_2}{F(\opt_1)-F(\opt_2)}\varepsilon\left(\frac{B\lambda_{\max}^{B+1}}{\lambda_{\max}-1}- \frac{\lambda_{\max}^{B+1}-\lambda_{\max}}{(\lambda_{\max}-1)^2}\right) + B+1-  \frac{\opt_1-\opt_2}{F(\opt_1)-F(\opt_2)}\varepsilon (B+1)\frac{\lambda_{\max}^{B+1}-1}{\lambda_{\max}-1}\\
    &= B+1+\frac{\opt_1-\opt_2}{F(\opt_1)-F(\opt_2)}\varepsilon\left(\frac{-\lambda_{\max}^{B+1}+B+1}{\lambda_{\max}-1}-\frac{\lambda_{\max}^{B+1}-\lambda_{\max}}{(\lambda_{\max}-1)^2}\right)\\
    &\overset{(a)}{\geq} B+1 + \frac{\lambda_{\max}-1}{\lambda_{\max}^{B+1}-1}\varepsilon\left(\frac{-\lambda_{\max}^{B+1}+B+1}{\lambda_{\max}-1}-\frac{\lambda_{\max}^{B+1}-\lambda_{\max}}{(\lambda_{\max}-1)^2}\right)\\
    &= B+1+\frac{\varepsilon}{\lambda_{\max}^{B+1}-1}\left(-\lambda_{\max}^{B+1}+B+1-\frac{\lambda_{\max}^{B+1}-\lambda_{\max}}{\lambda_{\max}-1}\right)\\
    &= B+1 - \varepsilon(\frac{\lambda_{\max}}{\lambda_{\max}-1}+o_\varepsilon(1))
    = \frac{\log\frac{1}{\varepsilon}}{\log\lambda_{\max}} + o(\log\frac{1}{\varepsilon}),
\end{align*}

where (a) follows from \eqref{eq:loglambda_max}, and the the term inside brackets is negative.

\end{proof}

\subsection{Proof of Lemma~\ref{lemma4.3}}\label{sec6_line}
\begin{proof}
From Proposition~\ref{prop:how_many_supp_point}, $p^\star=\frac{1-\opt_2}{\opt_1-\opt_2}$, so
    \begin{equation}\label{Fstar_equation}
        F^\star = F(\opt_1)p^\star+F(\opt_2)(1-p^\star) = \frac{F(\opt_1)-F(\opt_2)}{\opt_1-\opt_2}+\frac{\opt_1 F(\opt_2)-\opt_2 F(\opt_1)}{\opt_1-\opt_2}.
    \end{equation}
Let $G(x) = \frac{F(\opt_1)-F(\opt_2)}{\opt_1-\opt_2}x +\frac{\opt_1 F(\opt_2)-\opt_2 F(\opt_1)}{\opt_1-\opt_2}$, so $G(\opt_1) = F(\opt_1), G(\opt_2) = F(\opt_2)$ and $G(1) = F^\star$. Suppose there exist some $x'$ such that $F(x') > G(x')$, we will show that this assumption allows us to construct a feasible distribution whose expected reward is strictly greater than \(F^\star\), contradicting \eqref{eq:def Fstar}.

If $x'=1$, then $F(1)>G(1) = F^*$, contradiction. 

If $x'>1$, then consider \(\alpha(\{x'\}) = \frac{1-\opt_2}{x'-\opt_2}\), and \(\alpha(\{\opt_2\}) = \frac{x'-1}{x'-x_2^\star}\), so \(\mathbb{E}_\alpha[X]=x'\alpha(\{x'\}) + \opt_2\alpha(\{\opt_2\})=1\), which satisfies the constraint of \eqref{eq:def Fstar}.
Therefore we have,
\begin{align*}
    F(x')\alpha(\{x'\}) + F(\opt_2)\alpha(\{\opt_2\}) &= F(x')\frac{1-\opt_2}{x'-\opt_2} + F(\opt_2)\frac{x'-1}{x'-x_2^\star}\\
    \overset{(a)}{>}& \left(\frac{F(\opt_1)-F(\opt_2)}{\opt_1-\opt_2}x'+\frac{\opt_1F(\opt_2)-\opt_2F(\opt_1)}{\opt_1-\opt_2}\right)\frac{1-\opt_2}{x'-\opt_2} + F(\opt_2)\frac{x'-1}{x'-x_2^\star}\\
    =& \frac{1-\opt_2}{\opt_1-\opt_2}F(\opt_1) + \frac{(1-\opt_2)(\opt_1-x')+(x'-1)(\opt_1-\opt_2)}{(x'-\opt_2)(\opt_1-\opt_2)}F(\opt_2)\\
    =& \frac{1-\opt_2}{\opt_1-\opt_2}F(\opt_1) +\frac{\opt_1-1}{\opt_1-\opt_2}F(\opt_2) = F^\star,
\end{align*}
where (a) follows from \(F(x')>G(x')\).

If $x'<1$, then consider \(\alpha(\{\opt_1\}) = \frac{1-x'}{\opt_1-x'}\), and \(\alpha(\{x'\}) = \frac{\opt_1-1}{\opt_1-x'}\), we can prove 
\(
F(\opt_1)\alpha(\{\opt_1\})+F(x')\alpha(\{x'\}) > F^\star.
\)

Above all $F(x)\leq G(x)= \frac{F(\opt_1)-F(\opt_2)}{\opt_1-\opt_2}x+\frac{\opt_1 F(\opt_2)-\opt_2 F(\opt_1)}{\opt_1-\opt_2}$ for all $x\in [0,\lambda_{\max}]$.
\end{proof}
\subsection{Proof of Lemma~\ref{pi_0<=sth}}\label{sec6_pi0}

\begin{proof}
    For any policy $\lambda$, let $\widetilde{C} = \{\lambda(q) : q\in \mathcal{S}\} = \{\widetilde{\lambda_1}, \widetilde{\lambda_2}, \cdots\} \subseteq [0, \lambda_{\max}]$ is the image of arrival rate function $\lambda$. $X$ is a random variable on $\widetilde{C}$ with probability measure $\alpha$ satisfying $P(X=\tilde{\lambda_i}) = \alpha_i$ for all $i$ and $\sum_{i=1}^{|\widetilde{C}|}\alpha_i=1$.
    
    For every $\tilde{\lambda_i}\in \tilde{C}$, by Lemma~\ref{lemma4.3},we have
    \begin{equation}\label{lambdabar}
    \begin{aligned}
        F(\widetilde{\lambda_i})&\leq \frac{F(\opt_1)-F(\opt_2)}{\opt_1-\opt_2}\widetilde{\lambda_i}+\frac{\opt_1 F(\opt_2)-\opt_2 F(\opt_1)}{\opt_1-\opt_2}\\
        \widetilde{\lambda_i} &\geq \left(F(\widetilde{\lambda_i}) - \frac{\opt_1 F(\opt_2)-\opt_2 F(\opt_1)}{\opt_1-\opt_2}\right) \cdot \frac{\opt_1-\opt_2}{F(\opt_1)-F(\opt_2)}
    \end{aligned}
    \end{equation}
    Therefore, 
\begin{align*}
    \mathbb E_\alpha[X] &= \sum_{i=1}^{|\widetilde{C}|} \widetilde{\lambda_i}\alpha_i
    \overset{(a)}{\geq} \sum_{i=1}^{|\widetilde{C}|} \left(F(\widetilde{\lambda_i})\alpha_i - \frac{\opt_1 F(\opt_1)-\opt_2 F(\opt_1)}{\opt_1-\opt_2}\alpha_i\right) \cdot \frac{\opt_1-\opt_2}{F(\opt_1)-F(\opt_2)}\\
    &= \left(\sum_{i=1}^{|\widetilde{C}|} F(\widetilde{\lambda_i})\alpha_i-\frac{\opt_1 F(\opt_2)-\opt_2 F(\opt_1)}{\opt_1-\opt_2}\sum_{i=1}^{|\widetilde{C}|}\alpha_i\right)\cdot \frac{\opt_1-\opt_2}{F(\opt_1)-F(\opt_2)}\\
    &\overset{(b)}\geq \left(F^\star-\varepsilon-\frac{\opt_1 F(\opt_2)-\opt_2 F(\opt_1)}{\opt_1-\opt_2}\right)\cdot \frac{\opt_1-\opt_2}{F(\opt_1)-F(\opt_2)}\\
    &\overset{(c)}= \left(\frac{F(\opt_1)-F(\opt_2)}{\opt_1-\opt_2}+\frac{\opt_1 F(\opt_2)-\opt_2 F(\opt_1)}{\opt_1-\opt_2}-\varepsilon-\frac{\opt_1 F(\opt_2)-\opt_2 F(\opt_1)}{\opt_1-\opt_2}\right)\cdot \frac{\opt_1-\opt_2}{F(\opt_1)-F(\opt_2)}\\
    &= 1-\frac{\opt_1-\opt_2}{F(\opt_1)-F(\opt_2)}\varepsilon,
\end{align*}

where (a) is by applying \eqref{lambdabar}, (b) is by applying $F^\star-\sum_{i=0}^{|\widetilde{C|}}F(\widetilde{\lambda_i})\alpha_i\leq \varepsilon$, and (c) is by applying \eqref{Fstar_equation}.

By Proposition~\ref{lemma3.2}, $\mathbb{E}_\pi[\lambda(\bar q)] = \mathbb{E}_\alpha[X]\geq 1-\frac{\opt_1-\opt_2}{F(\opt_1)-F(\opt_2)}\varepsilon$, and we also have $\pi_0 = 1-\mathbb E_\pi[\lambda(\bar q)] \leq \frac{\lambda_1^\star-\lambda_2^\star}{F(\lambda_1^\star)-F(\lambda_2^\star)}\varepsilon$ (see Section~\ref{proof for prop 4.1}, Equation \eqref{idle time}).
\end{proof}

\section{} \label{appendix:Proof for 5.3}

\subsection{A General $F$ satisfying \eqref{eq:concave_like}}\label{concavelike_nonconcave}
Now we extend the lower bound established in the previous sub-section to general functions $F$ satisfying \eqref{eq:concave_like}, that may not be strictly concave.

Note that \eqref{eq:concave_like} immediately implies that $F(x)$ is always less than or equal to the straight line given by $F'(1)(x-1)+F(1)$ which passes through $(1,F(1))$. We make this observation formal in the following lemma: 
\begin{lemma}\label{nonconcave line}
    For any function $F$ satisfying Equation~\ref{eq:concave_like}, Assumption~\ref{assump2} and~\ref{assump4}, we have 
    $
    F(x)\leq F'(1)(x-1) + F(1)
    $
    for all $x\in [0,\lambda_{\max}]$. 
\end{lemma}
The Proof of Lemma~\ref{nonconcave line} is deferred to Appendix~\ref{proof_of_nonconcaveline}. Now, one of the following four cases may arise depending on whether 

$F(x)\leq F'(1)(x-1) + F(1)$ is a strict inequality or not:

\textbf{Case I:} There exists $x_1 \in (1, \lambda_{\max}]$ and $x_2 \in [0, 1)$ such that $F(x_k) = F'(1)(x_k-1) + F(1)$ for $k \in \{1, 2\}$. Note that such a case is not possible as it contradicts \eqref{eq:concave_like}. The reason is we can choose $p = \frac{1-x_2}{x_1-x_2}\in (0,1)$, such that 
\begin{align*}
    pF(x_1)+(1-p)F(x_2) &= \frac{1-x_2}{x_1-x_2}F(x_1)+\frac{x_1-1}{x_1-x_2}F(x_2)\\
    &= \frac{1-x_2}{x_1-x_2}\left(F'(1)(x_1-1)+F(1)\right) +\frac{x_1-1}{x_1-x_2}\left(F'(1)(x_2-1)+F(1)\right)\\
    &= F(1),
\end{align*}
which violate \eqref{eq:concave_like}, so Case I belongs to the category of Section~\ref{P_O_S}, where there exist a probabilistic optimal solution achieves $F^\star$.

\textbf{Case II:} For all $x \in [0, \lambda_{\max}]\backslash \{1\}$, we have $F(x) < F'(1)(x-1) + F(1)$. In this case, we can strengthen Lemma~\ref{nonconcave line} to show that there exists a concave quadratic function $G(x)$ such that $F(x)\leq G(x)\leq F'(1)(x-1)+F(1)$, as illustrated in Fig.~\ref{fig:quadratic_main_part}. The following lemma formalizes this claim:
\begin{lemma}\label{quadratic}
    If $F(x)< F'(1)(x-1) + F(1)$ for all $x\neq 1$, and satisfying Assumption~\ref{eq:smooth_F}, then there exists a concave quadratic function $G(x)$ with $G(1)=F(1)$, $G'(1)=F'(1)$ satisfying $F(x)\leq G(x)\leq F'(1)(x-1)+F(1)$ for all $x\in [0,\lambda_{\max}]$.
\end{lemma}
The Proof of Lemma~\ref{quadratic} is deferred to Appendix~\ref{proof_of_quad}. Now that we have constructed a strictly concave $G$ that is always above $F$, we can then use the result of Proposition~\ref{prop:strictly_concave} for $G$ to conclude this case. In particular, 
as $F\leq G$, so under the same control policy $\lambda$, we have $\mathbb{E}_\pi[G(\lambda(\bar q))]\geq \mathbb{E}_\pi[F(\lambda(\bar q))]$. Also, as $G$ is a strictly concave function, it satisfies \eqref{eq:concave_like}. Thus, by Proposition ~\ref{prop:how_many_supp_point}, we get $F^\star=G^\star=F(1)$. So, we have
\begin{align} \label{GF_eps}
     G^\star-\mathbb{E}_\pi[G(\lambda(\bar q))] \leq F^\star-\mathbb{E}_\pi[F(\lambda(\bar q))],
\end{align}
which indicates for any policy satisfying $F^\star-\mathbb{E}[F_\pi(\lambda(\bar q))] \leq \varepsilon$, then under the same policy, $G^\star-\mathbb{E}_\pi[G(\lambda(\bar q))] \leq \varepsilon$. In addition, by Proposition~\ref{prop:strictly_concave} ($G$ exhibits all derivatives with $G^{(3)} \equiv 0$), under any such policy, $\mathbb{E}_\pi[\bar{q}] \geq C_1/\sqrt{\varepsilon}$, which completes the proof.

\textbf{Case III:} There exists $x_0 \in [0, 1)$ such that $F(x_0)= F'(1)(x_0-1) + F(1)$ and $F(x)< F'(1)(x-1)+F(1)$ for all $x\in (1,\lambda_{\max}]$. In this case, similar to Lemma~\ref{quadratic}, we construct a function $G(x)$ which is concave and quadratic in $[1,\lambda_{\max}]$ and $G(x) = F'(1)(x-1) + F(1)$ for $x\in [0,1]$, which satisfies $G(1)=F(1), G'(1)=F'(1)$ and $F(x)\leq G(x)\leq F'(1)(x-1)+F(1)$ for all $x\in [0,\lambda_{\max}]$. 
Now we show $G$ satisfies \eqref{eq:concave_like}.
For any $x_1,x_2 \in [0, \lambda_{\max}] \backslash \{1\}, p \in (0, 1)$ satisfying $px_1+(1-p)x_2=1$, we have $(x_1-1)(x_2-1)<0$. Without loss of generality, assume $x_1\in (1,\lambda_{\max}], x_2\in [0,1)$, then $G(x_1)<F'(1)(x_1-1)+F(1)$ and $G(x_2) = F'(1)(x_2-1)+F(1)$. Therefore
\begin{align*}
    pG(x_1) + (1-p)G(x_2) &< p\left(F'(1)(x_1-1)+F(1)\right) + (1-p)\left(F'(1)(x_2-1)+F(1)\right)
    \\&=F(1)
    = G(1).
\end{align*}
Thus, $G$ satisfies \eqref{eq:concave_like}.
So, by Proposition~\ref{prop:how_many_supp_point}, we get $F^\star=G^\star=F(1)$, and \eqref{GF_eps} also hold in \text{Case III}. Also $\lim\limits_{x\to 1^+}G''(x)<0$ and $G^{(3)}\equiv 0$ for all \(x>1\) 
, therefore by Proposition~\ref{prop:strictly_concave}, $\mathbb{E}_\pi[\bar{q}] \geq C_1/\sqrt{\varepsilon}$ holds for any policy satisfying $F^\star-\mathbb{E}[F_\pi(\lambda(\bar q))] \leq \varepsilon$.

\textbf{Case IV:} There exists $x_0 \in (1,\lambda_{\max}]$ such that $F(x_0)= F'(1)(x_0-1) + F(1)$ and $F(x)< F'(1)(x-1)+F(1)$ for all $x\in [0,1)$. Similar to Case III, we construct a function $G(x)$ which is concave and quadratic in $[0,1)$ and $G(x) = F'(1)(x-1) + F(1)$ for $x\in [1,\lambda_{\max}]$, with $G(1)=F(1), G'(1)=F'(1)$ and $F(x)\leq G(x)\leq F'(1)(x-1)+F(1)$ for all $x\in [0,\lambda_{\max}]$. Moreover, similar to Case III, one can also show that $G$ satisfies \eqref{eq:concave_like}. We omit the details here for brevity.
Thus, by Proposition~\ref{prop:how_many_supp_point}, $F^\star=G^\star=F(1)$  and \eqref{GF_eps} also hold. In this case, $\lim\limits_{x\to 1^+}G''(x)=0$, so we cannot apply Proposition~\ref{prop:strictly_concave}. However, since $\lim\limits_{x\to 1^-}G''(x)<0$ and $G^{(3)}\equiv 0$ for all \(x<1\)
, we can apply the following Proposition~\ref{prop:strictly_concave_another_side} to show that $\mathbb{E}_\pi[\bar{q}] \geq C_1/\sqrt{\varepsilon}$ holds for any policy satisfying $F^\star-\mathbb{E}[F_\pi(\lambda(\bar q))] \leq \varepsilon$.
\begin{proposition} \label{prop:strictly_concave_another_side}
    Let $F: [0, \lambda_{\max}] \rightarrow \mathbb{R}_+$ be a continuously differentiable concave function. In addition, $F$ is thrice continuously differentiable for $x \in [0,1)$ with $\lim\limits_{x \rightarrow 1^-} F''(x) < 0$ and $\sup\limits_{x < 1}|F^{(3)}(x)| \leq M$, then, there exists $C_1, \varepsilon_0 > 0$ (depending on $F$), such that $q^\star \geq C_1/\sqrt{\varepsilon}$ for all $\varepsilon\leq \varepsilon_0$.
\end{proposition}

\begin{proof}
Fix a control policy $\lambda$ and let $\{\pi_i\}_{i=0}^\infty$ be the stationary distribution with $\pi_i = 0$ for all $i \notin \mathcal{S}$. For a given $j \in \mathcal{S}$, we start by upper bounding $\pi_j$. Since $F$ is concave, for any $c\in [0,\lambda_{\max}]$, we have:

\begin{align*}
    F(1)-\varepsilon \leq{}& \mathbb{E}_\pi[F(\lambda(\bar q))] = \sum_{i\in \mathcal{S}}F(\lambda_i)\pi_i = \sum_{i\in\mathcal{S}} F(\frac{\pi_{i+1}}{\pi_i})\pi_i\\
    \overset{(a)}\leq{}& \sum_{i=0}^j \left(F(1)+F'(1)(\frac{\pi_{i+1}}{\pi_i}-1)\right)\pi_i + \sum_{i\in\mathcal{S}\backslash[j]}^\infty \left(F(c)+F'(c)(\frac{\pi_{i+1}}{\pi_i}-c)\right)\pi_i\\
    ={}& \sum_{i=0}^j\left((F(1)-F'(1))\pi_i + F'(1)\pi_{i+1}\right)+\sum_{i\in \mathcal{S}\backslash[j]}\left((F(c)-cF'(c))\pi_i+F'(c)\pi_{i+1}\right)\\
    ={}& \left(F(1)-F'(1)\right)\pi_0 + \sum_{i=1}^jF(1)\pi_i + \left(F(c)-cF'(c)+F'(1)\right)\pi_{j+1}\\
    &+ \sum_{i\in \mathcal{S}\backslash[j+1]}^\infty\left(F(c)+(1-c)F'(c)\right)\pi_i\\
    ={}& \left(F(1)-F'(1)\right)\pi_0 + \sum_{i=1}^jF(1)\pi_i + \left(-F'(c)+F'(1)\right)\pi_{j+1}\\
    &+ \sum_{i\in \mathcal{S}\backslash[j]}\left(F(c)+(1-c)F'(c)\right)\pi_i,
\end{align*}
where step (a) uses concavity to upper bound $F$ by a tangent at $c$ when $i>j$, and by a tangent at 1 when $i\leq j$. Therefore we have:
\begin{align*}
    \left(F'(c)-F'(1)\right)\pi_{j+1} &\leq \varepsilon -F'(1)\pi_0 + \sum_{i\in \mathcal{S}\backslash[j]}\left(F(c)+(1-c)F'(c)-F(1)\right)\pi_i\\
     &\overset{(b)}\leq \varepsilon + \sum_{i\in \mathcal{S}\backslash[j]}\left(F(c)+(1-c)F'(c)-F(1)\right)\pi_i,
\end{align*}
where (b) is by $\pi_0\geq 0$.

Now, let $\delta = \sqrt{\frac{2\varepsilon}{-F''(1)}}$ and set $c=1-\delta$. Also, denote by $F''(1) = \lim\limits_{x\to 1^-}F''(x)$. Then, we get
\begin{align*}
    &F(c)+(1-c)F'(c)-F(1) = F(1-\delta) + \delta F'(1-\delta) -F(1)\\
    \leq{}& F(1)-F'(1)\delta + \frac{F''(1)}{2}\delta^2 + \frac{M}{6}\delta^3 + \delta\left(F'(1)-F''(1)\delta + \frac{M}{2}\delta^2\right)-F(1)\\
    ={}& -\frac{F''(1)}{2}\delta^2 + \frac{2M}{3}\delta^3,
\end{align*}
and $F'(c)-F'(1) = F'(1-\delta)-F'(1) \geq -F''(1)\delta -\frac{M}{2}\delta^2$.
Therefore we have:
\begin{align*}
    (-F''(1)\delta-\frac{M}{2}\delta^2)\pi_{j+1} &\leq \varepsilon + \sum_{i\in \mathcal{S}\backslash[j]} (-\frac{F''(1)}{2}\delta^2 + \frac{2M}{3}\delta^3)\pi_i \leq \varepsilon + \frac{-F''(1)}{2}\delta^2 + \frac{2M}{3}\delta^3,
\end{align*}
which we get the same inequality as \eqref{cite_in_appendix}, and the remaining part is same as the proof of Proposition~\ref{prop:strictly_concave}.
\end{proof}

\subsection{Proof of Lemma~\ref{nonconcave line}}\label{proof_of_nonconcaveline}
\begin{proof}
    Proof by contradiction. Here are two cases:
    \begin{itemize}
        \item Case I: There exist $ x_1\in (1,\lambda_{\max}]$ such that $F(x_1) > F'(1)(x_1-1)+F(1)$
        \item Case II: There exist $ x_2\in [0,1)$ such that $F(x_2) > F'(1)(x_2-1)+F(1)$
    \end{itemize}

For Case I, $\frac{F(x_1)-F(1)}{x_1-1} > F'(1) $, and 
by Assumption~\ref{assump2}, $F'(x)$ is continuous, so $\exists \delta>0$ such that $F'(x) \leq \frac{F(x_1)-F(1)}{x_1-1}$ for all $x\in [1-\delta,1]$. Hence we have
\[
F(1)-F(1-\delta) = \int_{1-\delta}^1F'(x)dx\leq \delta\frac{F(x_1)-F(1)}{x_1-1}
\]
which means $F(1-\delta)\geq F(1)-\delta\frac{F(x_1)-F(1)}{x_1-1}$. Then choose $p = \frac{\delta}{x_1-1+\delta}\in (0,1)$, and note that
\begin{align*}
    pF(x_1)+(1-p)F(1-\delta) &= \frac{\delta}{x_1-1+\delta}F(x_1)+\frac{x_1-1}{x_1-1+\delta}F(1-\delta)\\
    &\geq \frac{\delta}{x_1-1+\delta}F(x_1)+\frac{x_1-1}{x_1-1+\delta}\left(F(1)-\delta\frac{F(x_1)-F(1)}{x_1-1}\right) \\
    &= F(1),
\end{align*}
which violate \eqref{eq:concave_like}.

For Case II, just follow the idea similar to the second case, by choosing $\delta$ such that $F'(x) \geq \frac{F(1)-F(x_2)}{1-x_2}$ for all $x\in [1, 1+\delta]$, then $F(1+\delta)\geq F(1)+\delta \frac{F(1)-F(x_2)}{1-x_2}$, and choose \(p = \frac{1-x_2}{1+\delta-x_2}\in (0,1)\), therefore
\begin{align*}
    pF(1+\delta) + (1-p)F(x_2) &= \frac{1-x_2}{1+\delta-x_2}F(1+\delta) + \frac{\delta}{1+\delta-x_2}F(x_2) \\
    &\geq \frac{1-x_2}{1+\delta-x_2} \left(F(1)+\delta \frac{F(1)-F(x_2)}{1-x_2}\right) + \frac{\delta}{1+\delta-x_2}F(x_2) \\
    &=F(1),
\end{align*}
which violate \eqref{eq:concave_like}.
\end{proof}
\subsection{Proof of Lemma~\ref{quadratic}} \label{proof_of_quad}
\begin{proof}
For each $\alpha \in \mathbb{R}$, consider the quadratic
\[
G_\alpha(x) := F(1) + F'(1)(x-1) + \frac{\alpha}{2}(x-1)^2.
\]
Then $G_\alpha(1)=F(1)$ and $G_\alpha'(1)=F'(1)$, and $G_\alpha$ is concave iff $\alpha<0$. Let $T(x) = F(1) + F^\prime(1)(x-1)$ and define
\[
H(\alpha,x) := G_\alpha(x) - F(x)
= \frac{\alpha}{2}(x-1)^2 + T(x) - F(x).
\]

\textbf{Step 1: Control of $H(0,\cdot)$ near $x=1$:} Consider the function $
\varphi(x) := H(0,x) = T(x) - F(x).$ We have
$
\varphi(1) = 0, \varphi'(1) = 0,
$
because $T$ is the tangent of $F$ at $x=1$. Moreover,
\[
\varphi''(1) = T''(1) - F''(1) = -F''(1) > 0
\]
by the assumption $F''(1)<0$. Since $\varphi$ is twice continuously differentiable and $\varphi''(1)>0$, by Taylor's theorem, there exist constants $\delta>0, c_1>0$ such that
\[
\varphi(x) \ge c_1 (x-1)^2 \quad \forall x \in [0,\lambda_{\max}] \text{ with } |x-1|\le \delta.
\]
Equivalently,
\[
H(0,x) \ge c_1 (x-1)^2 \quad \forall x \in [0,\lambda_{\max}] \text{ with } |x-1|\le \delta.
\]
For any $\alpha \in \mathbb{R}$ and any $x \in [0, \lambda_{\max}]$,
\[
H(\alpha,x) = H(0,x) + \frac{\alpha}{2}(x-1)^2.
\]
Thus, for $|x-1|\le \delta$,
\[
H(\alpha,x)
\;\ge\; c_1 (x-1)^2 + \frac{\alpha}{2}(x-1)^2
= \Bigl(c_1 + \frac{\alpha}{2}\Bigr)(x-1)^2.
\]
If we choose $\alpha$ such that $\alpha > -2c_1$, then $c_1 + \frac{\alpha}{2} > 0$, and hence
\[
H(\alpha,x) \ge 0 \quad \forall x \in [0,\lambda_{\max}] \text{ with } |x-1|\le \delta.
\]

\textbf{Step 2: Control of $H(0,\cdot)$ away from $x=1$:} By the hypothesis of the lemma, $H(0,x) = T(x) - F(x) > 0$ for all $x \neq 1$, and $H(0,1)=0$. On the compact set $K := [0,\lambda_{\max}] \setminus (1-\delta,1+\delta)$, we have $H(0,x) > 0$ and $H(0,\cdot)$ is continuous, so $c_2 := \min\limits_{x \in K} H(0,x) > 0.$ Note that $x \in K$ implies $|x-1|\ge \delta$, so
\[
H(\alpha,x) = H(0,x) + \frac{\alpha}{2}(x-1)^2
\ge c_2 + \frac{\alpha}{2}\delta^2.
\]
Thus, if $\alpha \ge -\frac{2c_2}{\delta^2}$, then $H(\alpha,x)\ge 0$ for all $x \in K$.

\textbf{Step 3: Choosing $\alpha$:} We want $G_\alpha$ to be concave, so we require $\alpha<0$. Collecting the conditions from Steps 1 and 2, set
\[
\alpha_0 := \max\!\left(-2c_1,\; -\frac{2c_2}{\delta^2}\right).
\]
Then $\alpha_0 < 0$, so the open interval $(\alpha_0,0)$ is nonempty. Choose any $\alpha \in (\alpha_0,0).$ For this choice, $G_\alpha$ is a concave quadratic, and by the estimates above, we have
\[
H(\alpha,x) = G_\alpha(x) - F(x) \ge 0 \quad \text{for all } x \in [0,\lambda_{\max}].
\]

Finally, $G_\alpha(1)=F(1)$ and $G_\alpha'(1)=F'(1)$ by construction, and $G_\alpha''(1)=\alpha<0$.  
Thus $G := G_\alpha$ satisfies all the required properties.
\end{proof}

\subsection{Proof of Theorem~\ref{thm4.13}} \label{proof of simulation in appendix}
\begin{proof}
Consider the two-arrival policy
\begin{align}\label{2price}
\lambda(q) = \begin{cases}
    1+k_1, \quad \text{ if } 0\leq q<\tau \\
    1-k_2, \quad \text{ if } q\geq \tau,
\end{cases}
\end{align}
where \(k_1\in (0,\lambda_{\max}-1]\), \(k_2\in (0,1]\), and let $\pi$ be the stationary distribution of the $M_q/M/1$ queue under this policy. Note that, we have
\begin{align}
    \mathbb{E}[\bar{q}] \geq \frac{1}{4}\left(\tau + \frac{1}{k_2}\right).  \label{eq:lower_bound_q}
\end{align}
Intuitively, the $\tau$ term arises from the fact that the stationary distribution $\pi_i$ is increasing in $i=0,1,\cdots,\tau$, so, the expected queue length must be lower bounded by a uniform distribution supported on $0,1,\cdots,\tau$. In addition, the $1/k_2$ term reflects the exponential decay of the stationary distribution in the tail. Specifically, for $i>\tau$, the stationary probabilities satisfy 
$\pi_i = (1-k_2)\pi_{i-1} \approx e^{-k_2} \pi_{i-1}$. Therefore, the expected queue length contributed by the tail is lower bounded by that of an exponential distribution with rate $k_2$.
A complete proof of \eqref{eq:lower_bound_q} is provided in Section~\ref{proof_of_Eq_tau_k2}.

If $k_2 \leq \varepsilon^{3/4}$, then \(\mathbb{E}[\bar q]\geq \frac{1}{4k_2}\geq \frac{1}{4\varepsilon^{3/4}}\geq \Omega(\sqrt{\frac{\log1/\varepsilon}{\varepsilon}})\), the proof is complete. So, let's assume that $k_2 \geq \varepsilon^{3/4}$. As $F$ is strongly concave, 
so $\exists \alpha>0$, such that $F(x) \leq F(1) + (x-1)F'(1) - \alpha(x-1)^2$. Hence, we have
    \begin{align}
        F(1+k_1) \leq F(1) + k_1 F^\prime(1) - \alpha k_1^2, \quad F(1-k_2) \leq F(1) - k_2 F^\prime(1) - \alpha k_2^2. \label{eq:strong_convexity}
    \end{align}
Now, as the regret is bounded above by $\varepsilon$, we have
\begin{align*}
    \varepsilon &\geq F(1) - F(1+k_1) \pi_{q< \tau} - F(1-k_2) \pi_{q\geq \tau} \\
    &\overset{\eqref{eq:strong_convexity}}{\geq} F^\prime(1) \left(k_2 \pi_{q\geq \tau} - k_1 \pi_{q< \tau}\right) + \alpha \left(k_1^2 \pi_{q< \tau} + k_2^2 \pi_{q\geq \tau}\right). \numberthis \label{eq:regret_lower_bound}
\end{align*}
Now, we compute $\pi_{q< \tau}$ using the detailed balance equations. For $0 \le q < \tau$ we have $\lambda(q) = 1+k_1$ and $\mu=1$, hence $\pi_{q+1} = \pi_q \lambda(q) = (1+k_1)\,\pi_q$. Iterating this recursion starting from $\pi_0$ gives
\begin{equation}
  \pi_q = \pi_0 (1+k_1)^q,
  \qquad 0 \le q < \tau.
  \label{eq:pi-left}
\end{equation}
Similarly, for $q \ge \tau$ we have $\lambda_q = 1-k_2$ and again $\mu=1$, so we get
\begin{equation}
  \pi_q = \pi_\tau (1-k_2)^{\,q-\tau} = \pi_0 (1+k_1)^\tau (1-k_2)^{\,q-\tau}, 
  \qquad q \ge \tau.
  \label{eq:pi-right}
\end{equation}
Now, we compute $\pi_0$.
\begin{align}\label{pi_0_twoarrive}
    1 = \sum_{q=0}^\infty \pi_q &= \pi_0 \left(\sum_{q=0}^{\tau-1} (1+k_1)^q + (1+k_1)^\tau\sum_{q=\tau}^\infty (1-k_2)^{q-\tau}\right)\notag  \\
    &= \pi_0\left(\frac{(1+k_1)^{\tau}-1}{k_1} + \frac{(1+k_1)^\tau}{k_2}\right).
\end{align}
Using the above expression, we get
\begin{align}\label{pi_q_less_tau}
    \pi_{q< \tau} = \sum_{q=0}^{\tau-1} \pi_q = \pi_0 \frac{(1+k_1)^{\tau}-1}{k_1} = \frac{(t-1)k_2}{k_1 t + k_2(t-1)}, \quad \text{with } t = (1+k_1)^\tau.
\end{align}
Now, substituting the above expression in \eqref{eq:regret_lower_bound}, we get
\begin{align*}
    \varepsilon &\geq \frac{k_1k_2F^\prime(1)}{k_1 t + k_2(t-1)} + \frac{\alpha k_1k_2 (k_1(t-1) + k_2t)}{k_1 t + k_2(t-1)} \\
    &\geq \frac{k_1k_2F^\prime(1)}{(k_1 + k_2)t} + \frac{\alpha k_1k_2 (k_1(t-1) + k_2t)}{k_1 t + k_2(t-1)}
\end{align*}
As $F^\prime(1), \alpha > 0$, we get
\begin{align*}
     \frac{k_1k_2F^\prime(1)}{(k_1 + k_2)t} \leq \frac{k_1k_2F^\prime(1)}{k_1 t + k_2(t-1)} \leq \varepsilon, \quad \frac{\alpha k_1k_2 (k_1(t-1) + k_2t)}{k_1 t + k_2(t-1)} \leq \varepsilon.  \numberthis \label{eq:some_upper_bounds}
\end{align*}

Using the first inequality in \eqref{eq:some_upper_bounds} and noting that $t = (1+k_1)^\tau$, we get
\begin{align*}
    \tau \geq \frac{\log \frac{k_1k_2F^\prime(1)}{(k_1 + k_2)\varepsilon}}{\log (1+k_1)} \geq \frac{\log \frac{\min\{k_1, k_2\}F^\prime(1)}{2\varepsilon}}{\log (1+k_1)}
\end{align*}
Now, let's take three cases.

\textbf{Case I $(\varepsilon^{3/4}\leq k_1 \leq k_2)$:} As $k_1 \leq k_2$, using the second expression in \eqref{eq:some_upper_bounds}, we get $\alpha k_1 k_2 \leq \varepsilon$. Thus, by \eqref{eq:lower_bound_q}, we get
\begin{align*}
    \mathbb{E}[\bar{q}] \geq \frac{\alpha k_1}{4\varepsilon} + \frac{\log \frac{k_1F^\prime(1)}{2\varepsilon}}{4\log (1+k_1)}.
\end{align*}
As $k_1 \geq \varepsilon^{3/4}$, we get
\begin{align*}
    \mathbb{E}[\bar{q}] \geq \frac{\alpha k_1}{4\varepsilon} + \frac{\log \frac{F^\prime(1)}{2\varepsilon^{1/4}}}{4\log (1+k_1)} \overset{(a)}\geq \frac{\alpha k_1}{4\varepsilon} + \frac{\log \frac{F^\prime(1)^{4}}{16\varepsilon}}{16 k_1}= \Omega\left(\sqrt{\frac{\log 1/\varepsilon}{\varepsilon}}\right).
\end{align*}
The last inequality is because choosing \(k_1 = \Theta(\sqrt{\varepsilon\log\frac{1}{\varepsilon}})\) will balance the trade-off between two terms.

\textbf{Case II $(k_1 \leq \varepsilon^{3/4}\leq k_2)$:} First note that we should have $\tau \leq \varepsilon^{-2/3}$, as otherwise, the proof is complete by \eqref{eq:lower_bound_q}. Thus, as $k_1 \tau \leq \varepsilon^{1/12}\rightarrow 0$, we have $t = (1+k_1)^\tau \leq 1+2k_1 \tau \leq 2$ for $\varepsilon$ small enough. Thus, by the first inequality of \eqref{eq:some_upper_bounds}, we get
\begin{align*}
    \varepsilon \geq \frac{k_1k_2F^\prime(1)}{k_1 t + k_2(t-1)} \geq \frac{k_2F^\prime(1)}{2 + 2k_2\tau} \geq  \frac{k_2F^\prime(1)}{2 + 2k_2\varepsilon^{-2/3}}\implies k_2 \leq \frac{2\varepsilon}{F^\prime(1)-2\varepsilon^{1/3}} \leq \frac{4\varepsilon}{F^\prime(1)},
\end{align*}
where the last two inequalities follows for $\varepsilon$ small enough. Thus, we get $k_2 = O(\varepsilon)$, which then implies $\mathbb{E}[\bar{q}] \geq \Omega(1/\varepsilon)$ by \eqref{eq:lower_bound_q}.

\textbf{Case III $(k_1 \geq k_2 \geq \varepsilon^{3/4})$:} In this case, we have
\begin{align*}
    \mathbb{E}[\bar{q}] &\geq \frac{1}{4k_2} + \frac{\log \frac{k_2F^\prime(1)}{2\varepsilon}}{4\log (1+k_1)} \geq \frac{1}{4k_2} + \frac{\log \frac{F^\prime(1)}{2\varepsilon^{1/4}}}{4\log (1+k_1)}.
\end{align*}
By the first inequality in \eqref{eq:some_upper_bounds}, we get $t\geq \frac{k_1k_2F'(1)}{(k_1+k_2)\varepsilon} = \frac{k_2F'(1)}{(1+\frac{k_2}{k_1})\varepsilon}\geq \frac{k_2F'(1)}{2\varepsilon}\geq \frac{F'(1)}{2\varepsilon^{1/4}}\geq 2$ for $\varepsilon>0$ small enough. Now, by the second inequality in \eqref{eq:some_upper_bounds}, we get

\begin{align*}
    \varepsilon \geq \frac{\alpha k_1k_2 (k_1(t-1) + k_2t)}{k_1 t + k_2(t-1)} \geq \alpha k_1 k_2 \frac{t-1}{t} \geq \frac{\alpha k_1k_2}{2}.
\end{align*}
Thus, apply the same derivation as Case I, we get
\begin{align*}
    \mathbb{E}[\bar{q}] \geq \frac{\alpha k_1}{8\varepsilon} + \frac{\log \frac{F^\prime(1)}{2\varepsilon^{1/4}}}{4\log (1+k_1)} = \Omega\left(\sqrt{\frac{\log 1/\varepsilon}{\varepsilon}}\right).
\end{align*}
This completes the proof.
\end{proof}

\subsection{Proof of \eqref{eq:lower_bound_q}} \label{proof_of_Eq_tau_k2}
Before proving \eqref{eq:lower_bound_q}, we first prove the below Lemma. 
\begin{lemma}
    For the two-arrival policy defined in \eqref{2price}, we have 
    \begin{equation}
        \mathbb{E}[\bar q]\geq \frac{1}{\tau+\frac{1}{k_2}}\left(\frac{\tau(\tau-1)}{2}+\frac{\tau}{k_2}+\frac{1-k_2}{k_2^2}\right). \label{Eq_in_lemma}
    \end{equation}
\end{lemma}
\begin{proof}
    Define a new two-arrival policy 
\begin{align*}
\lambda(q') = \begin{cases}
    1, \quad &\text{ if } 0\leq q'<\tau \\
    1-k_2, \quad &\text{ if } q'\geq \tau,
\end{cases}
\end{align*}
then by coupling, we get \(\mathbb{E}[\bar q]\geq \mathbb{E}[\bar{q'}]\). Let \(\pi'\) be the stationary distribution of the \(M_{q'}/M/1\) queue under this policy. Then we have \(\pi'_i = \pi'_0, \forall 0\leq i<\tau\) and \(\pi'_i = (1-k_2)^{i-\tau}\pi'_0, \forall i\geq \tau\), where \(\pi'_0 = (\tau+\frac{1}{k_2})^{-1}\).
Therefore
\begin{align*}
    \mathbb{E}[\bar q]\geq \mathbb{E}[\bar{q'}] &= \sum_{i=0}^\infty i\pi_i = \sum_{i=0}^{\tau-1}i\pi_0 + \sum_{i=\tau}^\infty i(1+k_2)^{i-\tau}\pi_0\\
    &= (\tau+\frac{1}{k_2})^{-1}\left(\frac{\tau(\tau-1)}{2}+\sum_{i=0}^\infty (i+\tau)(1+k_2)^i\right)\\
    &= (\tau+\frac{1}{k_2})^{-1}\left(\frac{\tau(\tau-1)}{2}+\frac{\tau}{k_2}+\frac{1-k_2}{k_2^2}\right).
\end{align*}
\end{proof}
Then we will prove \eqref{eq:lower_bound_q} by proving the RHS of \eqref{Eq_in_lemma} is \(\ge\) the RHS of \eqref{eq:lower_bound_q}.
\begin{proof}
    We aim to show
    \begin{align*}
        \frac{\tau(\tau-1)}{2}+\frac{\tau}{k_2}+\frac{1-k_2}{k_2^2} &\geq \frac{1}{4}(\tau+\frac{1}{k_2})^2\\
        \frac{k_2^2\tau(\tau-1)}{2} + k_2\tau + 1-k_2 &\geq \frac{1}{4}k_2^2\tau^2 + \frac{1}{2}k_2\tau + \frac{1}{4} \\
        \frac{1}{4}k_2^2\tau^2 -\frac{k_2^2\tau}{2} + \frac{1}{2}k_2\tau + \frac{3}{4} -k_2 &\geq 0
    \end{align*}
    Let \(x = k_2\tau\), so \(x\geq k_2\). Define \(f(x) = \frac{1}{4}x^2 - \frac{1}{2}k_2x + \frac{1}{2}x + \frac{3}{4}-k_2\), then
    \begin{align*}
        f'(x) = \frac{1}{2}x-\frac{1}{2}k_2+\frac{1}{2}\geq \frac{1}{2}>0,
    \end{align*}
    so \(f(x)\) increases in \([k_2, \infty)\), hence
    \begin{align*}
        f(x)\geq f(k_2) = -\frac{1}{4}k_2^2 - \frac{1}{2}k_2+\frac{3}{4} = -\frac{1}{4}(k_2+1)^2 + 1\geq 0,
    \end{align*}
    where the last inequality is because \(k_2\leq 1\). The proof is finished.
\end{proof}

\section{}
\label{app:mdp}

This appendix describes the numerical Markov decision process (MDP)
used to construct the benchmark in Section~\ref{sec:numerical}.
For each congestion penalty \(w>0\), we consider the scalarized
average-reward problem
\(
\max\limits_{\lambda(\cdot)}
\{\mathbb{E}_{\pi}[F(\lambda(Q))]-w\mathbb{E}_{\pi}[Q]\},
\)
where \(\pi\) is the stationary distribution induced by the
state-dependent arrival policy. Varying \(w\) generates different points
on the numerical reward--queue-length trade-off curve.

For numerical computation, we truncate the queue-length state space to
\(\mathcal{S}_N=\{0,1,\ldots,N\}\) for a sufficiently large \(N\), and
impose \(\lambda(N)=0\) at the upper boundary. We uniformize the
controlled continuous-time birth--death process at rate
\(c=\lambda_{\max}+1\). For \(1\leq q<N\), under action
\(\lambda\in[0,\lambda_{\max}]\), the transition probabilities to states
\(q+1\), \(q-1\), and \(q\) are \(\lambda/c\), \(1/c\), and
\(1-(\lambda+1)/c\), respectively. At \(q=0\), the downward transition is
absent, so the probabilities of moving to state \(1\) and remaining at
state \(0\) are \(\lambda/c\) and \(1-\lambda/c\), respectively.


We use standard average-reward policy iteration
\cite{puterman1994markov}. Since every positive state has a
positive-probability downward transition, the MDP is unichain under every
stationary policy. Together with the finiteness of the state space and
the compactness of the action space \([0,\lambda_{\max}]\), this ensures
the convergence of policy iteration in our setting
\cite{hordijk1987convergence}. 

Starting from an initial stationary policy \(\lambda^{(0)}\), the policy
evaluation step computes the average reward \(g^{(k)}\) and relative value
function \(h^{(k)}\) from
\begin{equation}
g^{(k)}+h^{(k)}(q)
=
F\!\left(\lambda^{(k)}(q)\right)-wq
+
\sum_{j\in\mathcal{S}_N}
P\!\left(j\mid q,\lambda^{(k)}(q)\right)h^{(k)}(j),
\qquad q\in\mathcal{S}_N,
\label{eq:mdp_policy_evaluation}
\end{equation}
with the normalization \(h^{(k)}(0)=0\).

The policy improvement step updates, for each \(q<N\),
\begin{equation}
\lambda^{(k+1)}(q)
\in
\arg\max_{\lambda\in[0,\lambda_{\max}]}
\left\{
F(\lambda)-wq+
\sum_{j\in\mathcal{S}_N}
P(j\mid q,\lambda)h^{(k)}(j)
\right\},
\label{eq:mdp_policy_improvement}
\end{equation}
and sets \(\lambda^{(k+1)}(N)=0\).
Since only the upward transition probability depends on \(\lambda\),
the maximization in \eqref{eq:mdp_policy_improvement} is equivalently
\[
\max_{\lambda\in[0,\lambda_{\max}]}
\left\{
F(\lambda)
+
\frac{\lambda}{c}
\bigl(h^{(k)}(q+1)-h^{(k)}(q)\bigr)
\right\}.
\]
Thus, each policy-improvement step consists only of a one-dimensional
optimization over the full continuous action interval.

We alternate policy evaluation and policy improvement until
\(\|\lambda^{(k+1)}-\lambda^{(k)}\|_{\infty}<10^{-8}\).
For each converged policy corresponding to a given \(w\), we compute its
stationary distribution using the birth--death balance equations
\(\pi_{q+1}=\pi_q\lambda(q)\), followed by normalization. We then evaluate
\(\overline{F}=\sum_{q=0}^{N}\pi_qF(\lambda(q))\) and
\(\overline{Q}=\sum_{q=0}^{N}q\pi_q\).
The corresponding point on the numerical benchmark is
\(
((F(1)-\overline{F})/F(1),\,\overline{Q}).
\)
Repeating the procedure over a range of values of \(w\) produces the MDP
benchmark reported in Section~\ref{sec:numerical}.

%% file: sample-base.bib
@article{varma2023dynamic,
  title={Dynamic pricing and matching for two-sided queues},
  author={Varma, Sushil Mahavir and Bumpensanti, Pornpawee and Maguluri, Siva Theja and Wang, He},
  journal={Operations Research},
  volume={71},
  number={1},
  pages={83--100},
  year={2023},
  publisher={INFORMS}
}

@article{varma2021throughput,
  title={Throughput optimal routing in blockchain-based payment systems},
  author={Varma, Sushil Mahavir and Maguluri, Siva Theja},
  journal={IEEE Transactions on Control of Network Systems},
  volume={8},
  number={4},
  pages={1859--1868},
  year={2021},
  publisher={IEEE}
}

@article{kim2018value,
  title={The value of dynamic pricing in large queueing systems},
  author={Kim, Jeunghyun and Randhawa, Ramandeep S},
  journal={Operations Research},
  volume={66},
  number={2},
  pages={409--425},
  year={2018},
  publisher={INFORMS}
}

@book{moment_book,
author = {Shapiro, Alexander and Dentcheva, Darinka and Ruszczyński, Andrzej},
title = {Lectures on Stochastic Programming},
publisher = {Society for Industrial and Applied Mathematics},
year = {2009},
doi = {10.1137/1.9780898718751},
address = {},
edition   = {}
}

@book{norris1998markov,
  title={Markov chains},
  author={Norris, James R},
  number={2},
  year={1998},
  publisher={Cambridge university press}
}

@article{tsitsiklis2013power,
  title={On the power of (even a little) resource pooling},
  author={Tsitsiklis, John N and Xu, Kuang},
  journal={Stochastic Systems},
  volume={2},
  number={1},
  pages={1--66},
  year={2013},
  publisher={INFORMS}
}

@article{varma2020near,
  title={Near optimal control in ride hailing platforms with strategic servers},
  author={Varma, Sushil Mahavir and Castro, Francisco and Maguluri, Siva Theja},
  journal={arXiv preprint arXiv:2008.03762},
  year={2020}
}

@article{george2001dynamic,
  title={Dynamic control of a queue with adjustable service rate},
  author={George, Jennifer M and Harrison, J Michael},
  journal={Operations research},
  volume={49},
  number={5},
  pages={720--731},
  year={2001},
  publisher={INFORMS}
}

@article{ata2006dynamic,
  title={Dynamic control of an M/M/1 service system with adjustable arrival and service rates},
  author={Ata, Bari{\c{s}} and Shneorson, Shiri},
  journal={Management Science},
  volume={52},
  number={11},
  pages={1778--1791},
  year={2006},
  publisher={INFORMS}
}

@article{adusumilli2010dynamic,
  title={Dynamic admission and service rate control of a queue},
  author={Adusumilli, Kranthi Mitra and Hasenbein, John J},
  journal={Queueing Systems},
  volume={66},
  number={2},
  pages={131--154},
  year={2010},
  publisher={Springer}
}

@article{kumar2013dynamic,
  title={Dynamic service rate control for a single-server queue with Markov-modulated arrivals},
  author={Kumar, Ravi and Lewis, Mark E and Topaloglu, Huseyin},
  journal={Naval Research Logistics (NRL)},
  volume={60},
  number={8},
  pages={661--677},
  year={2013},
  publisher={Wiley Online Library}
}

@article{xia2017optimal,
  title={Optimal control of state-dependent service rates in a MAP/M/1 queue},
  author={Xia, Li and He, Qi-Ming and Alfa, Attahiru Sule},
  journal={IEEE Transactions on Automatic Control},
  volume={62},
  number={10},
  pages={4965--4979},
  year={2017},
  publisher={IEEE}
}

@article{arapostathis2019optimal,
  title={Optimal control of Markov-modulated multiclass many-server queues},
  author={Arapostathis, Ari and Das, Anirban and Pang, Guodong and Zheng, Yi},
  journal={Stochastic Systems},
  volume={9},
  number={2},
  pages={155--181},
  year={2019},
  publisher={INFORMS}
}

@article{down2022optimal,
  title={Optimal control of energy-aware queueing systems},
  author={Down, Douglas G},
  journal={Queueing Systems},
  volume={100},
  number={3},
  pages={417--419},
  year={2022},
  publisher={Springer}
}

@article{low1974a,
  title={Optimal dynamic pricing policies for an M/M/s queue},
  author={Low, David W},
  journal={Operations Research},
  volume={22},
  number={3},
  pages={545--561},
  year={1974},
  publisher={INFORMS}
}

@article{low1974b,
  title={Optimal pricing for an unbounded queue},
  author={Low, David W.},
  journal={IBM Journal of research and Development},
  volume={18},
  number={4},
  pages={290--302},
  year={1974},
  publisher={IBM}
}

@article{chen2001state,
  title={State dependent pricing with a queue},
  author={Chen, Hong and Frank, Murray Z},
  journal={Iie Transactions},
  volume={33},
  number={10},
  pages={847--860},
  year={2001},
  publisher={Springer}
}

@article{yoon2004optimal,
  title={Optimal pricing and admission control in a queueing system with periodically varying parameters},
  author={Yoon, Seunghwan and Lewis, Mark E},
  journal={Queueing Systems},
  volume={47},
  number={3},
  pages={177--199},
  year={2004},
  publisher={Springer}
}

@article{maoui2007congestion,
  title={Congestion-dependent pricing in a stochastic service system},
  author={Maoui, Idriss and Ayhan, Hayriye and Foley, Robert D},
  journal={Advances in Applied Probability},
  volume={39},
  number={4},
  pages={898--921},
  year={2007},
  publisher={Cambridge University Press}
}

@article{ccil2011dynamic,
  title={Dynamic pricing and scheduling in a multi-class single-server queueing system},
  author={{\c{C}}il, Eren Ba{\c{s}}ar and Karaesmen, Fikri and {\"O}rmeci, E Lerzan},
  journal={Queueing Systems},
  volume={67},
  number={4},
  pages={305--331},
  year={2011},
  publisher={Springer}
}

@article{kim2003optimal,
  title={Optimal incentive-compatible pricing for M/G/1 queues},
  author={Kim, Yong J and Mannino, Michael V},
  journal={Operations Research Letters},
  volume={31},
  number={6},
  pages={459--461},
  year={2003},
  publisher={Elsevier}
}

@article{maglaras2006revenue,
  title={Revenue management for a multiclass single-server queue via a fluid model analysis},
  author={Maglaras, Constantinos},
  journal={Operations Research},
  volume={54},
  number={5},
  pages={914--932},
  year={2006},
  publisher={INFORMS}
}

@article{maglaras2003pricing,
  title={Pricing and capacity sizing for systems with shared resources: Approximate solutions and scaling relations},
  author={Maglaras, Constantinos and Zeevi, Assaf},
  journal={Management Science},
  volume={49},
  number={8},
  pages={1018--1038},
  year={2003},
  publisher={INFORMS}
}

@article{lee2014optimal,
  title={Optimal pricing and capacity sizing for the GI/GI/1 queue},
  author={Lee, Chihoon and Ward, Amy R},
  journal={Operations Research Letters},
  volume={42},
  number={8},
  pages={527--531},
  year={2014},
  publisher={Elsevier}
}

@inproceedings{kingman1961single,
  title={The single server queue in heavy traffic},
  author={Kingman, John FC},
  booktitle={Mathematical Proceedings of the Cambridge Philosophical Society},
  volume={57},
  number={4},
  pages={902--904},
  year={1961},
  organization={Cambridge University Press}
}

@article{williams1998diffusion,
  title={Diffusion approximations for open multiclass queueing networks: sufficient conditions involving state space collapse},
  author={Williams, Ruth J},
  journal={Queueing systems},
  volume={30},
  number={1},
  pages={27--88},
  year={1998},
  publisher={Springer}
}

@article{harrison1998heavy,
  title={Heavy traffic analysis of a system with parallel servers: asymptotic optimality of discrete-review policies},
  author={Harrison, J Michael},
  journal={The Annals of Applied Probability},
  volume={8},
  number={3},
  pages={822--848},
  year={1998},
  publisher={Institute of Mathematical Statistics}
}

@article{gamarnik2012multiclass,
  title={Multiclass multiserver queueing system in the Halfin--Whitt heavy traffic regime: Asymptotics of the stationary distribution},
  author={Gamarnik, David and Stolyar, Alexander L},
  journal={Queueing Systems},
  volume={71},
  number={1},
  pages={25--51},
  year={2012},
  publisher={Springer}
}

@article{maguluri2018optimal,
  title={Optimal heavy-traffic queue length scaling in an incompletely saturated switch},
  author={Maguluri, Siva Theja and Burle, Sai Kiran and Srikant, Rayadurgam},
  journal={Queueing Systems},
  volume={88},
  number={3},
  pages={279--309},
  year={2018},
  publisher={Springer}
}

@inproceedings{huang2009delay,
  title={Delay reduction via Lagrange multipliers in stochastic network optimization},
  author={Huang, Longbo and Neely, Michael J},
  booktitle={2009 7th International Symposium on Modeling and Optimization in Mobile, Ad Hoc, and Wireless Networks},
  pages={1--10},
  year={2009},
  organization={IEEE}
}

@article{anjos2025optimal,
  title={Optimal electric vehicle charging with dynamic pricing, customer preferences and power peak reduction},
  author={Anjos, Miguel F and Brotcorne, Luce and Guillot, Ga{\"e}l},
  journal={INFOR: Information Systems and Operational Research},
  pages={1--20},
  year={2025},
  publisher={Taylor \& Francis}
}

@article{zang2024impact,
  title={The impact of information-granularity and prioritization on patients’ care modality choice},
  author={Zang, Lin and Hu, Yue and Roet-Green, Ricky and Sun, Shujing},
  year={2024},
  publisher={Stanford University Graduate School of Business Research Paper}
}

@inproceedings{maguluri2012stochastic,
  title={Stochastic models of load balancing and scheduling in cloud computing clusters},
  author={Maguluri, Siva Theja and Ying, Lei},
  booktitle={2012 Proceedings IEEE Infocom},
  pages={702--710},
  year={2012},
  organization={IEEE}
}

@article{erlang1909theory,
  title={The theory of probabilities and telephone conversations},
  author={Erlang, Agner Krarup},
  journal={Nyt. Tidsskr. Mat. Ser. B},
  volume={20},
  pages={33--39},
  year={1909}
}

@article{halfin1981heavy,
  title={Heavy-traffic limits for queues with many exponential servers},
  author={Halfin, Shlomo and Whitt, Ward},
  journal={Operations research},
  volume={29},
  number={3},
  pages={567--588},
  year={1981},
  publisher={Informs}
}

@article{shi2019process,
  title={Process flexibility for multiperiod production systems},
  author={Shi, Cong and Wei, Yehua and Zhong, Yuan},
  journal={Operations Research},
  volume={67},
  number={5},
  pages={1300--1320},
  year={2019},
  publisher={INFORMS}
}

@article{balseiro2025dynamic,
  title={Dynamic pricing for reusable resources: The power of two prices},
  author={Balseiro, Santiago R and Ma, Will and Zhang, Wenxin},
  journal={Operations Research},
  year={2025},
  publisher={INFORMS}
}

@article{besbes2025dynamic,
  title={Dynamic resource allocation: Algorithmic design principles and spectrum of achievable performances},
  author={Besbes, Omar and Kanoria, Yash and Kumar, Akshit},
  journal={Operations Research},
  volume={73},
  number={3},
  pages={1273--1288},
  year={2025},
  publisher={INFORMS}
}

@book{puterman1994markov,
  title={Markov Decision Processes: Discrete Stochastic Dynamic Programming},
  author={Puterman, Martin L.},
  year={1994},
  publisher={John Wiley \& Sons}
}

@article{hordijk1987convergence,
  title={On the Convergence of Policy Iteration in Finite State Undiscounted Markov Decision Processes: The Unichain Case},
  author={Hordijk, Arie and Puterman, Martin L.},
  journal={Mathematics of Operations Research},
  volume={12},
  number={1},
  pages={163--176},
  year={1987}
}

@article{gilardoni2010pinsker,
  title={On Pinsker's and Vajda's Type Inequalities for Csisz{\'a}r's $ f $-Divergences},
  author={Gilardoni, Gustavo L},
  journal={IEEE Transactions on Information Theory},
  volume={56},
  number={11},
  pages={5377--5386},
  year={2010},
  publisher={IEEE}
}
